\documentclass[11pt, reqno]{amsart}
\usepackage{amssymb,latexsym,amsmath,amsfonts,amsthm}

\numberwithin{equation}{section}

\usepackage[mathscr]{eucal}

\usepackage[T1]{fontenc}
\usepackage{txfonts}
\usepackage[pdftex,bookmarks=true]{hyperref}
\usepackage{graphicx}
\usepackage{enumerate}

\usepackage[all]{xy}
\usepackage{graphicx}
\usepackage{enumerate}

\theoremstyle{plain}
\newtheorem{theorem}{Theorem}[section]

\newtheorem{lemma}[theorem]{Lemma}
\newtheorem{prop}[theorem]{Proposition}

\theoremstyle{remark}
\newtheorem{rmk}[theorem]{Remark}

\newtheorem{exa}[theorem]{Example}

\theoremstyle{definition}
\newtheorem{dfn}[theorem]{Definition}

\newcommand{\real}{\mathbb{R}}

\begin{document}
\title[Distance measure spaces and Riemann surface laminations]{Compactness theorems for the spaces of distance measure spaces and Riemann surface laminations}
\author{Divakaran Divakaran and Siddhartha Gadgil}

\begin{abstract}
In this paper, we give a generalisation of Gromov's compactness theorem for metric spaces, more precisely, we give a compactness theorem for the space of distance measure spaces equipped with a \emph{generalised Gromov-Hausdorff-Levi-Prokhorov distance}.  Using this result we prove that the Deligne-Mumford compactification is the completion of the moduli space of Riemann surfaces under the generalised Gromov-Hausdorff-Levi-Prokhorov distance.  Further we prove a compactness theorem for the space of Riemann surface laminations.
\end{abstract}

\maketitle

\section{Introduction}

Since the landmark work of Gromov, a very fruitful technique in Riemannian geometry is to consider limits of Riemannian manifolds that are of a more general nature, such as metric spaces. The limits are taken in the sense of the \emph{Gromov-Hausdorff} distance, which can be defined for compact metric spaces or for metric spaces with given base points. As Riemannian manifolds are equipped with canonical measures, namely volumes, it has also turned out to be fruitful to consider limits of these as metric measure spaces. 

However, one case of fundamental importance in mathematics, the Deligne-Mumford compactification of the Moduli space of Riemann surfaces, or equivalently hyperbolic surfaces, cannot be naturally viewed (at least directly) in this setting. This is because the limit of hyperbolic surfaces is not compact, and indeed has in general several non-compact components. However, a limit using basepoints will only have one limiting non-compact component. This work is motivated by the desire to generalize convergence of metric measure spaces to include this case in a natural way, and furthermore allow us to study limits of surface laminations. 

Specifically, we study \emph{distance measure spaces}, where a distance is like a metric except that distances between the points may be infinite. A significant part of this work is devoted to defining an appropriate distance in this setting, and establishing all the desired properties of this distance.

One of the attractive features of our approach is that the Margulis lemma, which has a pivotal role in studying limits of hyperbolic metrics, has an elementary reformulation that is meaningful in contexts other that surfaces. Namely, for hyperbolic surfaces of a fixed genus $g$, given $\epsilon> 0$ there is a real number $M>0$ so that each hyperbolic surface $F$ of genus $g$ has a finite set of at most $M$ points whose $\epsilon$-neighbourhood has complement in $F$ with measure less than $\epsilon$. In other words, there are $\epsilon$-almost $\epsilon$-dense finite sets with uniformly bounded cardinality. As a consequence, the Deligne-Mumford compactification turns out to have a very natural description, namely it is the \emph{completion} of the space of hyperbolic metrics (with respect to our metric).

We also study the limits of hyperbolic surface laminations in the same sense. Indeed, the above allows us to also state a natural analogue of Margulis lemma in the case of hyperbolic surface laminations. However it turns out that this analogous statement is false. Hence our result needs as hypothesis something equivalent to the Margulis lemma.

In the rest of the introduction, we elaborate on the above remarks.

\subsection{Gromov's compactness theorem for metric spaces}
Gromov's notion of the Gromov-Hausdorff distance, an intrinsic version of the Hausdorff distance, was a major turning
point in the history of global Riemannian geometry.  The Gromov-Hausdorff distance is a distance on the space of 
compact metric spaces.  Introduced in connection with the study of almost flat manifolds, it was used by Gromov to 
prove that a finitely generated group has polynomial growth if and only if it has a nilpotent subgroup that is of 
finite index.  Later, it became an important tool in geometric group theory.   We recall the definition.

\begin{dfn}[$\varepsilon$--neighbourhood]
Given a metric space $(X,d)$, a set $S\subset X$ and $\varepsilon >0$, the
$\varepsilon$-neighbourhood of $S$ is given by
$$S^{\varepsilon}= \{x \in X | \ \exists s\in S, d(x,s)< \varepsilon \rbrace=
\cup_{s\in
S} \{x\in X| \ d(x,s)<\varepsilon\rbrace$$ 
\end{dfn}

\begin{dfn}[Hausdorff distance]
Let $X$ and $Y$ be two non-empty subsets of a metric space $(M, d)$. We define
the Hausdorff distance $d_H(X, Y)$ as
$$d_H(X,Y) = \inf\{\varepsilon > 0 | \ X \subseteq Y^\varepsilon \ \mbox{and}\ Y
\subseteq X^\varepsilon\rbrace$$
\end{dfn}

\subsubsection*{Some properties:}

\begin{itemize}
\item In general, $d_H(X,Y)$ may be infinite. For example, take $(M,d) = (\mathbb{R}, d(x,y) = \vert x-y \vert)$ and 
$X = [0,1]$ and $Y = \mathbb{R}$.  If both $X$ and $Y$ are bounded, then $d_H(X,Y)$ is guaranteed to be finite.  
\item We have $d_H(X,Y) = 0$ if and only if $X$ and $Y$ have the same closure.  
\item On the set of all non-empty compact subsets of $M$, $d_H$ is a metric.
\end{itemize}

\begin{dfn}[Gromov-Hausdorff distance]
Let $(X_i,d_i), i=1,2$, be compact metric spaces.  Consider pairs of isometric
embeddings $\iota_i :X_i \to Z$, $i=1,2$ of $X_i$ into a metric space $Z$.  For
such embeddings we can consider the Hausdorff distance $d_H$ between
$\iota_i(X_i)$ as subsets of $Z$.  The Gromov-Hausdorff distance is the infimum
over all such Hausdorff distances, i.e.,
$$d_{GH}(X_1,X_2) = \inf\{d_H\left(\iota_1(X_1),\iota_2(X_2)\right)| \ \iota_i:X_i\to Z\
\mathrm{ 
isometric\  embeddings}\rbrace.$$ 
\end{dfn}

Gromov proved a compactness theorem for the space of metric spaces equipped with this metric.

\begin{theorem}[Gromov compactness theorem]
A collection of compact metric spaces is pre-compact under the Gromov-Hausdorff
distance if and only if the following conditions are satisfied.
\begin{itemize}
\item There exists $C > 0$, such that every element in the collection has
diameter less than or equal to $C$.
\item Given $\varepsilon > 0$, there exists a natural number $N(\varepsilon)$, such that every element
of the collection admits an $\varepsilon$-net containing no more than
$N(\varepsilon)$ points.
\end{itemize}
\end{theorem}

This theorem has various applications including study of manifolds with Ricci curvature bounded below, Einstein 
metrics etc.  

\subsection{Compactification of the moduli space of Riemann surface and its
applications} 

The moduli space of Riemann surfaces is the space of all Riemann surfaces of a
fixed topological type up to biholomorphism. This spaces is not compact, as
sequences of such surfaces can degenerate. However there is a very elegant
compactification of moduli space due to Bers, Mumford and Deligne.

The Deligne-Mumford compactification of the moduli space of Riemann surfaces
and the closely related Gromov compactness theorem for $J$-holomorphic curves
in symplectic manifolds (in particular curves in an algebraic varieties) are
important results for many areas of mathematics. A major reason for this is
that the additional strata added in the compactification have co-dimension at
least $2$. This means that the moduli spaces have fundamental classes, and can
be treated as compact manifolds from the point of view of intersection theory.
This allows the construction of so called Gromov invariants, counting the
number of such curves.

Furthermore, the explicit nature of the compactification allows us to
understand, and sometimes rule out, degenerate curves. Such an approach allowed
Gromov to prove many important results in symplectic geometry.

As the strata of the compactification are well understood, one can use their
rich underlying structure to define and show algebraic properties of
extensions of intersection numbers, giving Gromov-Witten theory and quantum
cohomology. Such structures were used by Kontsevich to answer classical
questions in enumerative geometry.

\subsection{Motivation}
A sequence of compact hyperbolic surfaces may converge to a 
non-compact hyperbolic surface.  Moreover, the 
limit space in the Deligne-Mumford compactification is not a metric space; the distance between two points can be 
$\infty$.  Such spaces are called distance spaces.

\begin{center}
\includegraphics[scale=.5]{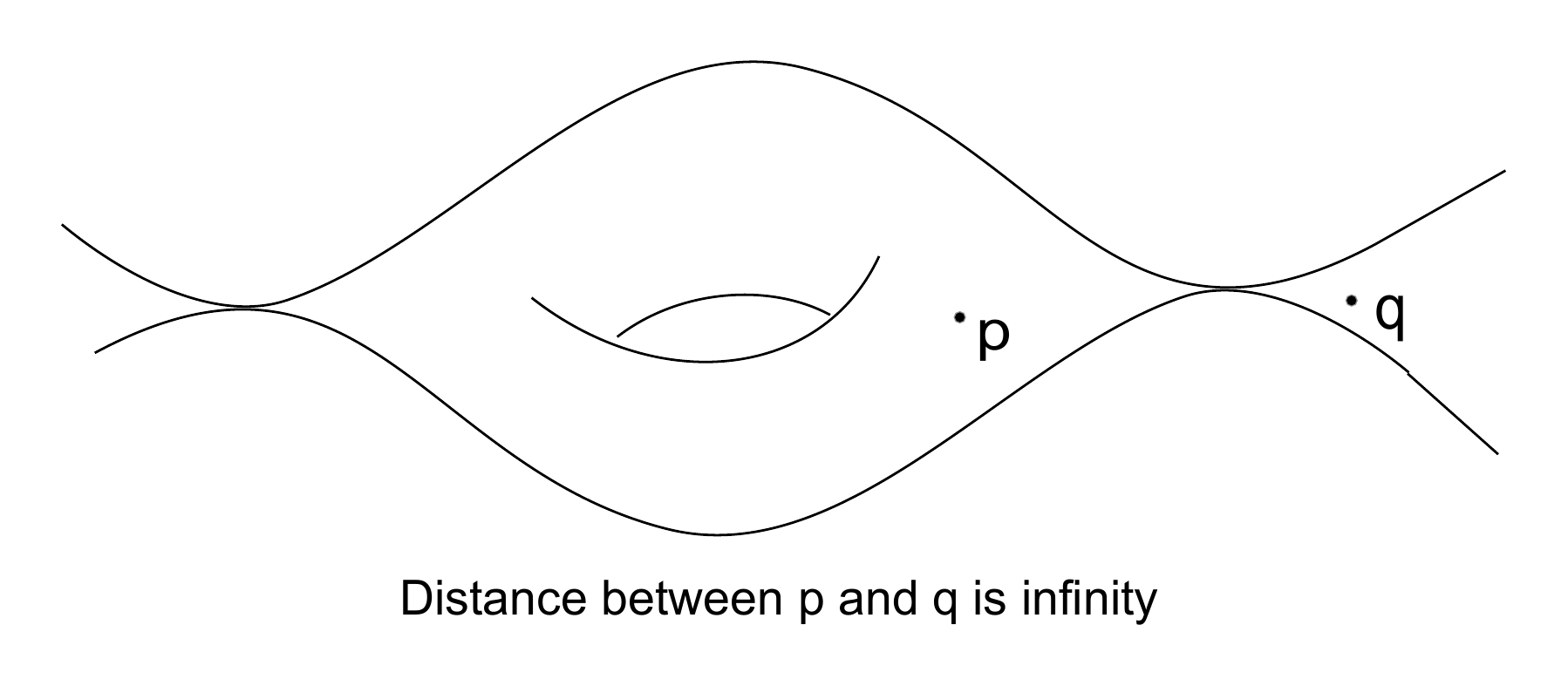}
\end{center}

While studying non-compact spaces, one usually uses pointed Gromov-Hausdorff convergence rather than
Gromov-Hausdorff convergence.  We recall the definition.

\begin{dfn}[Pointed Gromov-Hausdorff convergence]
A sequence of pointed spaces $(X_n,x_n)$ is said to converge to a pointed space $(X,x)$ if, for every $r>0$, the 
closed $r$ ball $\overline{B(x_n,r)}$ converges to the closed $r$ ball $\overline{B(x,r)}$ in the Gromov-Hausdorff 
sense.
\end{dfn} 

This notion is not very useful to study distance spaces as, by pointed
Gromov-Hausdorff convergence the maximum we can hope to retrieve is the information of the component in which the 
point $x$ lies, that is, all points which are at a finite distance from $x$.  For the same reason, the use of pointed
Gromov-Hausdorff convergence in the study of moduli space of Riemann surfaces is limited.

\begin{center}
\includegraphics[scale=.5]{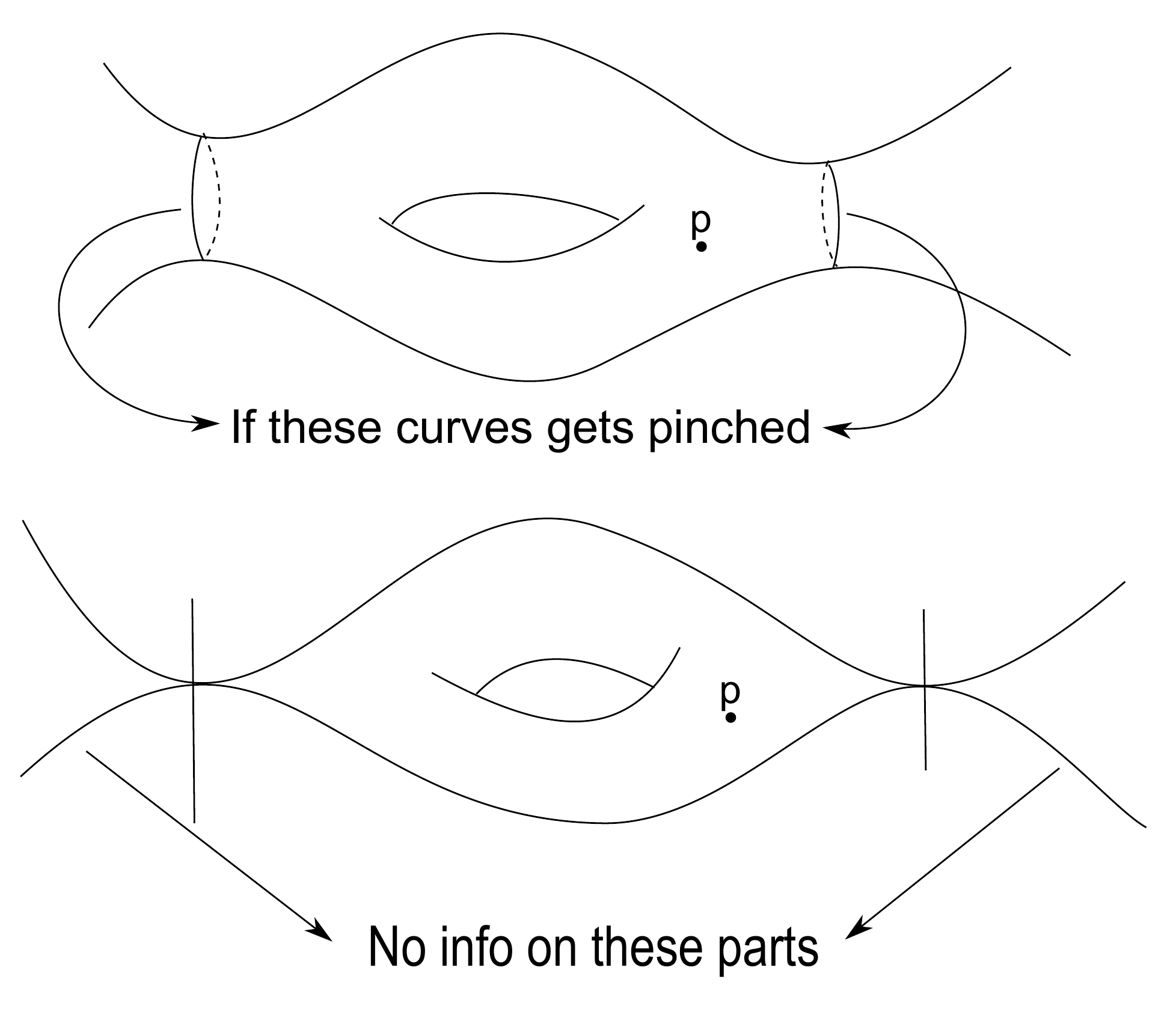}
\end{center}

One way to overcome this is to take several base points but, there is no natural way to do so.  It is also very messy.
We have instead used the approach of choosing a
finite measure and using a generalised Gromov-Hausdorff-Levy-Prokhorov(GHLP) distance.  This can be
viewed as a generalisation of the pointed Gromov-Hausdorff convergence.  The fact that many metric spaces, especially
Riemannian manifolds, come equipped naturally with a measure is further motivation for the use of measure instead of 
multiple points.  Moreover, it turns out that the Deligne-Mumford compactification has a very natural description using this metric, namely it is the \emph{completion} of the space of hyperbolic surfaces.
 
While Gromov's compactness theorem for J-holomorphic curves in symplectic manifolds, is an important tool in 
symplectic topology, its applicability is limited by the lack of general methods to construct
pseudo-holomorphic curves. One hopes that considering a more general class of objects in place of pseudo-holomorphic 
curves will be useful.  Generalising the domain of pseudo-holomorphic curves from Riemann surfaces to Riemann surface
laminations is a natural choice.

Riemann surface laminations are generalisations of classical Riemann surfaces.  More precisely, a Riemann surface 
lamination is a locally compact, separable, metrisable space $M$ with an open cover by flow boxes $\{U_i\rbrace_{i\in I}$ 
and homeomorphisms $\phi_i:U_i \to D_i \times T_i$
with $D_i$ an open set in $\mathbb{C}$ such that the coordinate changes in $U_i \cap U_j$ are of the form 
$$\phi_j \circ \phi_i^{-1}(z,t) = (\lambda_{ji}(z,t),\tau_{ji}(t))$$
where the map $z\to \lambda_{ji}(z,t)$ is holomorphic for each $t$.  Riemann surface laminations, coming from the 
qualitative study of ordinary differential equations in complex domain, have become central to the theory of 
holomorphic dynamical systems.  They occur naturally in many geometric and dynamical situations. 

Theorems such as the uniformisation theorem for surface laminations due to Alberto Candel (which is a partial 
generalisation of the uniformisation theorem for surfaces), generalisations of the Gauss-Bonnet theorem proved for 
some special cases, and topological classification of ``almost all" leaves using harmonic measures tell us that 
Riemann surface laminations exhibit many properties similar to Riemann surfaces.  Further, the success of essential 
laminations, as generalised incompressible surfaces, in the study of 3-manifolds suggests that a similar approach may
be useful in symplectic topology.

Compactness of the moduli space of Riemann surfaces was central in the use of pseudo-holomorphic curves to study 
symplectic 4-manifolds.  So, a compactness theorem for the space of Riemann surface laminations, analogous to the 
Deligne-Mumford compactification would be very useful.  In this paper we prove such a result.  However, as it turns out that there is no result analogous to the Margulis lemma, our result needs as hypothesis something equivalent to the Margulis lemma.  

\section{Metric geometry}
\subsection{Distance spaces}

A distance space is a generalisation of a metric space where, we allow the
distance between two points to be infinity. More precisely, 

\begin{dfn}\label{distance}
A tuple $(X,d)$ is called a distance space if  $X$ is a set and $d:X\times X \to
\mathbb{R}^+\cup \{\infty\rbrace$ ($\mathbb{R}^+$ denotes the set of all non-negative real numbers) is such that
\begin{itemize}
\item $d(x,y)=0 \Leftrightarrow x=y$.
\item $d(x,y) = d(y,x)$.
\item For all $z$ in $X$, $d(x,y)\leq d(x,z)+d(z,y)$ . 
\end{itemize}
\end{dfn}

Associated with a distance space there is the natural topology, generated by
balls of radius $r$, $B(x_0,r) = \{x\in X \ | \ d(x,x_0)< r\rbrace$, as in the case of metric space. We begin with a few examples.

\begin{exa}
A very simple but very useful example of a distance space is a finite set $X$
with a distance, i.e., a function to $[0, \infty]$ satisfying the conditions of
Definition~\ref{distance}.    
\end{exa}

\begin{exa}
Let $X := \{x \in \real: x \ge 0\rbrace$ and define the distance $d$ as 
\begin{itemize}
\item $d(0,0) = 0$.
\item If $x>0$, $d(x,0) = \infty$.
\item If $x,y \neq 0$ then, $d(x,y) = \vert \log(x) -\log(y) \vert$.
\end{itemize}
This example is related to hyperbolic geometry.  As
distance element, we have used $\frac{1}{x}dx$ instead of $dx$.    
\end{exa}
 
We will show that a distance space is a disjoint union of metric spaces in
a canonical manner.  Thus, many properties of distance spaces follow from the
analogous properties of metric spaces, by applying them to each component metric space.
Fix a distance space $X = (X, d)$.

\begin{dfn}
Let $\sim$ be the relation on $X$ given by $x \sim y$ if and only if $d(x,y)$ is
finite.
\end{dfn}

\begin{theorem}
The relation $\sim$ is an equivalence relation.
\end{theorem}

\begin{proof}
The fact that $\sim$ is reflexive and symmetric is obvious.  Transitivity
follows from the triangle inequality.  More precisely, $x \sim y$ and $y \sim z$ means that $d(x,y)$ and $d(y,z)$ are
finite.  As $d(x,z) \leq d(x,y) + d(y,z)$, $d(x,z)$ is finite.  Hence $x \sim z$.
\end{proof}

Clearly, each equivalence class, under the equivalence relation $\sim$, is a
metric space.  Thus, we have represented the distance space $X$
as a disjoint union of metric spaces.  We will use this fact to prove
some results, some of which will be used later. 

\begin{theorem}
Each equivalence class of $X$ under $\sim$ is open. 
\end{theorem}

\begin{proof}
Let $Y\subset X$ be an equivalence class.  Given $x_0\in Y$, $B(x_0,1)\subset Y$ as, if $x\in B(x_0, 1)$, then 
$d(x,x_0)<1<\infty$, so $x\sim x_0$, i.e., $x\in Y$.  As $x_0\in Y$ was arbitrary, $Y$ is open.  
\end{proof}

\begin{theorem}
Each equivalence class of $X$ under $\sim$ is closed.
\label{simclosed}
\end{theorem}

\begin{proof}
Note that the complement of an equivalence class $Y$ is a union of equivalence
classes, each of which is open. Hence $X\setminus Y$ is open, i.e., $Y$ is
closed.
\end{proof}

\begin{theorem}
A distance space (or a metric space) is separable if and only if it is second
countable.
\label{sep=sc}
\end{theorem}

\begin{proof}
\emph{Separable implies second countable:}  As $X$ is separable, there is a countable dense set, say $S$.  Thus,
given any point $x\in X$, there exists a sequence $\langle s_k \rangle_k \subset S$, such that $s_k$ converges to $x$.
Let $U$ be an open set containing $x$.  Note that $d(x,U^c) > 0$.  Choose $n_0\in \mathbb{N}$ such that 
$\frac{2}{n_0} < d(x,U^c)$.  The open sets $B\left(s_k,\frac{1}{n_0}\right), \forall k$  will eventually contain $x$ as $s_k$ 
converges to $x$.  Let $k_0$ be large enough that, $x\in B\left(s_{k_0},\frac{1}{n_0}\right) $.  If $y \in B\left(s_{k_0}, \frac{1}{n_0}\right)$
then, 
$$d(x,y) \leq d(x,s_k) + d(y,s_k) \leq \frac{2}{n_0} < d(x,U^c).$$  
So, $B\left(s_{k_0}, \frac{1}{n_0}\right) \subset U$.  Thus, the
countable collection of open sets $\left\lbrace B_{\frac{1}{n}}(s)\right\rbrace_{s\in S, n\in
\mathbb{N}}$ forms a basis.

\emph{Second countable implies separable:}  Let $\{U_i\rbrace_{i\in \mathbb{N}}$ be a
basis.
Consider the countable set $\{x_i\rbrace_{i\in \mathbb{N}}$ where $x_i$ is an
arbitrary point in $U_i$.  We claim that $\{x_i\rbrace_{i\in \mathbb{N}}$ is a
dense subset. 
To prove this, consider an arbitrary point $x$.  As $\{U_i\rbrace_{i\in \mathbb{N}}$
is
a basis, for every $n\in \mathbb{N}$ we can find an element $U_{i_n}$ such that
$U_{i_n}\subset B_{\frac{1}{n}}(x)$.  Thus, $d\left(x_{i_n},x\right) \leq \frac{1}{n}$, i.e., $x_{i_n}$ converges to $x$.  
\end{proof}\

\begin{lemma}
If a distance space (or a metric space) $X$ is second countable, then $Y\subset X$ is
second countable under the subspace topology.
\label{subsetsc}
\end{lemma}

\begin{proof}
Let $\{U_i\rbrace_{i\in
\mathbb{N}}$  be a countable basis for $X$, i.e., given any point $x\in X$ and an open set $U$ containing $x$, we can
find a $U_{i(x,U)}$ such that $x \in U_{i(x,U)} \subset U$.  Given a point $x \in Y$ and a set $V$ open in $Y$ with 
$x\in V \subset Y$, choose a $U$ open in $X$ such that $V = U \cap Y$.  The existence of such a set $U$ follows from 
the definition of subspace topology.  Now, there exists a $U_{i(x,U)}$ such that $x \in U_{i(x,U)} \subset U$.  Thus,
$x \in U_{i(x,U)} \cap Y \subset U \cap Y = V$.  Hence, $\{U_i\cap Y\rbrace_{i\in \mathbb{N}}$ is a
basis for $Y$.    Thus, $Y$ is second countable
\end{proof}

\begin{lemma}
If a distance space (or a metric space) $X$ is separable, then $Y\subset X$ is
separable under the subspace topology.
\label{subsetseparable}  
\end{lemma}

\begin{proof}
As $X$ is separable it is second countable, by Theorem \ref{sep=sc}.  Thus $Y$ is second countable, by Theorem \ref{subsetsc} and the fact that $X$ is 
second countable.  Hence, $Y$ is separable by Theorem \ref{sep=sc}.
\end{proof} 

\begin{theorem}
A distance space is separable if and only if the following conditions are
satisfied: 
\begin{itemize}
\item There are only a countable number of equivalence classes under $\sim$.
\item Each equivalence class under $\sim$ is separable.
\end{itemize} 
\end{theorem} 

\begin{proof}
To prove the if part we use Theorem \ref{sep=sc}.  As each equivalence class is
separable, it is second countable.  So, we get a countable basis for
each equivalence class.  There are only a countable number of equivalence
classes and the countable union of countable sets is countable.  So, we get a
countable basis for the distance space.  Thus the distance space is second
countable.  Use Theorem \ref{sep=sc} again to get the result.  

As $X$ is separable it is second countable, by Theorem \ref{sep=sc}. The equivalence classes are open and 
disjoint.  So for each equivalence class there exists at least one basis element. As $X$ is second
countable, there are at most
countably many equivalence classes.  Each equivalence class is separable by Lemma \ref{subsetseparable}.     
\end{proof}

\begin{theorem}
A subset, $S$, of a distance space, $X$, is compact if and only if the following
conditions are satisfied:
\begin{itemize}
\item $S$ intersects only finitely many equivalence classes under $\sim$.
\item The intersection of $S$ with each equivalence class is compact.
\end{itemize}
\end{theorem}

\begin{proof}
The if part follows from the fact that the finite union of compacts sets is
compact.  On the other hand, suppose $S$ is compact.  It cannot intersect
infinitely many equivalence classes as, taking an arbitrary point from each of
them gives a sequence which will have no convergent subsequence. Further, if 
$Y$
is an equivalence class under $\sim$,  then $Y\cap S$ is compact by
Theorem~\ref{simclosed} and the fact that a closed subset of a compact set is
compact. 
\end{proof}

\subsection{Metric measure spaces and Distance measure spaces}

\begin{dfn}
A metric measure space is a triple $(X,d,\mu)$, where $(X,d)$ forms a metric
space and $\mu$ is a finite regular Borel measure.  
\end{dfn}

\begin{dfn}
A distance measure space is a triple $(X,d,\mu)$, where $(X,d)$ forms a distance
space and $\mu$ is a finite regular Borel measure.  
\end{dfn}

\noindent An important class of 
examples is $(X,d,\mu)$ where, $X$ is a finite set.  We will in fact show that
any distance measure space, with finite measure, can be approximated with such a
space.

\subsection{Gromov-Hausdorff-Levy-Prokhorov distance}
In this section we will recall the definition of the
Gromov-Hausdorff-Levy-Prokhorov distance.  For this part we closely follow the
approach in~\cite{gadgil}, but we use a different-but-equivalent definition of the
Levy-Prokhorov metric.  Here we assume the measure $\mu$ to be
finite but, not necessarily a probability measure.  The two definitions  coincide
when we use probability measures.  A good general reference for the subject is \cite{gromov}.  For some recent 
developments see \cite{MetMeasSpref}.  We will later generalise the Gromov-Hausdorff-Levy-Prokhorov distance to 
distance measure
spaces. 

\begin{dfn}[Borel Measure]
Let $X$ be a locally compact Hausdorff space, and let Borel $\sigma$-algebra, $\mathcal{B}(X)$, be the smallest 
$\sigma$-algebra that contains the open sets of $X$.  Any measure $\mu$ defined on the Borel $\sigma$-algebra is 
called a Borel measure.   
\end{dfn}

\begin{dfn}[Levy-Prokhorov metric]
Let $\mathcal{B}(M)$ represent the Borel $\sigma$-algebra on $M$.  The
Levy-Prokhorov  metric $d_{\pi}$ between two finite Borel measures $\mu,\nu$ is defined as
$$d_{\pi} (\mu, \nu) := \inf \left\{ \varepsilon > 0 ~|~ \mu(A) \leq \nu
\left(A^{\varepsilon}\right) + \varepsilon \ \text{and} \ \nu (A) \leq \mu
(A^{\varepsilon}) + \varepsilon \ \text{for all} \ A \in \mathcal{B}(M)
\right\rbrace.  
$$ 
\end{dfn}

\begin{dfn}
If $c\in \mathbb{R}$, we define scaling of a measure $c\mu$ as $(c\mu)(E):=
c.\mu(E)$.
\end{dfn}

\begin{lemma}
If $c\leq 1$ then $d_{\pi}(c\mu,c\nu)\leq d(\mu,\nu)$.
\label{mscal<1}
\end{lemma}

\begin{proof}
For all $\varepsilon > 0$, such that $\mu(E)\leq \nu (E^{\varepsilon})+\varepsilon$, we have
$$c\mu(E)\leq c\nu (E^{\varepsilon})+c\varepsilon \leq c\nu (E^{\varepsilon})+\varepsilon. $$
Similarly, for all $\varepsilon > 0$, such that $\nu(E)\leq \mu (E^{\varepsilon})+\varepsilon$ we have
$c\nu(E)\leq c\mu (E^{\varepsilon})+\varepsilon$.  The result follows.  
\end{proof}

\begin{lemma}
If $c\geq 1$ then $d_{\pi}(c\mu,c\nu)\leq c.d_{\pi}(\mu,\nu)$
\label{mscal>1}
\end{lemma}

\begin{proof}
For all $\varepsilon > 0$, such that $\mu(E)\leq \nu (E^{\varepsilon})+\varepsilon$ we have
$$c\mu(E)\leq c\nu (E^{\varepsilon})+c\varepsilon \leq c\nu
(E^{c\varepsilon})+c\varepsilon. $$
Similarly, for all $\varepsilon > 0$, such that $\nu(E)\leq \mu (E^{\varepsilon})+\varepsilon$ we have
$c\nu(E)\leq c\mu (E^{c\varepsilon})+c\varepsilon$.  The result follows.
\end{proof}

\begin{dfn}
Given measurable spaces $(X_1,\Sigma_1)$ and $(X_2,\Sigma_2)$, a measurable mapping $f: X_1 \to X_2$ and a measure 
$\mu: \Sigma_1 \to [0, \infty]$, the push-forward measure of $\mu$ is defined to be the measure 
$f_*(\mu): \Sigma_2 \to [0,\infty]$ given by
$$(f_*(\mu))(B) = \mu(f^{-1}(B)) \ \text{for } B \in \Sigma_2$$
\end{dfn}

\begin{lemma}
Let $(X,d)$ and $(X',d')$ be two distance spaces and let $f: (X,d)
\hookrightarrow (X',d') $ be an isometric embedding.  If $\mu_1,\mu_2$ are two
measures on $(X,d)$ then,
$$d_{\pi}^{(X,d)}(\mu_1,\mu_2) = d_{\pi}^{(X',d')}(f_*(\mu_1),f_*(\mu_2))$$ 
\label{piinclusionnicebehaviour} 
\end{lemma}

\begin{proof}
Let $B$ be an arbitrary subset of $X'$.  Note that $f_*(\mu_i)(B) =
f_*(\mu_i)(B\cap f(X))$.  Thus, while calculating $d_{\pi}$ we have to consider only
subsets of $f(X)$.  So without loss of generality assume $B\subset f(X)$. Let $A = f^{-1}(B)$.  By definition, 
$\mu_i(A) = (f_*(\mu_i))(B)$.   As $f$ is injective, for all $S \subset X$, $f^{-1}(f(S)) = S$.  Thus, 
$(f_*(\mu_i))(f(A^{\varepsilon})) = \mu_i(f^{-1}(f(A^{\varepsilon}))) = \mu_i(A^{\varepsilon})$.  But,
$f(A^{\varepsilon}) = B^{\varepsilon}\cap f(X)$ as, $f$ is an isomorphism.  So, 
$\mu_i(A^{\varepsilon}) = (f_*(\mu_i))(f(A^{\varepsilon})) = f_*(\mu_i)\left(B^{\varepsilon}\right)$.  

For all 
$\varepsilon > d_{\pi}^{(X,d)}(\mu_1,\mu_2)$, we have $\mu_1(A) \leq \mu_2
(A^{\varepsilon}) + \varepsilon \ \text{and} \ \mu_2 (A) \leq \mu_1
(A^{\varepsilon}) + \varepsilon$.  Thus, $f_*(\mu_1)(B) \leq f_*(\mu_2)
(B^{\varepsilon}) + \varepsilon \ \text{and} \ f_*(\mu_2) (B) \leq f_*(\mu_1)
(A^{\varepsilon}) + \varepsilon$.   Hence $d_{\pi}^{(X,d)}(\mu_1,\mu_2)$ is less than or equal to 
$d_{\pi}^{(X',d')}(f_*(\mu_1),f_*(\mu_2))$.  Similarly, for all $\varepsilon > d_{\pi}^{(X',d')}(f_*(\mu_1),f_*(\mu_2))$
we have, $f_*(\mu_1)(B) \leq f_*(\mu_2)
(B^{\varepsilon}) + \varepsilon \ \text{and} \ f_*(\mu_2) (B) \leq f_*(\mu_1)
(A^{\varepsilon}) + \varepsilon$. Thus, we have $\mu_1(A) \leq \mu_2
(A^{\varepsilon}) + \varepsilon \ \text{and} \ \mu_2 (A) \leq \mu_1
(A^{\varepsilon}) + \varepsilon$.  Hence $d_{\pi}^{(X,d)}(\mu_1,\mu_2)$ is greater than or equal to 
$d_{\pi}^{(X',d')}(f_*(\mu_1),f_*(\mu_2))$.  Combining the two we have the result.
\end{proof}

\begin{dfn}[Gromov-Hausdorff-Levy-Prokhorov Distance]
We now define the Gromov-Hausdorff-Levy-Prokhorov distance (GHLP distance,
$d_{GHLP}$) between a pair of metric measure spaces $(X_i,d_i,\mu_i), i
= 1, 2$, with the underlying metric space $X_i$ assumed to be compact.  Consider
isometric embeddings $\iota_i : X_i \to Z, i = 1, 2$ of the spaces $X_i$ into a
metric space $Z$. These give rise to push forward probability measures
$(\iota_i)_*(\mu_i)$ on $Z$, given by $(\iota_i)_*(\mu_i)(E)=
\mu_i(\iota_i^{-1}(E))$. The distance between the metric measure spaces is the
infimum of the distance between the push forward measures over all
isometric embeddings, i.e.,
   $$d_{GHLP} (X_1 , X_2 ) = inf\{d_{\pi}((\iota_1)_*(\mu_1), (\iota_2)_*(\mu_2)) \ | \
\iota_i : X_i \to Z \ \text{isometric embeddings} \rbrace$$
where $Z$ in the infimum varies over all metric spaces.  
\end{dfn}

We can identify the
space $X_i$ with its image in $Z$.  Further we can assume that, under these
assumptions $Z=X_1\cup X_2$.  We shall often make such identifications and
identify the measures $\mu_i$ with the push forward measures on $Z$.

\begin{theorem}
GHLP distance satisfies  the triangle inequality.
\end{theorem}

\begin{proof}
The proof is the same as that of the corresponding result in~\cite{gadgil}.
\end{proof}

As the GHLP distance ignores sets of measure zero, we cannot expect an isometry between the two spaces $(X_1,d_1,\mu_1)$ and $(X_2,d_2,\mu_2)$ given that the distance between two spaces $d_{GHLP}((X_1,d_1,\mu_1),(X_2,d_2,\mu_2)) = 0$, but only that this is true up to ignoring an
appropriate class of sets with measure zero. 

\begin{lemma}
If $d_{GHLP}((X_1,d_1,\mu_1),(X_2,d_2,\mu_2))=0$, then $\mu_1(X_1)=\mu_2(X_2)$.
\label{equameas} 
\end{lemma}

\begin{proof}
For every $\varepsilon > 0$, there exists a metric space
$(Z(\varepsilon),d(\varepsilon))$ and isometries
$$\iota_i:(X_i,d_i)\to
(Z(\varepsilon),d(\varepsilon))$$
such that the push forward measures $\nu_i=
(\iota_i)_*(\mu_i)$ are at a distance $\varepsilon$ from each other.  But,
$\nu_i(Z(\varepsilon))=\mu_i(X_i)$.  Thus, 
$$ |\mu_1(X_1)-\mu_2(X_2)|=|\nu_1(Z(\varepsilon))-
\nu_2(Z(\varepsilon))|\leq\varepsilon.$$
As $\varepsilon$ was arbitrary, we get $\mu_1(X_1)=\mu_2(X_2)$.       
\end{proof}
 
\begin{theorem}
For two metric measure spaces $(X_1,d_1,\mu_1)$ and $(X_2,d_2,\mu_2)$, 
$$d_{GHLP}((X_1,d_1,\mu_1),(X_2,d_2,\mu_2))=0$$
if and only if there are open sets
$U_1$ and $U_2$ of measure zero such that there is a measure preserving isometry
between $(X_1 \setminus U_1)$ and $(X_2 \setminus U_2)$.
\end{theorem}

\begin{proof}
If there is a measure preserving isometry $\phi$ from $(X_1 \setminus U_1)$ to $(X_2 \setminus U_2)$,
then let $Z$ be equal to $X_2$.  Then we get that
$$d_{GHLP}((X_1,d_1,\mu_1),(X_1,d_2,\mu_2))\leq d_{\pi}(\phi_*(\mu_1),\mu_2)=0.$$
On the other hand assume $d_{GHLP}((X_1,d_1,\mu_1),(X_2,d_2,\mu_2))=0$.  We
have, by Lemma \ref{equameas} and Lemma \ref{mscal<1}, that
$$d_{GHLP}\left(\left(X_1,d_1,\frac{1}{\mu_1(X_1)}\mu_1\right),\left(X_2,d_2,\frac{1}{\mu_2(X_2)}
\mu_2\right)\right)=0.$$  Then by the corresponding result for probability measures (see
\cite{gadgil},for example) we get the required result.   
\end{proof}

\section{Distance measure spaces}

\subsection{Distance on the space of distance measure spaces}\label{SEC:defnrho}       

We will generalise the idea of the GHLP distance to get a distance on the space of distance measure spaces.  We first
consider some illustrative examples.

\begin{exa}
Let $X = \{x,y\rbrace$, $d_1 = d_2 = d$ where $d$ is defined as follows:

\begin{tabular}{l r}
\begin{minipage}{0.5\textwidth}
\begin{itemize}
\item $d(x,x) = 0 = d(y,y)$.
\item $d(x,y) = d$ for some fixed $d\in \real^+$.
\end{itemize}
\end{minipage}
&
\begin{minipage}{0.5\textwidth}
\includegraphics[scale=.4]{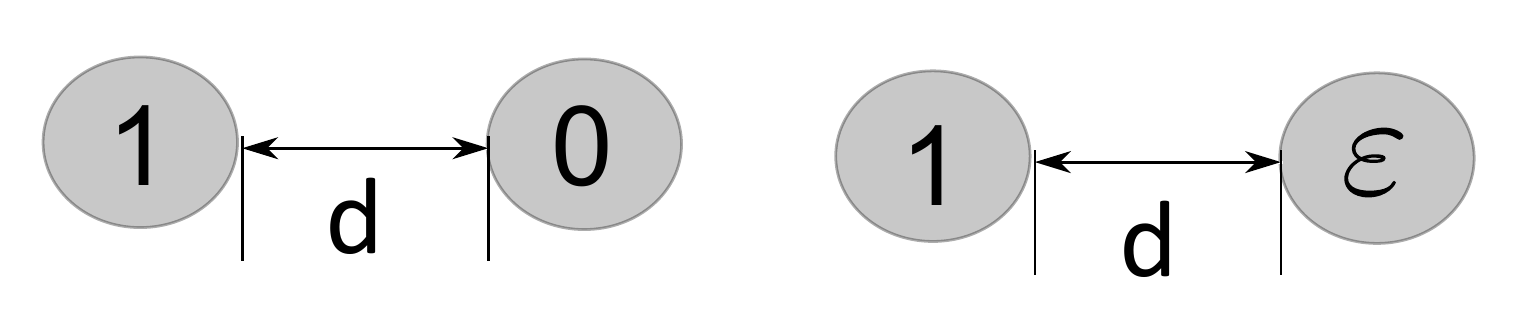}
\end{minipage}
\end{tabular}
\\

\noindent and let $\mu_1(x) = 1$, $\mu_1(y) = 0$ and $\mu_2(x) = 1$ and $\mu_2(y)
=\varepsilon$ be two measures on $X$.  In the GHLP distance, the
two metric measure spaces, $(X,d,\mu_1)$ and $(X,d,\mu_2)$, are close to each other if $\varepsilon$ is small.  We do
want to preserve this property when we generalise.   
\end{exa}  

\begin{exa}
Let $X = \{x,y\rbrace$, $\mu_1 = \mu_2$ be defined as $\mu_i(x)= 1 = \mu_i(y)$ and
$d_i$ for $i = 1,2$ is defined as follows:\\

\begin{tabular}{l r}
\begin{minipage}{0.5\textwidth}
\begin{itemize}
\item $d_i(x,x) = 0 = d_i(y,y)$.
\item $d_i(x,y) = l_i$ for some fixed $l\in \real$.
\end{itemize}
\end{minipage}
&
\begin{minipage}{0.5\textwidth}
\includegraphics[scale=.4]{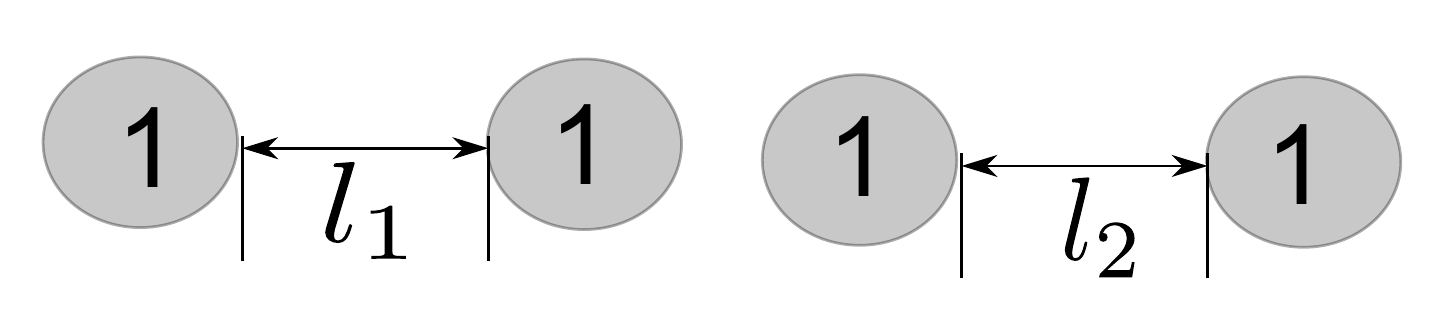}
\end{minipage}
\end{tabular}
\\

In the GHLP distance, these two metric measure spaces,
$(X,d,\mu_1)$ and $(X,d,\mu_2)$, are close to each other when $\vert l_1 - l_2
\vert$ is small.  We do want to preserve this property when we generalise.
\end{exa}

\noindent In addition to these two things we want two more crucial ingredients to
understand how to generalise.  Look at the following slightly different variants
of the above two examples.

\begin{exa}
Let $X = \{x,y\rbrace$, $d_1 = d_2 = d$ where $d$ is defined as follows:\\

\begin{tabular}{l r}
\begin{minipage}{0.5\textwidth}
\begin{itemize}
\item $d(x,x) = 0 = d(y,y)$.
\item $d(x,y) = \infty$.
\end{itemize}
\end{minipage}
&
\begin{minipage}{0.5\textwidth}
\includegraphics[scale=.4]{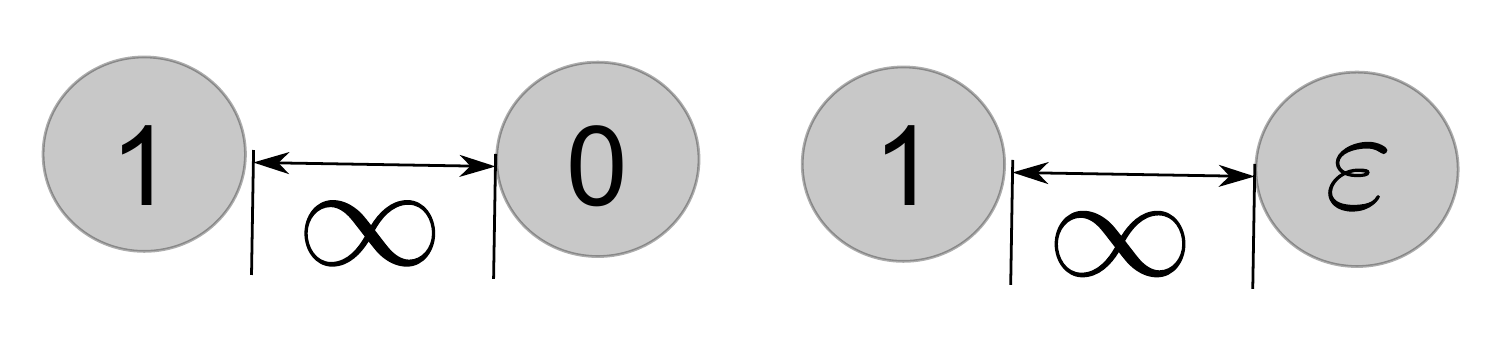}
\end{minipage}
\end{tabular}
\\

Let $\mu_1$ and $\mu_2$ be measures on $X$ given by $\mu_1(x) = 1$, $\mu_1(y) =
0$, $\mu_2(x) = 1$ and $\mu_2(y)
=\varepsilon$.  We want the two distance measure spaces, $(X,d,\mu_1)$ and $(X,d,\mu_2)$, to be
close to each other when $\varepsilon$ is small, in the generalised GHLP
distance.
\end{exa} 

\begin{exa}
Let $X = \{x,y\rbrace$, $\mu_1 = \mu_2$ be defined as $\mu_i(x)= 1 = \mu_i(y)$ and
$d_i$ for $i = 1,2$ is defined as follows:\\

\begin{tabular}{l r}
\begin{minipage}{0.5\textwidth}
\begin{itemize}
\item $d_i(x,x) = 0 = d_i(y,y)$.
\item $d_1(x,y) = L$ for some fixed $L\in \real$.
\item $d_2(x,y) = \infty$.
\end{itemize}
\end{minipage}
&
\begin{minipage}{0.5\textwidth}
\includegraphics[scale=.4]{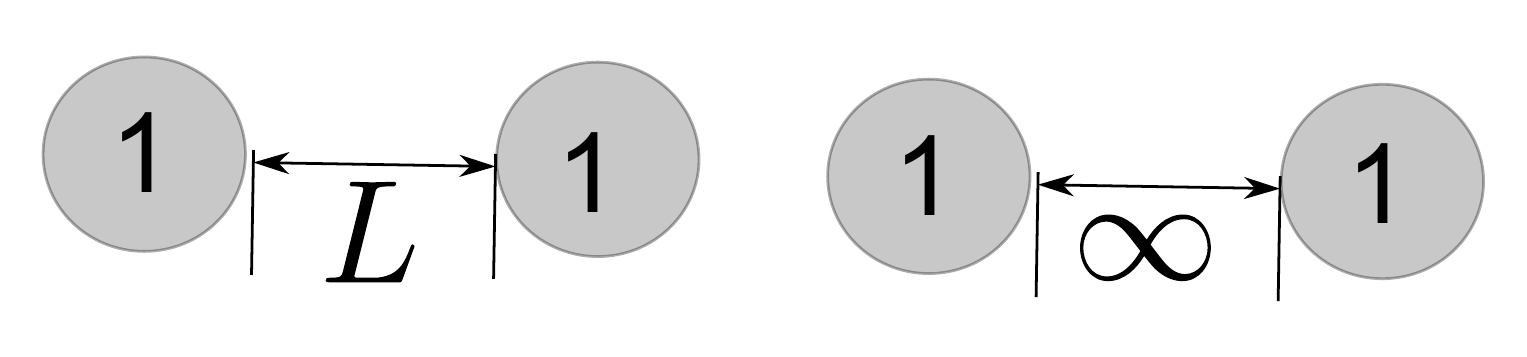}
\end{minipage}
\end{tabular}
\\

We want the two distance measure spaces,
$(X,d,\mu_1)$ and $(X,d,\mu_2)$, to be close to each other when $L$ is large, in the generalised GHLP distance. 
\end{exa}

\noindent Now we give the precise definition of the generalised GHLP distance.

\begin{dfn}[$L$--isometric embedding]\label{Lisometric}
A map $i:(M,d,\mu)\to (N,d',\nu)$ is called an
$L$--isometric embedding if $d'(f(x),f(y))=
d(x,y)$ whenever $\min \{d'(f(x),f(y)), d(x,y) \rbrace< L$.
\end{dfn}

\begin{dfn}[$L$--isometric $\varepsilon$--embedding]\label{Lisomepsemb}
A map from a distance measure space $(M,d,\mu)$ to a distance space $(N,d')$, $i:(M,d,\mu)\to (N,d')$, is called an
$L$--isometric $\varepsilon$--embedding if there exits $E \subset M$ such that
$\mu(E)\leq\varepsilon$ and $i:(M \setminus E)\to N$ is an
$L$--isometric embedding.
\end{dfn}

\begin{dfn}[Generalised GHLP distance, $d_{\rho}$]\label{rhodfn}
Consider $L$-isometric $\varepsilon$-embeddings $\iota_i : X_i \to Z, i = 1, 2$
of the spaces $X_i$ into a distance space $Z$. These give rise to push forward
measures $(\iota_i)_*(\mu_i)$ on $Z$. The distance between the distance measure
spaces is the infimum of an appropriate distance between the push forward 
measures over
all $L$--isometric $\varepsilon$--embeddings, namely,
   $$d_{\rho} = \inf\left\lbrace d_{\pi}((\iota_1)_*(\mu_1), (\iota_2)_*(\mu_2)) + \frac{1}{L} +
\varepsilon \ | \ \text{ $\iota_i : X_i \to Z$ an $L$-isometric 
$\varepsilon$-embedding}
\right\rbrace$$
where $Z$ varies over all metric spaces, $L\in\real^+$ and $\varepsilon \in
\real^+ \cup \{0\rbrace$.  
\end{dfn}

\noindent As remarked after the definition of GHLP distance, we can assume that, $Z=\iota_1(X_1) \cup \iota_2(X_2)$.

\begin{lemma}
If $d_{\rho}((X_1,d_1,\mu_1),(X_2,d_2,\mu_2)) = c$ then, $\vert \mu_1(X_1) - \mu_2(X_2)
\vert \leq c$ 
\label{total-measures-close-when-the-metric-measure-spaces-are}  
\end{lemma}

\begin{proof}
Assume the contrary.   Then, for every space $Z$ and every $L$-isometric
$\varepsilon$-embedding $\iota_i: X_i \to Z$, if we denote $\nu_i =
(\iota_i)_*(\mu_i)$ then, 
$$\vert \nu_1(Z) - \nu_2(Z) \vert = \vert \mu_1(X_1)- \mu_2(X_2) \vert > c.$$
Thus, $d_{\pi}(\nu_1,\nu_2) > c$ and hence $d_{\rho}((X_1,d_1,\mu_1),(X_2,d_2,\mu_2)) >
c$.  A contradiction. 
\end{proof}

\noindent From the definition of $d_{\rho}$ we have

\begin{lemma}
Let $(X,d)$ and $(X',d')$ be two distance spaces and let $f: (X,d)
\hookrightarrow (X',d') $ be an isometric embedding.  If $\mu$ is a measure on
$(X,d)$  and $\mu'$ is a measure on $X'$, then,
$$d_{\pi}^{(X',d')}(f_*(\mu),\mu') \geq d_{\rho}((X,d,\mu), (X',d',\mu')).$$
\label{pirho}
\end{lemma}

\subsection{Gluing and maximal metric}
A wonderful exposition of the ideas in this section, with more details, can be found in \cite{metricgeometry}.
\subsubsection{Gluing}
Let $(X,d)$ be a metric space and let $R$ be an equivalence relation on $X$.  The quotient
semi-distance (a generalisation of distance, allowing distance between unequal points to be
zero) $d_R$ is defined as
$$d_R(x,y) = \inf \left\lbrace \sum_{i=1}^k d(p_i,q_i) : p_1 = x, q_k = y, k\in \mathbb{N}  \right\rbrace$$
where the infimum is taken over all choices of $\{p_i\}$ and $\{q_i\}$ such that $q_i$ is
$R$-equivalent to $p_{i+1}$ for all $i = 1,...,k-1$.  We associate with the semi-distance
space, $(X,d_R)$, the distance space $(X/d_R, d_R)$ by identifying points with zero $d_R$-distance.
Given two distance spaces $(X,d_X)$ and $(Y,d_Y)$ we can define a distance on the union $X \cup Y$ as follows:
\begin{itemize}
\item For $x,x'$ in $X$, $d_{X\cup Y}(x,x') = d_X(x,x')$.
\item For $y,y'$ in $Y$, $d_{X\cup Y}(y,y') = d_Y(y,y')$.
\item If $x\in X$ and $y\in Y$,$d_{X\cup Y}(x,y) = \infty$.
\end{itemize}
If we have a bijection $I:X' \to Y'$ between subsets $X'$ and $Y'$, consider the
equivalence relation $R$ on $Z = X\cup Y$ generated by the relations $x \sim y$ if 
$y = I(x)$.  The result of gluing $X$ and $Y$ along $I$ by definition is the quotient space
$(Z/d_R,d_R)$.

\subsubsection{Maximal metric}
\label{maxmet}
Let $b: X\times X \to \mathbb{R}^{+}\cup \{\infty\}$ be an arbitrary function.  Consider
the class $D$ of all semi-distances $d$ on $X$ such that $d\leq b$, that is, 
$d(x,y)\leq b(x,y)$ for all $x,y\in X$.  Then $D$ has a unique maximal semi-distance 
$d_m$ such that $d_m\geq d$ for all $d\in D$, given by 
$$d_m(x,y)= sup\{d(x,y):d\in D\}.$$

Thus, given a set $X$ covered by a collection of subsets $\{X_{\alpha}\}$ each carrying a semi-distance $d_{\alpha}$, consider the class of semi-distance $D$ of all semi-distances $d$, such that $d(x,y) \leq d_{\alpha}(x,y)$ whenever $x,y \in X_{\alpha}$.  Then there is a unique maximal semi-distance $d_m \in D$.
 
Theorem 3.1.27 in \cite{metricgeometry} tells us that the distance as a result of gluing $X$ and $Y$ along $I$ is the maximal metric on $(X\cup Y, d_{X\cup Y})$, denoted as $d_I$, such that
\begin{itemize}
\item For all $z,z'\in X\cup Y$, $d(x,x') \leq d_{X\cup Y}(z,z')$.
\item $d(x,I(x)) = 0$.
\end{itemize}
We can generalise this concept a bit.  Given two distance spaces $(X,d_X)$ and $(Y,d_Y)$
and a function $I:X' \to Y'$ between subsets $X'$ and $Y'$, such that $d_X(x,x')<\delta$ if
$I(x) = I(x')$, we can consider the maximal metric, denoted by $d_I^{\delta}$, such that
\begin{itemize}
\item For all $z,z'\in X\cup Y$, $d(x,x') \leq d_{X\cup Y}(z,z')$.
\item $d(x,I(x)) \leq \delta$.
\end{itemize} 

\begin{theorem}
Let $R$ be the equivalence relation generated by the relations $x \sim y$ if 
$y = I(x)$.  The maximal distance $d_I^{\delta} = \inf \left\lbrace \sum_{i=1}^k d_{X\cup Y}(p_i,q_i) + (k-1)\delta : p_1 = x, q_k = y, k\in \mathbb{N}  \right\rbrace =:d_s$
where the infimum is taken over all choices of $\{p_i\}$ and $\{q_i\}$ such that $q_i$ is $R$-equivalent to $p_{i+1}$ for all $i = 1,...,k-1$.
\end{theorem}   

\begin{proof}
Let $D$ denote the class of semi-distances satisfying
\begin{itemize}
\item For all $z,z'\in X\cup Y$, $d(z,z') \leq d_{X\cup Y}(z,z')$.
\item For all $x\in X'$, $d(x,I(x)) \leq \delta$.
\end{itemize}
As, $d_{X\cup Y}(z,z')$ and $d_{X\cup Y}(x,x) + d_{X\cup Y}(I(x),I(x)) + \delta$ belong to the set over which the infimum is taken, $d_s \in D$.  So, it suffices to prove that $d_s \geq d$ for any semi-distance 
$d\in D$.  Let $x,y \in X\cup Y$ and $\{p_i\}$ and $\{q_i\}$ be as defined.  Then by the triangle inequality
$$d(x,y) \leq \sum_{i=1}^{k}d(p_i,q_i) + \sum_{i=1}^{k-1}d(q_i,p_i+1) $$
Note that, $p_{i+1} = I(q_i)$ or $q_i = I(p_{i+1})$ or $I(q_i) = I(p_{i+1})$.  But, both $d(q_i,I(q_i)) \text{ and } d(I(p_{i+i}),p_{i+1})$ are less than or equal to $\delta$.  And if $I(q_i) = I(p_{i+1})$, 
$$d(q_i,p_{i+1}) \leq d_{X\cup Y}(q_i,p_{i+1})\leq d_X(q_i,p_{i+1}) \leq \delta.$$ Hence,
$$d(x,y) \leq \sum_{i=1}^{k}d_{X\cup Y}(p_i,q_i) + \sum_{i=1}^{k-1}\delta, $$ that is, $d\leq d_s$.
\end{proof}

\begin{theorem}
If $\vert d_X(x,y) - d_Y(I(x),I(y))\vert \leq \delta$ for all $x,y \in X$, then the inclusions $i_X: (X,d_X) \to (X\cup Y,d_I^{\delta})$ and $i_Y(Y,d_Y) \to (X\cup Y, d_I^{\delta})$ are isometries.
\label{inclmaxmet}
\end{theorem}

\begin{proof}
Let $R$ be the equivalence relation generated by the relations $x \sim y$ if 
$y = I(x)$.  Let $\{p_i\}$ and $\{q_i\}$ be such that $q_i$ is $R$-equivalent to $p_{i+1}$.
Using previous theorem, without loss of generality we can assume that, if $p_i$ belongs to
$X$, then $q_i$ belongs to $X$ as well, because, otherwise $d_{X\cup Y}(p_i,q_i) = \infty$.
Similarly, we can assume that, if $p_i$ belongs to $Y$, then $q_i$ belongs to $Y$.
  So, $k$ is odd.  
  
\subsubsection*{Case 1:} Both $x,x' \in X$ (that is, $p_1,q_k\in X$) and $p_2\in Y$. Then, observe that   

\begin{align*}
d_{X\cup Y}(p_1,q_1) + d_{X\cup Y}(p_{2},q_{2}) + \delta &= d_{X\cup Y}(p_1,q_1) + d_{X\cup Y}(I(q_1),I(p_{3})) + \delta\\
&\geq d_{X\cup Y}(p_1,q_1) + d_{X\cup Y}(q_1,p_{3}) - \delta + \delta\\
&= d_{X\cup Y}(p_1,q_1) + d_{X\cup Y}(q_1,p_{3})\\
&\geq d_{X\cup Y}(p_1,p_{3})
\end{align*}

and,

$$d_{X\cup Y}(p_1,p_3) + d_{X\cup Y}(p_3,q_{3}) \geq d_{X\cup Y}(p_1,q_{3}).$$

Therefore, by induction, we have,
\begin{align*}
\sum_{i=1}^k d_{X\cup Y}(p_i,q_i) + (k-1)\delta &\geq d_{X\cup Y}(p_1,q_k) + \frac{(k-1)\delta}{2}\\
&\geq d_{X\cup Y}(p_1,q_k) = d_X(x,x').
\end{align*} 

\subsubsection*{Case 2:} Both $x,x' \in X$, that is, $p_1,q_k\in X$ and, $p_2\in X$. Then,
$$d_{X\cup Y}(p_1,q_1) + d_{X\cup Y}(p_{2},q_{2}) \geq d_{X\cup Y}(p_1,q_2)$$
So, we are in Case 1 or Case 2 again, with lesser number of terms.  Therefore by induction and Case 1 we have the result.  

\subsubsection*{Case 3} Both $y,y'\in Y$, that is $p_1,q_k\in Y$ and $p_2\in X$.  Then, observe that,
\begin{align*}
d_{X\cup Y}(p_1,q_1) + d_{X\cup Y}(p_{2},q_{2}) + \delta &\geq d_{X\cup Y}(p_1,I(p_2)) + d_{X\cup Y}(p_{2},q_{2}) + \delta\\
&\geq d_{X\cup Y}(p_1,I(p_2)) + d_{X\cup Y}(I(p_{2}),I(q_{2})) - \delta + \delta\\
&= d_{X\cup Y}(p_1,I(p_2)) + d_{X\cup Y}(I(p_{2}),I(q_{2})) \\
&\geq d_{X\cup Y}(p_1,I(q_2)) 
\end{align*}

and 
$$d_{X\cup Y}(p_1,I(q_2)) + d_{X\cup Y}(I(q_2),q_{3}) \geq d_{X\cup Y}(p_1,q_3).$$

Therefore, by induction, we have,
\begin{align*}
\sum_{i=1}^k d_{X\cup Y}(p_i,q_i) + (k-1)\delta &\geq d_{X\cup Y}(p_1,q_k) + \frac{(k-1)\delta}{2}\\
&\geq d_{X\cup Y}(p_1,q_k) = d_X(y,y').
\end{align*}

\subsubsection*{Case 4:} Both $y,y'\in Y$, that is $p_1,q_k\in Y$ and $p_2\in Y$.  Then,
$$d_{X\cup Y}(p_1,q_1) + d_{X\cup Y}(p_{2},q_{2}) \geq d_{X\cup Y}(p_1,q_2).$$
So, we are in Case 3 or Case 4 again, with lesser number of terms.  Therefore by induction and Case 3 we have the result. 
  
\end{proof}

We continue to use the same notation as above.
\begin{theorem}
If $I$ is an $L$-isometric embedding then so are the inclusion maps $i_X: (X,d_X) \to \left(X\cup Y,d_I^{\delta}\right)$ and $i_Y:(Y,d_Y) \to \left(X\cup Y, d_I^{\delta}\right)$ for all $\delta \geq 0$.
\label{inclmxmetlisom}
\end{theorem}

\begin{proof}
If $x,x' \in X$ are such that $d_X(x,x') < L$, then $d_Y(I(x),I(y)) = d_X(x,y)$.  Therefore, $\vert d_Y(I(x),I(y)) - d_X(x,y) \vert = 0 < \delta$.  Hence, by the same arguments
as in Theorem \ref{maxmet} we have $d_I^{\delta}(x,x') = d_X(x,x')$.  Similarly, if $y,y'\in Y$ and $d_Y(y,y') < L$, then $d_I^{\delta}(y,y') = d_Y(y,y')$.    
\end{proof}

\subsection{Triangle inequality}
We can now prove
triangle inequality for $d_{\rho}$ by the arguments similar to the one in
\cite{gadgil} for the Gromov-Hausdorff-Prokhorov distance.  Let $(X_1,d_1, \mu_1)$,
$(X_2,d_2,\mu_2)$ and $(Y,d_Y,\nu)$ be distance measure spaces.  Then,

\begin{theorem}[Triangle Inequality] 
$$d_{\rho}((X_1,d_1,\mu_1),(X_2,d_2,\mu_2)) \leq d_{\rho}((X_1,d_1,\mu_1),(Y,d_Y,\nu)) + d_{\rho}((X_2,d_2,\mu_2), (Y,d_Y,\nu))$$
\label{triangleinequality}
\end{theorem}   

\begin{proof}
Let $\varepsilon > 0$ be arbitrary.  By definition, for $i=1,2$ we can find
spaces $Z_i$ and maps $f_i$ and $g_i$ such that $f_i: X_i \to Z_i,g_i:Y \to Z_i$
are $L_i$-isometric
$\varepsilon_i$-embeddings and $Z_i = f_i(X_i)\cup g_i(Y)$ and
$d_{\pi}((f_i)_*(\mu_i),(g_i)_*(\nu)) + \frac{1}{L_i} + \varepsilon_i \leq d_{\rho}(X_i,Y) +
\varepsilon$.  Now let $Z$ be the distance space obtained from $Z_1\sqcup Z_2$
by identifying $g_1(y)$ with $g_2(y)$ for all $y \in Y$, and the distance
the maximal distance as defined in Section \ref{maxmet}.  Denote the inclusion from $Z_i$
 to $Z$ by $\iota_i$.  Then, $\iota_i \circ f_i: X_i \to Z$ are 
$min\{L_1,L_2\rbrace$-isometric $max\{\varepsilon_1,\varepsilon_2\rbrace$-embeddings by Theorem \ref{inclmxmetlisom}.
  
\begin{equation*}
d_{\rho}(X_1,X_2) \leq d_{\pi}^Z((\iota_1 \circ f_1)_*(\mu_1),(\iota_2 \circ f_2)_*(\mu_2)) +
max\{\varepsilon_1,\varepsilon_2\rbrace + max\left\lbrace \frac{1}{L_1}, \frac{1}{L_2}\right\rbrace
\end{equation*} 

\noindent By triangle inequality for $d_{\pi}$,
\begin{align*}
d_{\pi}^Z((\iota_1 \circ f_1)_*(\mu_1),(\iota_2 \circ f_2)_*(\mu_2)) &\leq d_{\pi}^Z((\iota_1 \circ f_1)_*(\mu_1),(\iota_1 \circ g_1)_*(\nu)) +
d_{\pi}^Z((\iota_1 \circ g_1)_*(\nu),(\iota_2 \circ f_2)_*(\mu_2))\\
&\leq d_{\pi}^Z((\iota_1 \circ f_1)_*(\mu_1),(\iota_1 \circ g_1)_*(\nu)) + d_{\pi}^Z((\iota_1 \circ g_1)_*(\nu),(\iota_2 \circ g_2)_*(\nu))\\
& \ \ \ \ \ \ \ \  + d_{\pi}^Z((\iota_2 \circ g_2)_*(\nu),(\iota_2 \circ f_2)_*(\mu_2)) \\
&\leq d_{\pi}^{Z_1}((f_1)_*(\mu_1),(g_1)_*(\nu)) + d_{\pi}^Z((\iota_1 \circ g_1)_*(\nu),(\iota_2 \circ g_2)_*(\nu))\\
& \ \ \ \ \ \ \ \  + d_{\pi}^{Z_2}((g_2)_*(\nu),  (f_2)_*(\mu_2))\\
&\leq d_{\pi}^{Z_1}((f_1)_*(\mu_1),(g_1)_*(\nu)) + (\varepsilon_1 + \varepsilon_2) + d_{\pi}^{Z_2}((g_2)_*(\nu),  (f_2)_*(\mu_2))\\
& \ \ \ \ (\text{$g_i$ is an $\varepsilon_i$-embedding.})\\
&\leq d_{\pi}^{Z_1}((f_1)_*(\mu_1),(g_1)_*(\nu)) + 2\varepsilon + d_{\pi}^{Z_2}((g_2)_*(\nu),(f_2)_*(\mu_2)) 
\end{align*}  
We also have $max\{\varepsilon_1,\varepsilon_2\rbrace\leq \varepsilon_1 +
\varepsilon_2$ and $max\left\lbrace \frac{1}{L_1}, \frac{1}{L_2}\right\rbrace \leq \frac{1}{L_1} +
\frac{1}{L_2}$.  Thus by the inequalities 
$$d_{\pi}^{Z_i}((f_i)_*(\mu_i),(g_i)_*(\nu))  + \varepsilon_i +\frac{1}{L_i} \leq d_{\rho}(X_i,Y) +
\varepsilon$$
we have,
\begin{align*}
d_{\pi}^Z((\iota_1 \circ f_1)_*(\mu_1),(\iota_2 \circ f_2)_*(\mu_2)) + max\{\varepsilon_1,\varepsilon_2\rbrace +
max\left\lbrace \frac{1}{L_1}, \frac{1}{L_2}\right\rbrace &\leq d_{\pi}^{Z_1}((f_1)_*(\mu_1),(g_1)_*(\nu)) \\
& + d_{\pi}^{Z_2}((g_2)_*(\nu),(f_2)_*(\mu_2)) \\
& + \varepsilon_1 + \varepsilon_2 + \frac{1}{L_1} + \frac{1}{L_2} + 2\varepsilon\\  
&\leq d_{\rho}(X_1,Y) + d_{\rho}(X_2,Y) + 4\varepsilon
\end{align*}
As $\varepsilon>0$ was arbitrary the result follows.
\end{proof}

\subsection{Quasi-isometry between nearby spaces}

\begin{theorem}
Given spaces $(X,d_X,\mu)$ and $(Y,d_Y,\nu)$ such that $d_{\rho}((X,d_X,\mu),(Y,d_Y,\nu)) < \delta < \frac{1}{\sqrt{2}}$, there exists
\begin{itemize}
\item a subset $\widehat{X} \subset X$ of measure less than $\delta$,
\item a function $f: (X \setminus \widehat{X}) \to Y$ such that $d_Y(f(x),f(x')) < d_X(x,x') + 2\delta$ and 
      $d_X(x,x') < d_Y(f(x),f(x')) + 2\delta$ for all $x,x'$ such that 
      $d_Y(f(x),f(x')) < \frac{1}{\delta} - 2\delta$ or $d_X(x,x') < \frac{1}{\delta} - 2\delta$
\item for all $E \subset Y$, $\mu(f^{-1}(E)) \leq \nu(E^{2\delta}) + 2\delta$ and $\nu(E) \leq \mu(f^{-1}(E)^{2\delta}) + 2\delta$,
\item $\nu(Y)\leq \nu(f(X\setminus \widehat{X})^{2\delta}) + 2\delta$
\end{itemize}
Further we have,
\begin{itemize}
\item $\left( f^{-1}(E) \right)^{2\delta} \subset f^{-1}\left(E^{2\delta}\right)$
\item If $\delta < \frac{1}{2}$, 
$\nu(E) \leq \mu\left(f^{-1}(E^{2\delta})\right) + 2\delta$.
\end{itemize}

\label{quasiisomgivdistsmall}
\end{theorem}

\begin{proof}
The proof of this theorem will take nearly the whole of this subsection.  To summarize, we will construct a set $\widehat{X}$ and a function $f: (X \setminus \widehat{X}) \to Y$ and show that it satisfies all the properties.
 
Note that as $d_{\rho}(X,Y) < \delta$, there exists a metric space $(Z,d)$ and $L$-isometric $\varepsilon$-embeddings $\varphi : X \to Z$ and $\psi: Y\to Z$ such that 
$$d_{\pi}(\varphi_*(\mu),\psi_*(\nu)) + \frac{1}{L} + \varepsilon < \delta_1 := \left(\frac{d_{\rho}(X,Y) + \delta}{2}\right)  < \delta.$$
Thus, $\frac{1}{L}$ and $\varepsilon$ are both less than $\delta$.  For all $x\in X, d(\varphi(x),\psi(Y))$ is attained for some point
$y\in Y$ as $\psi(Y)$ is closed.  Define $f(x)$ to be one such $y$.  Note that, except for the set $\widehat{X} = \lbrace x\in X : d(\phi(x),\psi(Y))>\delta\rbrace$, of measure less than $\delta$,
$d(\varphi(x),\psi(f(x))) = d(\varphi(x),\psi(Y)) < \delta$.  So, even though the choice of
 $f(x)$ is not unique, any two choices are not farther than $2\delta$ from each other.
 
\begin{lemma}
If $\delta_2 = \frac{\delta - \delta_1}{2} = \frac{\delta - d_{\rho}(X,Y)}{4}$ then, $\mu\left(\left(\widehat{X} \right)^{\delta_2}\right) \leq \delta_1$.
\label{trickforlimmapbnlimsp} 
\end{lemma}

\begin{proof}
Note that,  as $x\in \left(\widehat{X} \right)^{\delta_2}$, there exists $x'\in \widehat{X}$ such that $d_X(x,x')<\delta_2 < 
\delta < \frac{1}{\delta}$.  Thus, 
\begin{align*}
\delta &< d(\varphi(x'),\psi(Y)) \leq d(\varphi(x'),\varphi(x)) + d(\varphi(x),\psi(Y))\\
&= d_X(x,x') + d(\varphi(x),\psi(Y))\\
&=\delta_2 + d(\varphi(x),\psi(Y)).
\end{align*}
That is, $d(\varphi(x),\psi(Y)) > \delta - \delta_2 = \frac{\delta + \delta_1}{2} > \delta_1$.  As $x\in \left(\widehat{X} \right)^{\delta_2}$
was arbitrary we get, $\left(\widehat{X} \right)^{\delta_2} \subset \lbrace x\in X: d(\varphi(x),\psi(Y))>\delta - \delta_2\rbrace=: S$.
Note that as $d_{\pi}(\varphi_*(\mu),\psi_*(\nu)) + \frac{1}{L} + \varepsilon < \delta_1$, $d_{\pi}(\varphi_*(\mu),\psi_*(\nu)) < \delta_1$.
That is  $\mu\left(\left(\widehat{X} \right)^{\delta_2}\right) \leq \mu(S) = \varphi_*(\mu)\left(\varphi(S)\right) \leq \psi_*(\nu)\left(S^{\delta_1}\right) + \delta_1 =\delta_1$.  
\end{proof}
Lemma \ref{trickforlimmapbnlimsp} will be used in the proof of Proposition \ref{limmapbnlimsp}.

\begin{lemma}
For all $x,x'$ such that $d_Y(f(x),f(x')) < \frac{1}{\delta} - 2\delta$ or $d_X(x,x') < \frac{1}{\delta} - 2\delta$ we have $d_Y(f(x),f(x')) < d_X(x,x') + 2\delta$ and $d_X(x,x') < d_Y(f(x),f(x')) + 2\delta$.
\end{lemma}

\begin{proof}
As $\frac{1}{L} <\delta$, that is, $\frac{1}{\delta} < L$, $d(\varphi(x),\varphi(x')) = d_X(x,x')$ for all 
$x,x' \notin \widehat{X}$ with $d(\varphi(x),\varphi(x'))$ $< \frac{1}{\delta}$ or $d_X(x,x') < \frac{1}{\delta}$.
We have, $\frac{1}{\delta} - 2\delta < L$, therefore, $d_X(x,x') < \frac{1}{\delta} - 2\delta $ implies that 
$d(\phi(x),\phi(x')) = d_X(x,x')$.  Thus, $d(\phi(x),\phi(x')) \leq \frac{1}{\delta} - 2\delta$.  So, 
$$d(\psi(f(x)),\psi(f(x'))) \leq d(\psi(f(x)),\varphi(x)) + d(\varphi(x),\varphi(x')) + d(\varphi(x'),\psi(f(x'))) \leq \frac{1}{\delta}$$
Similarly we can prove that if $d_Y(f(x),f(x'))<\frac{1}{\delta} - 2\delta$ then, $d(\varphi(x),\varphi(x'))<\frac{1}{\delta}$.
Hence, for all $x,x' \notin \widehat{X}$ with $d_Y(f(x),f(x')) < \frac{1}{\delta} - 2\delta$ or $d_X(x,x') < \frac{1}{\delta} - 2\delta$; 
\begin{align*}
d_Y(f(x),f(x')) &= d(\psi(f(x)),\psi(f(x'))) \\
& \leq d(\psi(f(x)),\varphi(x)) + d(\varphi(x),\varphi(x')) + d(\varphi(x'),\psi(f(x')))\\
& < \delta + d_X(x,x') + \delta = d_X(x,x') + 2\delta.
\end{align*}

\noindent If $f(x) \neq f(x')$ then,
\begin{align*}
d_X(x,x') & = d(\varphi(x), \varphi(x')) \\
& \leq d(\varphi(x),\psi(f(x))) + d(\psi(f(x)), \psi(f(x'))) + d(\psi(f(x')), \varphi(x'))\\
& < \delta + d_Y(f(x),f(x')) + \delta = d_Y(f(x),f(x')) + 2\delta.
\end{align*}

\noindent If $f(x) = f(x')$ then 
\begin{align*}
d_X(x,x') &= d(\varphi(x), \varphi(x'))\\
&\leq d(\varphi(x),\psi(f(x))) + d(\psi(f(x')), \varphi(x'))\\
&< 2\delta = d_Y(f(x),f(x')) + 2\delta.
\end{align*}
\end{proof}

\begin{lemma}
For all $E \subset Y$, $\mu(f^{-1}(E)) \leq \nu(E^{2\delta}) + 2\delta$ and $\nu(E) \leq \mu(f^{-1}(E)^{2\delta}) + 2\delta$
\end{lemma}
 
\begin{proof}
\noindent For all $E \subset Y$, 
\begin{align*}
\mu(f^{-1}(E)) &\leq \varphi_*(\mu)(\varphi(f^{-1}(E))) + \varepsilon : \ \ (\text{as $\varphi$ is an $\varepsilon$-embedding})\\
 &\leq \varphi_*(\mu)((\psi(E))^{\delta}) + \varepsilon : \ \ \left(\text{as $\varphi(f^{-1}(E)) \subset \psi(E)$}\right)\\
 &\leq \psi_*(\nu)((\psi(E))^{2\delta}) + \varepsilon + \delta \ : (\text{as $d_{\pi}(\varphi_*(\mu),\psi_*(\nu))<\delta$})\\
 &\leq \nu\left(\psi^{-1}\left((\psi(E))^{2\delta}\right)\right) + \varepsilon + \delta \\
\end{align*}
Observe that $x\in \psi^{-1}\left((\psi(E))^{2\delta}\right)$ implies that $\psi(x)\in (\psi(E))^{2\delta}$ that is, $d(\psi(x),\psi(E))< 2\delta$.
As $\psi$ is an $L$-isometric $\varepsilon$-embedding with $L > \frac{1}{\delta} > 2\delta$, $x\in E^{2\delta}$.  
Thus, $\psi^{-1}\left((\psi(E))^{2\delta}\right) \subset E^{2\delta}$.  Thus,
\begin{align*}
\mu(f^{-1}(E)) &\leq \nu\left(\psi^{-1}\left((\psi(E))^{2\delta}\right)\right) + \varepsilon + \delta \\
 &\leq \nu(E^{2\delta}) + \delta + \varepsilon \\
 &\leq \nu(E^{2\delta}) + 2\varepsilon.
\end{align*}

Similarly, $\forall E \subset Y$,
\begin{align*}
\nu(E) \leq \psi_*\left(\nu\right)\left(\left(\varphi\left(f^{-1}(E)\right)\right)^{\delta}\right) + \varepsilon \leq \varphi_*(\mu)\left(\left(\varphi\left(f^{-1}(E)\right)\right)^{2\delta}\right) + \varepsilon + \delta \leq \mu\left(\left(f^{-1}(E)\right)^{2\delta}\right) + 2\delta.
\end{align*}
\end{proof}

\begin{lemma}
$\nu(Y)\leq \nu(f(X\setminus \widehat{X})^{2\delta}) + 2\delta$
\end{lemma}

\begin{proof}
\begin{align*}
\nu(Y) &\leq \delta + \psi_*(\nu)(Z) = \delta + \psi_*(\nu)\left(Z\setminus \left(\varphi\left(X\setminus \widehat{X}\right)^{\delta}\right)\right) + \psi_*(\nu)\left(\phi\left(X\setminus \widehat{X}\right)^{\delta}\right)\\
 &\leq \delta + \varphi_*(\mu)\left(\left(Z\setminus \left(\varphi\left(X\setminus \widehat{X}\right)^{\delta}\right)\right)^{\delta}\right) + \delta + \psi_*(\nu)\left(\varphi\left(X\setminus \widehat{X}\right)^{\delta}\right)\\
 &\leq \delta + 0 + \delta + \psi_*(\nu)\left(\varphi\left(X\setminus \widehat{X}\right)^{\delta}\right) = \psi_*(\nu)\left(\varphi\left(X\setminus \widehat{X}\right)^{\delta}\right) + 2\delta\\
 &\leq \psi_*(\nu)\left(\psi\left(f\left(X\setminus \widehat{X}\right)\right)^{2\delta}\right) + 2\delta  \left\lbrace  \text{as $\varphi\left(X\setminus \widehat{X}\right) \subset \psi\left(f\left(X\setminus \widehat{X}\right)\right)^{\delta} $} \right\rbrace\\
 &\leq \nu\left(\psi^{-1}\left(\psi\left(f\left(X\setminus \widehat{X}\right)\right)^{2\delta}\right)\right) + 2\delta.
\end{align*}

\begin{lemma}
$\psi^{-1}\left(\psi\left(f\left(X\setminus \widehat{X}\right)\right)^{2\delta}\right) \subset f\left(X\setminus \widehat{X}\right)^{2\delta}$
\end{lemma}

\begin{proof}
\begin{align*}
x\in \psi^{-1}\left(\psi\left(f\left(X\setminus \widehat{X}\right)\right)^{2\delta}\right) &\implies \psi(x)\in \psi\left(f\left(X\setminus \widehat{X}\right)\right)^{2\delta}\\
&\implies \exists y\in f\left(X\setminus \widehat{X}\right) \ s.t \ d(\psi(x),\psi(y))< 2\delta < \frac{1}{\delta}\\
&\implies d_Y(x,y) = d(\psi(x),\psi(y))< 2\delta\\
&\implies x\in f\left(X\setminus \widehat{X}\right)^{2\delta}.   
\end{align*}
As $x\in \psi^{-1}(\psi(f(X\setminus \widehat{X}))^{2\delta})$ was arbitrary, $\psi^{-1}(\psi(f(X\setminus \widehat{X}))^{2\delta}) \subset f\left(X\setminus \widehat{X}\right)^{2\delta}$.
\end{proof}
So,
\begin{align*}
\nu(Y) &\leq \nu\left(\psi^{-1}\left(\psi\left(f\left(X\setminus \widehat{X}\right)\right)^{2\delta}\right)\right) + 2\delta\\
&\leq \nu\left(f\left(X\setminus \widehat{X}\right)^{2\delta}\right) + 2\delta.
\end{align*}

\end{proof}

\begin{lemma}
We have $\left( f^{-1}(E) \right)^{2\delta} \subset f^{-1}\left(E^{2\delta}\right)$.
\label{subsetcontrolonmeas}
\end{lemma}

\begin{proof}
 Let $x\in \left( f^{-1}(E) \right)^{2\delta}$.  Then, there exists $y \in X$ such that 
$d_X(x,y) < 2\delta < \frac{1}{\delta} - 2\delta$ and $f(y)\in E$.  So, $d_Y(f(x),f(y)) < d_X(x,y) + 2\delta$,
that is, $f(x)\in E^{2\delta}$.  Hence, $x\in f^{-1}\left(E^{2\delta}\right)$.  As $x\in \left( f^{-1}(E) \right)^{2\delta}$
was arbitrary $\left( f^{-1}(E) \right)^{2\delta} \subset f^{-1}\left(E^{2\delta}\right)$.  
\end{proof}

\begin{lemma}
If $\delta < \frac{1}{2}$, 
$\nu(E) \leq \mu\left(f^{-1}(E^{2\delta})\right) + 2\delta$.
\label{controlonmeas}
\end{lemma}

\begin{proof}
We already saw that $\nu(E) \leq \mu\left(\left(f^{-1}(E)\right)^{2\delta}\right) + 2\delta$.  But, by 
Lemma \ref{subsetcontrolonmeas} $\mu\left(\left(f^{-1}(E)\right)^{2\delta}\right) \leq \mu\left(f^{-1}\left(E^{2\delta}\right)\right)$.  So,
$\nu(E) \leq \mu\left(f^{-1}\left(E^{2\delta}\right)\right) + 2\delta$.
\end{proof}

Thus, we have completed the proof of Theorem \ref{quasiisomgivdistsmall} by proving all the necessary properties for the set $\widehat{X}$ and the function $f: (X \setminus \widehat{X}) \to Y$ we constructed. 
\end{proof}

From the proof of Theorem \ref{quasiisomgivdistsmall} the following lemma is clear.

\begin{lemma}
Let $d_{\rho}(X,Y) < \delta$.  Construct maps $f: \left(X \setminus \widehat{X}\right) \to Y$ and 
$g: \left(Y \setminus \widehat{Y}\right) \to X$ as in Theorem \ref{quasiisomgivdistsmall}.  Then for 
$x,y \in X \setminus \left(\widehat{X} \cup f^{-1}\left(\widehat{Y}\right)\right)$ with $d_X(x,y) < \frac{1}{\delta} - 4\delta$,  
$$d_X(x,y) -4\delta \leq d_X(g\circ f (x), g\circ f(y)) \leq d_X(x,y) + 4 \delta.$$
Further, $\mu\left(f^{-1}\left(\widehat{Y}\right)\right) \leq \nu\left(\widehat{Y}^{2\delta}\right) + 2\delta$.
\end{lemma}

The following lemma is clear.

\begin{lemma}
Let $d_{\rho}(X^n,X)$ converge to zero.  Construct maps $f^n: \left(X^n \setminus \widehat{X^n}\right) \to X$ and 
$g^n: \left(X \setminus \widehat{X}^n\right) \to X^n$ as in Theorem \ref{quasiisomgivdistsmall}.  Define $X' = X/ \thicksim$ 
where $\thicksim$ is defined as $x\thicksim y$ if $f^n \circ g^n (x) = f^n \circ g^n (y) $ for infinitely many $n$.  
Then $d_{\rho}(X^n, X')$ converges to zero. 
\label{nocollapse}
\end{lemma}

\subsection{Positivity of the distance}
As the definition of $d_{\rho}$ ignores, for example, isolated points of measure
zero, we do not expect $d_{\rho}((X_1,d_1,\mu_1),(X_2,d_2,\,u_2)) = 0$ to imply that
there is a measure preserving isometry between the two, but only that it is true
up to ignoring an appropriate class of sets with measure zero.

\begin{theorem}
Given two complete separable distance measure spaces $(X,d_X,\mu),(Y,d_Y,\nu)$, we have 
$d_{\rho}((X,d_X,\mu),(Y,d_Y,\nu)) = 0$ if and only if there are open sets
$U \subset X$ and $V \subset Y$, of zero measure, so that there is a
measure preserving isometry between  $(X \setminus U,d_X|_{X \setminus U},\mu)$ and 
$(Y \setminus V, d_Y|_{Y \setminus V},\nu)$.
\label{positivityofrho}
\end{theorem}

\begin{proof}
The proof of this theorem will take the whole of this subsection.  To summarize, for each $\frac{1}{n^2}$ we construct a quasi-isometry from a subset of $X$ to $Y$ as in Theorem \ref{quasiisomgivdistsmall} and take the limit as $n$ tends to infinity.  We will prove that such a limit exists and it is a measure preserving isometry from $X$ to $Y$.

For each $n$ the inequality $d_{\rho}((X,d_X,\mu),(Y,d_Y,\nu)) < \frac{1}{n^2}$ gives a map from 
$f_n: X\setminus X^n \to Y$ as in Theorem \ref{quasiisomgivdistsmall}.  Let $X_n = X\setminus X^n$.  
Then $(X_n,d|_{X_n} = d_n, \mu|_{X_n} = \mu_n)$ converges to $(X,d_X,\mu)$.  We shall construct a map $f:X \to Y$
which is the limit of the maps $f_n:X_n \to Y$.  If $X(k) = \cup_{i=k}^{\infty} X^i$, $\mu(X(k)) \leq \sum_{i=k}^{\infty} \frac{1}{i^2}$
which goes to zero as $k$ goes to infinity.  Thus $\mu(\cap_{k=1}^{\infty} X(k)) = 0$.  Denote $\cap_{k=1}^{\infty} X(k)$
by $\underline{X}$.  For any point $x\in X\setminus \underline{X}$, $x$ is in the domain of all but finitely many $n$'s.  

\begin{lemma}
Given $x\in X\setminus \underline{X}$, the infinite sequence $f_n(x)$ has a convergent subsequence. 
\end{lemma}
\begin{proof}
Assume the contrary.  Then there exists an $\varepsilon_1$ such that 
$$d_Y(f_n(x),f_m(x)) \geq \varepsilon_1$$
for all $n,m$.  Choose $\varepsilon < \min \left( \left\lbrace \left(n^2 - \frac{2}{n^2} \right),\left(\varepsilon_1 -  \frac{4}{n^2}\right) : n\in \mathbb{N} \right\rbrace \cap \left\lbrace x\in \mathbb{R} : x >0 \right\rbrace \right)$.
The terms $\frac{2}{n^2}$ and $\frac{4}{n^2}$ goes to zero as $n$ goes to infinity,  and the terms $\left(n^2 - \frac{2}{n^2} \right)$
and $\left(\varepsilon_1 -  \frac{4}{n^2}\right)$ increases with $n$.  Hence the minimum exists.  By using 
Theorem \ref{quasiisomgivdistsmall}, we have for all $n$ such that the terms $\left(n^2 - \frac{2}{n^2} \right)$ and
$\left(\varepsilon_1 -  \frac{4}{n^2}\right)$ are positive (i.e for all but finitely many $n$),
$$f_n(B(x,\varepsilon)) \subset B\left(f_n(x),\varepsilon + \frac{2}{n^2}\right).$$

\noindent Thus,
$$(f_n(B(x,\varepsilon)))^{\frac{2}{n^2}} \subset B\left(f_n(x),\varepsilon + \frac{2}{n^2}\right)^{\frac{2}{n^2}}.$$
So,
\begin{align*}
\nu(B\left(f_n(x),\varepsilon + \frac{2}{n^2}\right)^{\frac{2}{n^2}}) &\geq \nu((f_n(B(x,\varepsilon)))^{\frac{2}{n^2}})\\
&\geq \mu(B(x,\varepsilon)) - \frac{2}{n^2} \\
& \ \ \ \ \ \ \ \ (\text{apply Theorem \ref{quasiisomgivdistsmall}}).
\end{align*}
Furthermore, $(B(f_n(x),\varepsilon + \frac{2}{n^2}))^{\frac{2}{n^2}}$ are disjoint.  This contradicts with the assumption that $Y$ has finite measure.   
\end{proof}

Choose a countable dense subset $S = \lbrace x_1,x_2,...\rbrace \subset (X \setminus \underline{X})$.  Choose a subsequence $f_{n_{1,k}}$ of $f_n$ such that
$f_{n_{1,k}}(x_1)$ converges.  We choose a further subsequence $f_{n_{2,k}}$ of $f_{n_{1,k}}$
such that $f_{n_{2,k}}(x_2)$ converges. Observe that $f_{n_{2,k}}(x_1)$
continues to converge.
Iterating this process, we obtain subsequences $f_{n_{j,k}}$ so that
$f_{n_{j,k}}(x_l)$ converges for $l\leq j$. It follows that for the
diagonal sequence $f_{n_{k,k}}$ we have the corresponding convergence for
all points in $S$.  Replace $f_n$  by this diagonal subsequence.  Define 
$$f(x_i) = \lim_{n\to \infty}f_n(x_i)$$

\begin{lemma}
$d_Y(f(x_i),f(x_j)) \leq d_X(x_i,x_j).$
\label{fdisdec'}
\end{lemma}  

\begin{proof}
Given $x_i$ and $x_j$, if $d_Y(f(x_i),f(x_j)) = \infty = d_X(x_i,x_j)$, then it is obvious.  Otherwise, choose $n$
so large that $\min\lbrace d_Y(f(x_i),f(x_j)),d_X(x_i,x_j)\rbrace < n^2 - \frac{2}{n^2}$.  Then,
\begin{align*}
d_Y(f_n(x_i), f_n(x_j)) \leq d_X(x_i, x_j) + \frac{2}{n^2}.
\end{align*}
Thus, 
\begin{align*}
d_Y(f(x_i),f(x_j)) &= \lim_{n \to \infty} d_Y(f_n(x_i), f_n(x_j)) \\
&\leq \lim_{n \to \infty} d_X(x_i, x_j) + \frac{2}{n^2}.
\end{align*}
As $\frac{2}{n^2}$ tends to zero as $n$ tends to infinity we have the result.
\end{proof}

\begin{lemma}
$d_Y(f(x_i),f(x_j)) \geq d_X(x_i,x_j).$
\label{fdisinc'}
\end{lemma}  

\begin{proof}
Given $x_i$ and $x_j$, if $d_Y(f(x_i),f(x_j)) = \infty = d_X(x_i,x_j)$, then it is obvious.  Otherwise, choose $n$ so
large that $\min\lbrace d_Y(f(x_i),f(x_j)),d_X(x_i,x_j)\rbrace < n^2 - \frac{2}{n^2}$.  Then,

\begin{align*}
d_Y(f_n(x_i), f_n(x_j)) \geq  d_X(x_i, x_j) - \frac{2}{n^2}.
\end{align*}
Thus,

\begin{align*}
d_Y(f(x_i),f(x_j)) &= \lim_{n \to \infty} d_Y(f_n(x_i), f_n(x_j)) \\
&\geq \lim_{n \to \infty} d_X(x_i, x_j) - \frac{2}{n^2}.
\end{align*}
As $\frac{2}{n^2}$ tends to zero as $n$ tends to infinity we have the result.
\end{proof}

From Lemma \ref{fdisdec'} and Lemma \ref{fdisinc'} we have

\begin{lemma}
The map $f$ is an isometry on $S$.
\label{fisomonS1}
\end{lemma}

As $S$ is dense in $X \setminus \underline{X}$ given any point $X \setminus \underline{X}$ there exist a sequence
$s_n \in S$ such that $d_X(s_n,x)$ tends to $0$.  Define
$f(x)= \lim_{n\to \infty} f(s_n)$ where $s_n$ tends to $x$ as $n$ tends
to $\infty$.  
\subsubsection*{The map $f$ is well defined:} Let $\langle s_n \rangle_n$ and $\langle t_n \rangle_n$ be two
sequence which converge to $x$.  Then the sequence $s_1,t_1,s_2,t_2,...$ is
Cauchy and so is the sequence $f(s_1),f(t_1),
f(s_2),f(t_2),...$ by Lemma \ref{fisomonS1}.  Thus it converges
and has the same limit as $\langle f(s_n) \rangle_n$ and $\langle f(t_n) \rangle_n$.  Hence
$\lim_{n\to \infty} f(s_n) = \lim_{n\to \infty} f(t_n)$.

\subsubsection*{The map $f$ is an isometry:}Given two points $x,x' \in X \setminus \underline{X}$ consider sequences 
$\langle s_n \rangle_n \subset S$ and $\langle t_n \rangle_n \subset S$ such that $s_n$ converges to $x$ and $t_n$ 
converges to $x'$.  Then,
\begin{align*}
d_Y(f(x),f(x')) &\leq d_Y(f(x),f(s_n)) + d_Y(f(s_n),f(t_n)) + d_Y(f(t_n), f(x'))\\
&\leq d_Y(f(x),f(s_n)) + d_X(s_n,t_n) + d_Y(f(t_n), f(x'))
\end{align*}
Similarly, $d_Y(f(s_n),f(t_n)) \leq d_Y(f(x),f(s_n)) + d_Y(f(x),f(x')) + d_Y(f(t_n), f(x'))$ implies that
\begin{align*}
d_Y(f(x),f(x')) &\geq d_Y(f(s_n),f(t_n)) - (d_Y(f(x),f(s_n)) + d_Y(f(t_n), f(x'))\\
&= d_X(s_n,t_n) - (d_Y(f(x),f(s_n)) + d_Y(f(t_n), f(x')).
\end{align*} 

As  $d_Y(f(x),f(s_n))$, $d_Y(f(t_n), f(x'))$ tends to zero and $d_X(s_n,t_n)$ tends to $d_X(x,x')$ we have the result.

\subsubsection*{Extend $f$:}
Let $U$ be the interior of $\underline{X}$.  Clearly $\mu(U)\leq \mu(\underline{X}) = 0$. Extend $f$ to $X\setminus U$
as follows.  Given any point in $x\in X\setminus U$ there exists a sequence of points $x_n\in X\setminus \underline{X}$
such that $x_n$ converges to $x$.  Define $f(x) = \lim_{n\to \infty} f(x_n)$.  This is well defined as $f$ is an 
isometry on $X\setminus \underline{X}$.  $U$ is an open set.  So, $X\setminus U$ is closed.  Let $y_n\in f(X\setminus U)$
and $y_n$ converges to $y$.  If $y_n = f(x_n)$, $f$ is an isometry and $y_n$ converges imply that $x_n$ converges to some
point $x$.  This $x\in X\setminus U$ as $X\setminus U$ is closed.  So, $y\in f(X\setminus U)$.  Thus, $f(X\setminus U)$ 
is closed.  Let $V = Y\setminus (f(X\setminus U))$.  Then $V$ is open.

\subsubsection*{The map $f$ preserves measure:}
Given $\varepsilon > 0 $ there exists $N(\varepsilon)\in \mathbb{N}$ such that for all $n\geq N(\varepsilon)$, 
$\left(f^{-1}(E)\right) \subset f_n^{-1}(E^{\varepsilon})$.  Thus, for all $n\geq N(\varepsilon)$,  
$$\mu\left(f^{-1}(E)\right) \leq \mu\left(f_n^{-1}\left(E^{\varepsilon}\right)\right) \leq \nu\left(E^{\varepsilon + \frac{2}{n^2}}\right) + \frac{2}{n^2}.$$
Let $C$ be a closed set in $Y$.  We have for all $\varepsilon > 0$, $\mu\left(f^{-1}(C)\right) \leq \nu\left(C^{\varepsilon + \frac{2}{n^2}}\right) + \frac{2}{n^2}$.
Fix an $\varepsilon$ and choose $n$ so large that $\frac{2}{n^2} < \varepsilon$.  We have $\mu\left(f^{-1}(C)\right) \leq \nu\left(C^{2\varepsilon}\right) + \varepsilon$.
As $C$ is a closed set, given any open set $U$ containing $C$, there exists a number $\varepsilon$ such that $C^{\varepsilon} \subset U$.
Using the fact that $\nu(C) = \inf \left\lbrace \nu(U) : \text{$U$ is an open set containing $C$} \right\rbrace = \inf \left\lbrace \nu(C^{\varepsilon}): \varepsilon >0\right\rbrace$
we have, $\mu\left(f^{-1}(C)\right) \leq \nu(C)$.  
$\mu, \nu$ are regular measures and $f$ satisfies $\mu\left(f^{-1}(C)\right) \leq \nu(C)$ for all closed sets imply that $\mu\left(f^{-1}(E)\right) \leq \nu(E)$ for all sets.

\begin{lemma}
$\nu(f(X\setminus U)) = \nu(Y)$
\end{lemma}

\begin{proof}
\begin{align*}
\nu(Y) &\leq \nu\left(f_n \left(X\setminus X^n\right)^{\frac{2}{n^2}}\right) + \frac{2}{n^2}\\
&\leq \mu \left( f_n^{-1} \left(f_n \left(X\setminus X^n\right)^{\frac{4}{n^2}}\right)\right) + \frac{4}{n^2} \left\lbrace \text{by Lemma \ref{controlonmeas}} \right\rbrace
\end{align*}

\begin{lemma}
$f_n^{-1} \left(f_n \left(X\setminus X^n\right)^{\frac{4}{n^2}}\right) \subset \left( X\setminus X^n \right)^{\frac{4}{n^2}}$
\end{lemma}

\begin{proof}
\begin{align*}
x\in f_n^{-1} \left(f_n \left(X\setminus X^n\right)^{\frac{4}{n^2}}\right) &\implies f_n(x)\in f_n \left(X\setminus X^n\right)^{\frac{4}{n^2}}\\
&\implies \exists y\in X\setminus X^n \ s.t \ d_Y(f_n(x),f_n(y))<\frac{4}{n^2} < n^2 - \frac{2}{n^2} (\forall n\geq 2)
\end{align*}
Thus, $d_X(x,y) = d_Y(f_n(x),f_n(y)) < \frac{4}{n^2}$, that is, $x\in \left( X\setminus X^n \right)^{\frac{4}{n^2}}$.  
As $x\in f_n^{-1} \left(f_n \left(X\setminus X^n\right)^{\frac{4}{n^2}}\right)$ was arbitrary, 
$f_n^{-1} \left(f_n \left(X\setminus X^n\right)^{\frac{4}{n^2}}\right) \subset \left( X\setminus X^n \right)^{\frac{4}{n^2}}$.
\end{proof} 

So,
\begin{align*}
\nu(Y) &\leq \mu \left( f_n^{-1} \left(f_n \left(X\setminus X^n\right)^{\frac{4}{n^2}}\right)\right) + \frac{4}{n^2}\\
&\leq \mu\left(\left(X\setminus X^n\right)^{\frac{4}{n^2}}\right) + \frac{4}{n^2}\\
&\leq \mu(X) + \frac{4}{n^2}
\end{align*}
As the inequality $\nu(Y)\leq \mu(X) + \frac{4}{n^2}$ holds for every $n$, $\nu(Y)\leq \mu(X)$.  On the other hand the inequality $\mu(f^{-1}(E)) \leq \nu(E)$ gives, 
$$\nu(Y) \leq \mu(X) = \mu(X\setminus U) \leq \mu(f^{-1}(f(X\setminus U))) \leq \nu(f(X\setminus U)).$$
Also, $\nu(f(X\setminus U)) \leq \nu(Y)$ as $f(X\setminus U) \subset Y$.  So, $\nu(f(X\setminus U)) = \nu(Y)$.     
\end{proof}

\begin{lemma}
If $E$ is any set in $Y$, then,
$$\mu\left(f^{-1}\left(E^{2\varepsilon}\right) \right) \leq \mu\left( f^{-1}\left((E\cap (Y\setminus V))^{2\varepsilon}\right) \right) + \mu \left( f^{-1}\left( (E\cap V)^{2\varepsilon}\right) \right).$$ 
\end{lemma}

\begin{proof}
Now, $\nu(E) \leq \mu \left(f_n^{-1}\left(E^{\frac{2}{n^2}}\right)\right) + \frac{2}{n^2}$ by Lemma \ref{controlonmeas}.
Given $\varepsilon>0$ there exists a $N(\varepsilon)\in \mathbb{N}$ such that for all $n\geq N(\varepsilon)$, 
$f_n^{-1}\left(E^{\frac{2}{n^2}}\right) \subset f^{-1}\left(E^{\frac{2}{n^2} + \varepsilon}\right)$.  Thus, 
\begin{align*}
\nu(E) \leq \mu\left(f_n^{-1}\left(E^{\frac{2}{n^2}}\right)\right) + \frac{2}{n^2} &\leq \mu\left(f^{-1}\left(E^{\frac{2}{n^2} + \varepsilon}\right)\right) + \frac{2}{n^2}.
\end{align*}
Note that
\begin{align*}
f^{-1}\left(E^{\frac{2}{n^2} + \varepsilon}\right) &= f^{-1}\left(((E\cap (Y\setminus V))\cup (E\cap V))^{\frac{2}{n^2} + \varepsilon}\right)\\
 &= f^{-1}\left((E\cap (Y\setminus V))^{\frac{2}{n^2} + \varepsilon} \cup (E\cap V)^{\frac{2}{n^2} + \varepsilon}\right)\\
&\subset f^{-1}\left((E\cap (Y\setminus V))^{\frac{2}{n^2} + \varepsilon}\right)\cup f^{-1}\left( (E\cap V)^{\frac{2}{n^2} + \varepsilon}\right)
\end{align*} 
Thus, $\mu\left(f^{-1}\left(E^{\frac{2}{n^2} + \varepsilon}\right) \right) \leq \mu\left( f^{-1}\left((E\cap (Y\setminus V))^{\frac{2}{n^2} + \varepsilon}\right) \right) + \mu \left( f^{-1}\left( (E\cap V)^{\frac{2}{n^2} + \varepsilon}\right) \right)$.
Fix an $\varepsilon$ and choose $n$ so large that $\frac{2}{n^2} < \varepsilon$.  Then, 
$$\mu\left(f^{-1}\left(E^{2\varepsilon}\right) \right) \leq \mu\left( f^{-1}\left((E\cap (Y\setminus V))^{2\varepsilon}\right) \right) + \mu \left( f^{-1}\left( (E\cap V)^{2\varepsilon}\right) \right).$$
\end{proof}

Let $C$ be a closed set in $Y$.  As $C$ is a closed set and $V$ is an open set, $C\cap V$ is closed.  So, for 
sufficiently small $\varepsilon$, $(C\cap V)^{2\varepsilon} \subset V$.  So, 
$\mu \left( f^{-1}\left( (C\cap V)^{2\varepsilon}\right) \right) = 0$, that is, 
$\mu\left(f^{-1}\left(C^{2\varepsilon}\right) \right) \leq \mu\left( f^{-1}\left((C\cap (Y\setminus V))^{2\varepsilon}\right) \right)$.
Observe that,

\begin{lemma}
$f^{-1}\left((C\cap (Y\setminus V))^{2\varepsilon}\right) \subset \left( f^{-1}(C)\right)^{2\varepsilon}$
\end{lemma}

\begin{proof}
\begin{align*}
x\in f^{-1}\left((C\cap (Y\setminus V))^{2\varepsilon}\right) &\implies f(x)\in (C\cap (Y\setminus V))^{2\varepsilon} \\
&\implies f(x)\in (C\cap (f(X\setminus U)))^{2\varepsilon}\\
&\implies \exists y\in X\setminus U \ s.t \ f(y)\in C \text{ and } D(f(x),f(y)) < 2\varepsilon\\
&\implies x\in \left( f^{-1}(C)\right)^{2\varepsilon} \lbrace\text{$f$ is an isometry} \rbrace.    
\end{align*}
As $x\in f^{-1}\left((C\cap (Y\setminus V))^{2\varepsilon}\right)$ was arbitrary, $f^{-1}\left((C\cap (Y\setminus V))^{2\varepsilon}\right) \subset \left( f^{-1}(C)\right)^{2\varepsilon}$.
\end{proof}
  So, 
\begin{align*}
\nu(E) &\leq \mu\left(f^{-1}\left(C^{2\varepsilon}\right) \right) + 2\varepsilon\\
 &\leq \mu\left( f^{-1}\left((C\cap (Y\setminus V))^{2\varepsilon}\right) \right) + 2\varepsilon\\
&\leq \mu\left( \left(f^{-1}\left(C\right)\right)^{2\varepsilon} \right) + 2\varepsilon
\end{align*}
Note that $C$ is closed implies that $f^{-1}(C)$ is also closed.  Using the fact that 
$\mu(f^{-1}(C)) = \inf \left\lbrace \mu(U) : \text{$U$ is an open set containing 
$f^{-1}(C)$} \right\rbrace = \inf \left\lbrace\mu\left(\left(f^{-1}(C)\right)^{\varepsilon}\right): \varepsilon >0\right\rbrace$ we have, 
$$\mu(f^{-1}(C)) \geq \nu(C).$$ 

Hence, $\mu(f^{-1}(C)) = \nu(C)$ for every closed set $C$.  As $\mu, \nu$ are regular measures and $f$ preserves the measure of
every closed set, $f$ is measure preserving. 

Thus, the function $f:X\to Y$ is a measure preserving isometry.  So, we have completed the proof of Theorem \ref{positivityofrho}.    
\end{proof}

\subsection{Limit map between limit spaces}
\begin{prop}
Let $(X,d_X,\mu),(Y,d_{Y},\nu)$ be 
complete separable distance measure spaces.  Let $(X^n,d_{X^n},\mu^n), (Y^n,d_{Y^n},\nu^n)$ be such that, 
$d_{\rho}(X^n, X)$ converges to $0$ and $d_{\rho}(Y^n,Y)$ converges to $0$.  Given measure preserving isometries
$f^n: X^n \to Y^n$ we can construct a measure preserving isometry $f: X\setminus \underline{X} \to Y$ where 
$\underline{X}\subset X$ is a set of measure zero.
\label{limmapbnlimsp}
\end{prop}

\begin{proof}
This proof is very similar to the proof of Proposition \ref{positivityofrho}.  Again, the proof of this theorem will take the whole of this subsection.  To summarize, we construct maps 
$\varphi^n : \left(Y^n \setminus \widetilde{Y^n}\right) \to Y$ and $\psi^n:\left(X \setminus \widehat{X}^n\right) \to X^n$ as in Theorem \ref{quasiisomgivdistsmall} and take the limit of $\varphi^n \circ f^n \circ \psi^n (x)$ as $n$ tends to infinity.  We will prove that such a limit exists, it is a measure preserving isometry from $X \setminus \underline{X}$ (for some set $\underline{X}$ we define) to $Y$, and $\underline{X}$ has measure zero.

As $d_{\rho}(Y^n,Y)$ converges to $0$,
without loss of generality, we can assume that $\delta_2^n = d_{\rho}(Y^n,Y) < \frac{1}{n^2}$.  
As $d_{\rho}(X^n, X)$ converges to $0$, without loss of generality, we can assume that 
$\delta_1^n = d_{\rho}(X^n,X) < \frac{\frac{1}{n^2} - \delta_2^n}{4} $.  First construct maps 
$\varphi^n : \left(Y^n \setminus \widetilde{Y^n}\right) \to Y$ and $\psi^n:\left(X \setminus \widehat{X}^n\right) \to X^n$, 
as in Theorem \ref{quasiisomgivdistsmall}. Let $X(k) = \cup_{i=k}^{\infty} \widehat{X}^i$, $Y(k) = \cup_{i=k}^{\infty} (\psi^n)^{-1}\left(\left(f^n\right)^{-1}\left( \widetilde{Y^n}\right)\right)$
and $\underline{X} = \left[\cap_{n=1}^{\infty} X(k)\right] \cup \left[\cap_{n=1}^{\infty} Y(k)\right]$. 


\begin{lemma}
$\mu(\underline{X}) = 0$.
\end{lemma}

\begin{proof}
Note that, $\mu(X(k)) \leq \sum_{i=k}^{\infty} \frac{1}{i^2}$ which goes to zero as $k$ goes to infinity.  Thus, the measure 
$\mu\left(\cap_{k=1}^{\infty} X(k)\right) = 0$.  On the other hand, the following computation shows that the measures 
$\mu(Y(k))$ goes to zero as $k$ goes to infinity.
\begin{align*}
\mu(Y(k)) &\leq \sum_{i=k}^{\infty} \mu\left((\psi^k)^{-1}\left((f^k)^{-1}\left( \widetilde{Y^k}\right)\right)\right)\\
&\leq \sum_{i=k}^{\infty}\mu^k \left( \left((f^k)^{-1}\left( \widetilde{Y^k}\right)\right)^{\frac{\frac{1}{k^2} - \delta_2^k}{4}} \right) + \frac{\frac{1}{k^2} - \delta_2^k}{4} \ \ \ \lbrace Theorem \ref{quasiisomgivdistsmall}\rbrace\\
&= \sum_{i=k}^{\infty}\mu^k \left( (f^k)^{-1}\left(\left( \widetilde{Y^k}\right)^{\frac{\frac{1}{k^2} - \delta_2^k}{4}}\right) \right) + \frac{\frac{1}{k^2} - \delta_2^k}{4} \ \ \ \lbrace\text{$f^k$ is an isometry}\rbrace\\
&= \sum_{i=k}^{\infty}\nu^k\left(\left( \widetilde{Y^k}\right)^{\frac{\frac{1}{k^2} - \delta_2^k}{4}}\right) + \frac{\frac{1}{k^2} - \delta_2^k}{4} \ \ \ \lbrace\text{$f$ is measure preserving}\rbrace\\
&\leq \sum_{i=k}^{\infty}\frac{\frac{1}{k^2}+ \delta_2^k}{2} + \frac{\frac{1}{k^2} - \delta_2^k}{4} \ \ \ \lbrace Lemma \ref{trickforlimmapbnlimsp}\rbrace\\
&= \sum_{i=k}^{\infty}\frac{\frac{3}{k^2}+ \delta_2^k}{4}\\
&\leq \sum_{i=k}^{\infty}\frac{1}{k^2}.
\end{align*}
goes to zero as $k$ goes to infinity.  Thus $\mu\left(\cap_{k=1}^{\infty} Y(k)\right) = 0$.
Hence $\mu(\underline{X}) = 0$. 
\end{proof}

\begin{lemma}
Given a point $x \in (X \setminus \underline{X})$, the infinite sequence
$\varphi^n \circ f^n \circ \psi^n (x)$ has a convergent subsequence. 
\end{lemma}

\begin{proof}
Assume the contrary.  Then there exists an $\varepsilon_1$ such that 
$$d_Y\left(\varphi^n \circ f^n \circ \psi^n (x),\varphi^m \circ f^m \circ \psi^m (x)\right) \geq \varepsilon_1$$
Choose $\varepsilon < \min \left( \left\lbrace\left(\frac{1}{\delta_1^n} - 2\delta_1^n\right),\left(\frac{1}{\delta_2^n} - 2\delta_1^n - 2\delta_2^n\right),\left(\varepsilon_1 - 4\delta_1^n - 4\delta_2^n\right) : n\in \mathbb{N} \right\rbrace \cap \left\lbrace x\in \mathbb{R} : x >0 \right\rbrace \right)$.
As $n$ goes to infinity, $\delta_i^n$ goes to zero and the terms $\left(\frac{1}{\delta_1^n} - 2\delta_1^n\right)$,$\left(\frac{1}{\delta_2^n} - 2\delta_1^n - 2\delta_2^n\right)$,$\left(\varepsilon_1 - 4\delta_1^n - 4\delta_2^n\right)$
becomes bigger and bigger.  Hence the minimum exists.  By using Theorem \ref{quasiisomgivdistsmall} twice, we have 
for all $n$ such that all the three terms 
$\left(\frac{1}{\delta_1^n} - 2\delta_1^n\right)$,$\left(\frac{1}{\delta_2^n} - 2\delta_1^n - 2\delta_2^n\right)$, $\left(\varepsilon_1 - 4\delta_1^n - 4\delta_2^n\right)$
are positive (i.e for all but finitely many $n$), 
$$\varphi^n \circ f^n \circ \psi^n(B(x,\varepsilon)) \subset B(\varphi^n \circ f^n \circ \psi^n(x),\varepsilon + 2 \delta_1^n + 2\delta_2^n).$$

\noindent Thus,
$$(\varphi^n \circ f^n \circ \psi^n(B(x,\varepsilon)))^{2\delta_1^n + 2 \delta_2^n} \subset (B(\varphi^n \circ f^n \circ \psi^n(x),\varepsilon + 2 \delta_1^n + 2\delta_2^n))^{2\delta_1^n + 2 \delta_2^n}.$$
So,
\begin{align*}
\nu((B(\varphi^n \circ f^n \circ \psi^n(x),\varepsilon + 2 \delta_1^n + 2\delta_2^n))^{2\delta_1^n + 2 \delta_2^n}) &\geq \nu((\varphi^n \circ f^n \circ \psi^n(B(x,\varepsilon)))^{2\delta_1^n + 2 \delta_2^n})\\
&\geq \mu(B(x,\varepsilon)) - 2\delta_1^n - 2\delta_2^n \\
& \ \ \ \ \ \ \ \ (\text{apply Theorem \ref{quasiisomgivdistsmall} twice}).
\end{align*}
Furthermore, $(B(\varphi^n \circ f^n \circ \psi^n(x),\varepsilon + 2 \delta_1^n + 2\delta_2^n))^{2\delta_1^n + 2 \delta_2^n}$
are disjoint.  This contradicts with the assumption that $Y$ has finite measure.   
\end{proof}

Choose a countable dense subset $S = \lbrace x_1,x_2,...\rbrace \subset (X \setminus \underline{X})$.  Choose a subsequence
$\varphi^{n_{1,k}} \circ f^{n_{1,k}} \circ \psi^{n_{1,k}}$ of $\varphi^n \circ f^n \circ \psi^n$ such that
$\varphi^{n_{1,k}} \circ f^{n_{1,k}} \circ \psi^{n_{1,k}}(x_1)$ converges.  We choose a further subsequence 
$\varphi^{n_{2,k}} \circ f^{n_{2,k}} \circ \psi^{n_{2,k}}$ of 
$\varphi^{n_{1,k}} \circ f^{n_{1,k}} \circ \psi^{n_{1,k}}$
such that $\varphi^{n_{2,k}} \circ f^{n_{2,k}} \circ \psi^{n_{2,k}}(x_2)$ converges. Observe that 
$\varphi^{n_{2,k}} \circ f^{n_{2,k}} \circ \psi^{n_{2,k}}(x_1)$ continues to converge.
Iterating this process, we obtain subsequences $\varphi^{n_{j,k}} \circ f^{n_{j,k}} \circ \psi^{n_{j,k}}$ so that
$\varphi^{n_{j,k}} \circ f^{n_{j,k}} \circ \psi^{n_{j,k}}(x_l)$ converges for $l\leq j$. It follows that for the
diagonal sequence $\varphi^{n_{k,k}} \circ f^{n_{k,k}} \circ \psi^{n_{k,k}}$ we have the corresponding convergence for
all points in $S$.  Replace $\varphi^n \circ f^n \circ \psi^n$  by this diagonal subsequence.  Define 
$$f(x_i) = \lim_{n\to \infty} \varphi^n \circ f^n \circ \psi^n (x_i)$$

\begin{lemma}
$d_Y(f(x_i),f(x_j)) \leq d_X(x_i,x_j).$
\label{fdisdec}
\end{lemma}  

\begin{proof}
Given $x_i$ and $x_j$, if $d_Y(f(x_i),f(x_j)) = \infty = d_X(x_i,x_j)$, then it is obvious.  Otherwise, choose $n$ so
large that $\min\lbrace d_Y(f(x_i),f(x_j)),d_X(x_i,x_j)\rbrace < \min\left\lbrace\frac{1}{\delta_1^n},\frac{1}{\delta_2^n}\right\rbrace - 2\delta_1^n - 2\delta_2^n$.
Then,
\begin{align*}
d_Y(\varphi^n \circ f^n \circ \psi^n(x_i), \varphi^n \circ f^n \circ \psi^n(x_j)) &\leq  d_{Y^n}(f^n \circ \psi^n(x_i), f^n \circ \psi^n(x_j)) + 2 \delta_2^n \\
&=   d_{X^n}(\psi^n(x_i), \psi^n(x_j)) + 2\delta_2^n\\
&\leq  d_X(x_i, x_j) + 2\delta_1^n + 2\delta_2^n.
\end{align*}
Thus, 
\begin{align*}
d_Y(f(x_i),f(x_j)) &= \lim_{n \to \infty} d_Y(\varphi^n \circ f^n \circ \psi^n(x_i), \varphi^n \circ f^n \circ \psi^n(x_j)) \\
&\leq \lim_{n \to \infty} d_X(x_i, x_j) + 2\delta_1^n + 2\delta_2^n.
\end{align*}
As $\delta_i^n$ tends to zero as $n$ tends to infinity we have the result.
\end{proof}

\begin{lemma}
$d_Y(f(x_i),f(x_j)) \geq d_X(x_i,x_j).$
\label{fdisinc}
\end{lemma}  

\begin{proof}
Given $x_i$ and $x_j$, if $d_Y(f(x_i),f(x_j)) = \infty = d_X(x_i,x_j)$, then it is obvious.  Otherwise, choose $n$ so
large that $\min\lbrace d_Y(f(x_i),f(x_j)),d_X(x_i,x_j)\rbrace < \min\left\lbrace\frac{1}{\delta_1^n},\frac{1}{\delta_2^n}\right\rbrace - 2\delta_1^n - 2\delta_2^n$.
Then,

\begin{align*}
d_Y(\varphi^n \circ f^n \circ \psi^n(x_i), \varphi^n \circ f^n \circ \psi^n(x_j)) &\geq  d_{Y^n}(f^n \circ \psi^n(x_i), f^n \circ \psi^n(x_j)) - 2 \delta_2^n \\
&=  d_{X^n}(\psi^n(x_i), \psi^n(x_j)) - 2\delta_2^n\\
&\geq d_X(x_i, x_j) - 2\delta_1^n - 2\delta_2^n.
\end{align*}
Thus,

\begin{align*}
d_Y(f(x_i),f(x_j)) &= \lim_{n \to \infty} d_Y(\varphi^n \circ f^n \circ \psi^n(x_i), \varphi^n \circ f^n \circ \psi^n(x_j)) \\
&\geq \lim_{n \to \infty} d_X(x_i, x_j) - 2\delta_1^n - 2\delta_2^n.
\end{align*}
As $\delta_i^n$ tends to zero as $n$ tends to infinity we have the result.
\end{proof}

From Lemma \ref{fdisdec} and Lemma \ref{fdisinc} we have

\begin{lemma}
The map $f$ is an isometry on $S$.
\label{fisomonS}
\end{lemma}

As $S$ is dense in $X \setminus \underline{X}$ given any point $X \setminus \underline{X}$ there exist a sequence 
$s_n \in S$ such that $d(s_n,x)$ tends to $0$.  Define
$f(x)=\lim_{n\to \infty}f(s_n)$ where $s_n$ tends to $x$ as $n$ tends
to $\infty$.  
\subsubsection*{The map $f$ is well defined:} Let $\langle s_n \rangle_n$ and $\langle t_n \rangle_n$ be two
sequence which converge to $x$.  Then the sequence $s_1,t_1,s_2,t_2,...$ is
Cauchy and so is the sequence $f(s_1),f(t_1),
f(s_2),f(t_2),...$ by Lemma \ref{fisomonS}.  Thus it converges
and has the same limit as $\langle f(s_n) \rangle_n$ and $\langle f(t_n) \rangle_n$.  Hence
$\lim_{n\to \infty} f(s_n) = \lim_{n\to \infty} f(t_n)$.

\subsubsection*{The map $f$ is an isometry:}Given two points $x,x' \in X \setminus \underline{X}$ consider sequences 
$\langle s_n \rangle_n \subset S$ and $\langle t_n \rangle_n \subset S$ such that $s_n$ converges to $x$ and $t_n$ 
converges to $x'$.  Then,
\begin{align*}
d_Y(f(x),f(x')) &\leq d_Y(f(x),f(s_n)) + d_Y(f(s_n),f(t_n)) + d_Y(f(t_n), f(x'))\\
&\leq d_Y(f(x),f(s_n)) + d_X(s_n,t_n) + d_Y(f(t_n), f(x'))
\end{align*}
Similarly, $d_Y(f(s_n),f(t_n)) \leq d_Y(f(x),f(s_n)) + d_Y(f(x),f(x')) + d_Y(f(t_n), f(x'))$ implies that
\begin{align*}
d_Y(f(x),f(x')) &\geq d_Y(f(s_n),f(t_n)) - (d_Y(f(x),f(s_n)) + d_Y(f(t_n), f(x'))\\
&= d_X(s_n,t_n) - (d_Y(f(x),f(s_n)) + d_Y(f(t_n), f(x')).
\end{align*} 

As  $d_Y(f(x),f(s_n))$, $d_Y(f(t_n), f(x'))$ tends to zero and $d_X(s_n,t_n)$ tends to $d_X(x,x')$ we have the result.

\subsubsection*{The map $f$ preserves measure:}

\begin{lemma}
If $A \subset Y$, $\mu\left(f^{-1}(A)\right) \leq \nu\left(A^{\varepsilon + 2\delta_1^n + 2\delta_2^n}\right) + 2\delta_1^n + 2\delta_2^n$. 
\end{lemma}

\begin{proof}
If $E \subset Y$
\begin{align*}
\mu\left((\psi^n)^{-1}\circ (f^n)^{-1}\circ (\varphi^n)^{-1}(E)\right) &\leq \mu^n\left(((\psi^n)^{-1}\circ (f^n)^{-1}(E))^{2\delta_2^n}\right) + 2\delta_2^n\\
&\leq \mu^n\left((f^n)^{-1}\left(\left((\varphi^n)^{-1}(E)\right)^{2\delta_2^n}\right)\right) + 2\delta_2^n \\
&\ \ \ \ \ (\text{as $f^n$ is an isometry})\\
&\leq \nu^n \left( \left( (\varphi^n)^{-1}(E) \right)^{2\delta_2^n}\right) + 2\delta_2^n\\
& \ \ \ \ \ (\text{as $f^n$ is measure preserving})\\
\end{align*}

Observe that,

\begin{lemma}
$\left((\varphi^n)^{-1}(E)\right)^{2\delta_2^n} \subset \left(\varphi^n\right)^{-1}\left(E^{2\delta_2^n}\right)$.
\end{lemma}

\begin{proof}
$x\in \left((\varphi^n)^{-1}(E)\right)^{2\delta_2^n}$ means that $d_{Y^n}(x,(\varphi^n)^{-1}(E))< 2\delta_2^n$,
that is, there exists $y\in X$ such that $\varphi^n(y)\in E$ and $d_{Y^n}(x,y)< 2\delta_2^n$.  If $n$ is large enough
so that $\delta_2^n <\frac{1}{2}$ then, $d_{Y^n}(x,y)< 2\delta_2^n < \frac{1}{\delta_2^n} - 2\delta_2^n$.  So, 
$d_Y(\varphi^n(x),\varphi^n(y))  < d_{Y^n}(x,y) + 2\delta_2^n$.  Thus, 
$x\in \left(\varphi^n\right)^{-1}\left(E^{2\delta_2^n}\right)$.  As 
$x\in \left((\varphi^n)^{-1}(E)\right)^{2\delta_2^n}$ was arbitrary, 
$\left((\varphi^n)^{-1}(E)\right)^{2\delta_2^n} \subset \left(\varphi^n\right)^{-1}\left(E^{2\delta_2^n}\right)$.
\end{proof}
 
Thus,
\begin{align*}
\mu\left((\psi^n)^{-1}\circ (f^n)^{-1}\circ (\varphi^n)^{-1}(E)\right) 
&\leq \nu^n \left( \left( (\varphi^n)^{-1}(E) \right)^{2\delta_2^n}\right) + 2\delta_2^n\\
&\leq \nu^n\left(\left(\varphi^n\right)^{-1}\left(E^{2\delta_2^n}\right)\right) + 2\delta_2^n\\
&\leq \nu\left(E^{2\delta_1^n + 2\delta_2^n}\right) + 2\delta_1^n + 2\delta_2^n.
\end{align*}

Let $A \subset Y$ be arbitrary, given $\varepsilon >0$ there exists $N(\varepsilon)$ such that, $f^{-1}(A)$ is 
contained in $(\psi^n)^{-1}\circ (f^n)^{-1}\circ (\varphi^n)^{-1}\left(A^{\varepsilon}\right)$ for all 
$n \geq N(\varepsilon)$.  Thus,
$$\mu(f^{-1}(A)) \leq \mu\left( (\psi^n)^{-1}\circ (f^n)^{-1}\circ (\varphi^n)^{-1}\left(A^{\varepsilon}\right) \right) $$
Replacing $E$ by $A^\varepsilon$ in the previous computation we get
\begin{align*}
\mu(f^{-1}(A)) &\leq \mu\left( (\psi^n)^{-1}\circ (f^n)^{-1}\circ (\varphi^n)^{-1}\left(A^{\varepsilon}\right) \right)\\
&\leq \nu\left(A^{\varepsilon + 2\delta_1^n + 2\delta_2^n}\right) + 2\delta_1^n + 2\delta_2^n.
\end{align*}  
\end{proof}

\begin{lemma}
If $C$ be a closed set in $Y$, $\mu(f^{-1}(C)) \leq \nu(C)$.
\end{lemma}

\begin{proof}
Let $C$ be a closed set in $Y$.  We have for all $\varepsilon> 0$, 
$\mu(f^{-1}(C)) \leq \nu\left(C^{\varepsilon + 2\delta_1^n + 2\delta_2^n}\right) + 2\delta_1^n + 2\delta_2^n$.  Fix 
an $\varepsilon$ and choose $n$ so large that $2\delta_1^n + 2\delta_2^n < \varepsilon$.  We have 
$\mu(f^{-1}(C)) \leq \nu\left(C^{2\varepsilon}\right) + \varepsilon$.   As $C$ is a closed set, given any open set 
$U$ containing $C$, there exists a number $\varepsilon$ such that $C^{\varepsilon} \subset U$.  Using the fact that 
$\mu(C) = \inf \left\lbrace \mu(U) : \text{$U$ is an open set containing $C$} \right\rbrace = \inf \left\lbrace\mu(C^{\varepsilon}): \varepsilon >0\right\rbrace$
we have, $\mu(f^{-1}(C)) \leq \nu(C)$.
\end{proof}

Similarly,
\begin{lemma}
If $E \subset Y$, $\nu(E) \leq \mu \left(f^{-1}\left(E^{\varepsilon + 2\delta_1^n + 2\delta_2^n}\right)\right) + 2\delta_1^n + 2\delta_2^n$. 
\end{lemma}

\begin{proof}
If $E \subset Y$,
\begin{align*}
\nu(E) &\leq \nu^n\left( \left( \phi^n \right)^{-1} \left( E^{2\delta_2^n} \right) \right) + 2\delta_2^n\\
& \ \ \ \ \ (\text{by Lemma \ref{controlonmeas}})\\
&\leq \mu^n \left( \left( f^n \right)^{-1} \circ \left( \varphi^n \right)^{-1} \left( E^{2\delta_2^n} \right) \right) + 2\delta_2^n\\
& \ \ \ \ \ (\text{as $f$ is measure preserving})\\
&\leq \mu \left( \left( \psi^n \right)^{-1} \left( \left( \left( f^n \right)^{-1} \circ \left( \varphi^n \right)^{-1} \left( E^{2\delta_2^n} \right) \right)^{2\delta_1^n} \right) \right) + 2\delta_1^n + 2\delta_2^n\\
& \ \ \ \ \ (\text{by Lemma \ref{controlonmeas}})\\
&\leq \mu \left( \left( \psi^n \right)^{-1} \circ \left( f^n \right)^{-1} \left( \left( \left( \varphi^n \right)^{-1} \left( E^{2\delta_2^n} \right) \right)^{2\delta_1^n} \right) \right) + 2\delta_1^n + 2\delta_2^n\\
& \ \ \ \ \ (\text{as $f$ is an isometry})\\ 
\end{align*}  

Observe that,
\begin{lemma}
$\left(\left(\varphi^n\right)^{-1}\left(E^{2\delta_2^n}\right)\right)^{2\delta_1^n} \subset \left(\varphi^n\right)^{-1}\left(E^{2\delta_2^n + 2\delta_1^n}\right)$.
\end{lemma}

\begin{proof}
A point $x\in \left(\left(\varphi^n\right)^{-1}\left(E^{2\delta_2^n}\right)\right)^{2\delta_1^n}$ means that 
$d_{Y^n}\left(x,\left(\varphi^n\right)^{-1}\left(E^{2\delta_2^n}\right)\right)< 2\delta_1^n$, that is, there exists $y\in Y^n$ 
such that $\varphi^n(y)\in E^{2\delta_2^n}$ and $d_{Y^n}(x,y)< 2\delta_1^n$.  If $n$ is large enough so that 
$\delta_1^n <\frac{1}{2}$ then, $d_{Y^n}(x,y)< 2\delta_1^n < \frac{1}{\delta_1^n} - 2\delta_1^n$.  So, 
$d_Y(\varphi^n(x),\varphi^n(y))  < d_{Y^n}(x,y) + 2\delta_1^n$.  Thus, 
$x\in \left(\varphi^n\right)^{-1}\left(\left(E^{2\delta_2^n}\right)^{2\delta_1^n}\right) \subset \left(\varphi^n\right)^{-1}\left(E^{2\delta_2^n + 2\delta_1^n}\right) $.
As $x\in \left(\left(\varphi^n\right)^{-1}\left(E^{2\delta_2^n}\right)\right)^{2\delta_1^n}$ was arbitrary, we have 
$\left(\left(\varphi^n\right)^{-1}\left(E^{2\delta_2^n}\right)\right)^{2\delta_1^n} \subset \left(\varphi^n\right)^{-1}\left(E^{2\delta_2^n + 2\delta_1^n}\right)$.
\end{proof}   
Thus,
\begin{align*}
\nu(E) &\leq \mu \left( \left( \psi^n \right)^{-1} \circ \left( f^n \right)^{-1} \left( \left( \left( \varphi^n \right)^{-1} \left( E^{2\delta_2^n} \right) \right)^{2\delta_1^n} \right) \right) + 2\delta_1^n + 2\delta_2^n\\
&\leq \mu \left( \left( \psi^n \right)^{-1} \circ \left( f^n \right)^{-1} \circ \left( \varphi^n \right)^{-1} \left( E^{2\delta_1^n + 2\delta_2^n} \right) \right) + 2\delta_1^n + 2\delta_2^n.
\end{align*} 
Let $A \subset Y$ be arbitrary, given $\varepsilon >0$ there exists $N(\varepsilon)$ such that, 
$(\psi^n)^{-1}\circ (f^n)^{-1}\circ (\varphi^n)^{-1}\left( A \right)$ is contained in 
$f^{-1}\left(A^{\varepsilon}\right)$   for all $n \geq N(\varepsilon)$.  Thus,
$$\mu\left( (\psi^n)^{-1}\circ (f^n)^{-1}\circ (\varphi^n)^{-1}\left(A \right) \right) \leq \mu(f^{-1}(A^{\varepsilon})).$$
Using this inequality for $A = E^{2\delta_1^n + 2\delta_2^n}$ we have,
\begin{align*}
\nu(E) &\leq \mu \left( \left( \psi^n \right)^{-1} \circ \left( f^n \right)^{-1} \circ \left( \varphi^n \right)^{-1} \left( E^{2\delta_1^n + 2\delta_2^n} \right) \right) + 2\delta_1^n + 2\delta_2^n\\
&\leq \mu \left(f^{-1}\left(E^{\varepsilon + 2\delta_1^n + 2\delta_2^n}\right)\right) + 2\delta_1^n + 2\delta_2^n.
\end{align*}
\end{proof}

\begin{lemma}
If $C$ be a closed set in $Y$, $\mu(f^{-1}(C)) \geq \nu(C)$. 
\end{lemma}

\begin{proof}
Let $C$ be a closed set in $Y$.  We have for all $\varepsilon> 0$, $\nu(C) \leq \mu \left( f^{-1}\left(C^{\varepsilon + 2\delta_1^n + 2\delta_2^n} \right) \right) + 2\delta_1^n + 2\delta_2^n$. 
 Fix an $\varepsilon$ and choose $n$ so large that $2\delta_1^n + 2\delta_2^n < \varepsilon$.  
Then, $\nu(C) \leq \mu \left( f^{-1}\left(C^{2\varepsilon} \right) \right) + \varepsilon$.   As $C$ is a closed set, given
 any open set $U$ containing $C$, there exists a number $\varepsilon$ such that $C^{\varepsilon} \subset U$.  Using 
 the fact that $\mu(C) = \inf \left\lbrace \mu(U) : \text{$U$ is an open set containing $C$} \right\rbrace = \inf \left\lbrace\mu(C^{\varepsilon}): \varepsilon >0\right\rbrace$
 we have, $\mu\left(f^{-1}(C)\right) \geq \nu(C)$.
\end{proof}

Thus, for every closed set $C$, $\mu\left(f^{-1}(C)\right) = \nu(C)$.  As $\mu, \nu$ are regular measures and $f$ preserves the 
measure of every closed set, $f$ is measure preserving.  

Thus we have completed the proof of Theorem \ref{limmapbnlimsp} by showing that the set $\underline{X}$ has zero measure and $f:X\setminus \underline{X} \to Y$ is a measure preserving isometry. 

\end{proof}  

\subsection{Approximating a distance measure space by a finite distance measure
space}
In this subsection we will prove that any complete separable distance measure space
$(X,d,\mu)$, with finite measure, can be approximated by a finite distance
measure space, i.e., by $(Y,d',\nu)$ where $Y$ is a finite set.

\begin{theorem}
If $X$ is a complete separable distance (metric) space and $\mu$ a finite
measure on $X$. Then given $\varepsilon > 0$, there exist a compact set
$K_{\varepsilon}$ such that $\mu((K_{\varepsilon})^C)\leq \varepsilon$.
\label{cslem}
\end{theorem}
 
\begin{proof}
Let $\lbrace x_i\rbrace_{i\in \mathbb{N}}$ be a countable dense set.  Then for every $k\in
\mathbb{N}$,  $\cup_{i=1}^{\infty} B_{\frac{1}{k}}(x_i) = X$.  Thus for each $k
\in \mathbb{N}$ you can find $N(k)$ such that $U_k = \cup_{i=1}^{N(k)}
B_{\frac{1}{k}}(x_i)$ has measure at least $M- \frac{\varepsilon}{2^k}$ (where
$M= \mu(X)$).  The set $U_{\varepsilon} = \cap_{1}^{\infty}U_k$ is a totally bounded set as 
$U_{\varepsilon} \subset \cup_{i=1}^{N(k)} B_{\delta}(x_i)$ if, $\delta > \frac{1}{k}$.  Hence its
closure, $K_{\varepsilon}$, is compact.  Also, $\mu\left((K_{\varepsilon})^C\right) \leq \mu\left(\cup_{1}^{\infty}U_k^C\right) =
\sum_1^{\infty} \mu\left(U_k^C\right) \leq \sum_1^{\infty} \frac{\varepsilon}{2^k} =
\varepsilon $.      
\end{proof}    

\begin{theorem}
Given a complete separable distance (metric) measure space $(X,d,\mu)$ such that $\mu$ is finite,
we can find another measure  $\mu_{\varepsilon}$ such that $\mu_{\varepsilon}$
is supported on a finite set and $d_{\pi}(\mu, \mu_{\varepsilon}) \leq \varepsilon$.
\label{approxbymeaswhfinitesupp}
\end{theorem}

\begin{proof}
Using Theorem \ref{cslem} we get a compact set $K_{\varepsilon}$ such that
$\mu\left(K_{\varepsilon}^C\right)\leq \varepsilon$.  Consider the open cover
$\lbrace B(x,\varepsilon) \ | \ x\in K_{\varepsilon}\rbrace$ of $K_{\varepsilon}$.  This has a
finite sub cover as, $K_{\varepsilon}$ is compact.  Take the centres of the
corresponding balls to get a finite set $X_{\varepsilon}$.  Let $X_{\varepsilon}
= \lbrace x_1,...,x_n\rbrace$.  Define the atomic measure $\mu_\varepsilon$ inductively.
Define
$$\mu_\varepsilon(x_1)=\mu\left(B_{\varepsilon}(x_1)\right)$$  
Given $\mu_\varepsilon(x_i), \forall i<k$ define 
$$\mu_\varepsilon(x_{k})= \mu\left(B_{\varepsilon}(x_k)\setminus \cup_{i=1}^{k-1}
B_{\varepsilon}(x_i)\right)$$

\subsubsection*{1. $\mu_\varepsilon(E)\leq
\mu(E^{\varepsilon})+\varepsilon$}
$$\mu_\varepsilon(E) = \mu_\varepsilon(E\cap X_{\varepsilon}) \leq \mu(E^{\varepsilon})\leq
\mu(E^{\varepsilon})+\varepsilon$$

The first equality follows from the definition of
$\mu_{\varepsilon}$.  The second inequality follows because for each $x_i$, 
$\mu_{\varepsilon}(x_i) \leq \mu(\lbrace x_i\rbrace^{\varepsilon})$.   

\subsubsection*{2. $\mu(E)\leq \mu_\varepsilon(E^{\varepsilon}) +
\varepsilon$}

The set $E^{\varepsilon}$ contains all $x_i$ such that $B(x_i,\varepsilon) \cap E \neq \emptyset$. So, 
$\mu_\varepsilon(E^{\varepsilon}) \geq \mu(\cup_{i \ s.t \ B(x_i,\varepsilon) \cap E \neq \emptyset} B(x_i,\varepsilon))$.
As $K_{\varepsilon} \cap E \subset \cup_{i \ s.t \ B(x_i,\varepsilon) \cap E \neq \emptyset} B(x_i,\varepsilon)$, $\mu\left(\cup_{i \ s.t \ B(x_i,\varepsilon) \cap E \neq \emptyset} B(x_i,\varepsilon)\right) \geq \mu(K_{\varepsilon} \cap E))$ 
which is greater than or equal to $\mu(E) - \varepsilon$ as, $\mu((K_{\varepsilon})^C)\leq \varepsilon$.

Thus we have proved that $d_{\pi}(\mu_{\varepsilon},\mu) \leq \varepsilon$.  

\end{proof}

\begin{theorem}
Given a complete separable distance (metric) measure space $(X,d,\mu)$ such that $\mu$ is finite,
we can find another finite distance (metric) measure, i.e., has only finitely many points, such that, 
$(X_{\varepsilon},d',\mu_\varepsilon)$ such that 
$$d_{\rho}((X,d,\mu), (X_{\varepsilon},d',\mu_\varepsilon)) \leq \varepsilon.$$
\label{approxdmsbyfinitedms}
\end{theorem}

\begin{proof}
Choose $X_{\varepsilon}$ and $\mu_{\varepsilon}$ as in the proof of theorem
\ref{approxbymeaswhfinitesupp} and let $d' = d$.  Then we have that  
$$d_{\rho}((X,d,\mu), (X_{\varepsilon},d',\mu_\varepsilon)) \leq \varepsilon.$$
\end{proof}

\subsection{Space of finite distance measure spaces with bounded total measure
and bounded number of points is compact}
  
\begin{theorem}
Let $\mathcal{F}(N,M)$ be the space of finite distance measure spaces with total measure less than or equal to $M$ 
and number of points less than or equal to $N$.  Then, $\mathcal{F}(N,M)$ is compact.
\label{spacefdmscompact}
\end{theorem}

\begin{proof}
We will show compactness by showing that any sequence has a convergent subsequence.  Let 
$\langle (X_n,d_n,\mu_n)\rangle_n$ be a sequence in $\mathcal{F}(N,M)$.  Then, there are infinitely many 
$X_n$ with same cardinality.  So, without loss of generality we can assume that all $X_n$ have same cardinality.  
So, let $X_n = \lbrace x_1,...,x_n\rbrace = X$.  The sequence $\langle d_n(x_i,x_j) \rangle_n$ is contained in $[0,\infty]$ hence
, has a convergent subsequence.  Choose a subsequence, $(X_{n_k},d_{n_k},\mu_{n_k})$, of the $(X_n,d_n,\mu_n)$ such 
that, $d_{n_k}(x_i,x_j)$ converges as $k$ tends to infinity, for all $i,j$.  

Define $Y$ to be the set $X/\sim$, where $\sim$ is the equivalence relation defined as, $x\sim y$ if $d_{n_k}(x,y)$ 
converge to zero.  Define a distance $d$ on $Y$ as $d([x],[y]) = \lim_{k\to \infty} d_{n_k}(x,y)$. Finally define a 
measure $\mu$ on $Y$ as $\mu([x]) = \lim_{k\to \infty} \left( \sum_{y\sim x} \mu_{n_k}(x) \right)$.

It is clear that $(X_{n_k},d_{n_k},\mu_{n_k})$ converges to $(Y,d,\mu))$ as $k$ tends to infinity.  As the sequence
$(X_n,d_n,\mu_n)$ was arbitrary, $\mathcal{F}(N,M)$ is compact.        
\end{proof}

\subsection{Completeness of the space of distance measure spaces}

\begin{lemma}
Let $(X_n,d_n,\mu_n)$ be a Cauchy sequence.  For all $m$, let $(X_n^m,d_n^m,\mu_n^m)$
be a finite distance measure space such that
$d_{\rho}((X_n^m,d_n^m,\mu_n^m),(X_n,d_n,\mu_n)) \leq \frac{1}{2^m}$.  Then
$(X_n^m,d_n^m,\mu_n^m)$ has uniformly bounded total measure for all $m$.
\label{cauchyboundedmeasure}
\end{lemma}

\begin{proof}
The proof of this theorem will take this entire subsection.  To summarize, we will first approximate each $X_n$ by a finite space.  Giving us a Cauchy sequence of finite distance measure spaces.  Take its limit.  We will do this for better and better approximations (closer than $\frac{1}{2^m}$) and will take the limit of the sequence of limiting finite spaces, as $m$ tends to infinity.  Finally, we will complete this space to obtain the space $X$ and show that $X_n$ converge to $X$.    

As the sequence $\langle (X_n,d_n,\mu_n) \rangle_n$ is Cauchy, so is the sequence 
$\langle (X_n^m,d_n^m,\mu_n^m) \rangle_n$.  Hence, given $\varepsilon > 0$, there exists $N(\varepsilon)$, such that
$d_{\rho}((X_n^m,d_n^m,\mu_n^m), (X_{N(\varepsilon)}^m,d_{N(\varepsilon)}^m,\mu_{N(\varepsilon)}^m)) < \varepsilon$ 
for all $n \geq N(\varepsilon)$.  Then, by Lemma \ref{total-measures-close-when-the-metric-measure-spaces-are}, 
$\mu_n^m(X_n^m) \leq \mu_{N(\varepsilon)}^m(X_{N(\varepsilon)}^m) +  \varepsilon$
\end{proof}

\begin{theorem}
Let $(X,d)$ be a separable distance space and $\mathcal{P}(X)$ the space of all borel measures on $X$.  $(X,d)$ is complete if and only if
$(\mathcal{P}(X),d_{\pi})$ is complete.
\end{theorem}

\begin{proof}
Apply the corresponding theorem for metric spaces to each equivalence class of
$\sim$.  See \cite{parthasarathy} or \cite{gaans} for more details.
\end{proof}

Let $\mathfrak{X}$ the quotient of the set of all complete separable distance measure spaces with
the equivalence relation $X\sim Y$ if and only if $d_{\rho}(X,Y)=0$.

\begin{theorem}
$\mathfrak{X}$ is complete.
\label{completenessunderrho} 
\end{theorem}

\begin{proof}
Consider a Cauchy sequence $(X_n,d_n,\mu_n)$.  Using Theorem
\ref{approxdmsbyfinitedms} for each $n$, we obtain a sequence
$(X_n^m,d_n^m,\mu_n^m)$ such that $d_{\rho}((X_n^m,d_n^m,\mu_n^m),(X_n,d_n,\mu_n))
\leq \frac{1}{2^m}$. Without loss of generality we can assume that $X_n^m \supset X_n^{m-1}$ as otherwise we can 
replace $X_n^m$ by $X_n^m \cup X_n^{m-1}$ and then define  $\mu_n^{k+1}$ as in the proof of Theorem \ref{approxbymeaswhfinitesupp}
on this set. By Lemma \ref{cauchyboundedmeasure} for each $m$ we have that, $(X_n^m,d_n^m,\mu_n^m)$
has bounded total measure.  We will prove that we can choose $X_n^m$, such that there is a uniform bound (not dependent on $n$) on the cardinality as well.  


\begin{lemma}
For each $m$, we can choose $X_n^m$, such that there is an upper bound on the 
cardinality of $X_n^m$  which does not depend on $n$.
\end{lemma}

\begin{proof}
Choose an $N(m)$ such that $d_{\rho}((X_n,d_n,\mu_n),(X_k,d_k,\mu_k)) < \frac{1}{2^{m+3}}$, for all $n \geq N(m)$.  
Thus, there exists a metric space $(Z,d)$ and $L$-isometric $\varepsilon$-embeddings $f_1: (X_{N(m)},d_{N(m)}) \to (Z,d)$ 
and $f_2: (X_k,d_k) \to (Z,d)$ such that the push-forward measures $\nu_{N(m)} = (f_1)_*(\mu_{N(m)})$ and 
$\nu_k = (f_2)_*(\mu_k)$ satisfy $d_{\pi}(\nu_{N(m)},\nu_k) + \frac{1}{L} + \varepsilon \leq \frac{1}{2^{m+3}} $.
So, $d_{\pi}(\nu_{N(m)},\nu_k) \leq \frac{1}{2^{m+3}} $.

Let $X_{N(m)}^{m+2} = \lbrace x_1,...,x_M\rbrace$.  For each $i$ such that 
$f_2(X_k) \cap B\left(f_1(x_i), \frac{1}{2^{m+2}}\right) \neq \emptyset$, let $z_i \in X_k$ be an arbitrary point such that
$f_2(z_i) \in B\left(f_1(x_i), \frac{1}{2^{m+2}}\right)$ and  define 
$$\widehat{X_k^m} = \left\lbrace z_i \ | \ i \ s.t \ f_2(X_k) \cap B\left(f_1(x_i), \frac{1}{2^{m+2}}\right) \neq \emptyset\right\rbrace.$$


\begin{lemma}
$\mu_k\left(\left(\widehat{X_k^m}\right)^\frac{1}{2^m}\right) \geq \mu_k(X_k) - \frac{1}{2^m}$.
\end{lemma}

\begin{proof}
The proof is a series of simple computations.  First observe that,
\begin{lemma}
$\left(f_2 \left(\widehat{X_k^m}\right)\right)^{\frac{1}{2^m}} \cap f_2(X_k) \subset f_2\left(\left(\widehat{X_k^m}\right)^{\frac{1}{2^m}}\right)$.
\end{lemma}

\begin{proof}
If 
$y \in \left(f_2\left(\widehat{X_k^m}\right)\right)^{\frac{1}{2^m}} \cap f_2(X_k)$ then there exists $x\in X_k$ such that $y = f_2(x)$ and 
there exists $x' \in \widehat{X_k^m}$ such that $\frac{1}{2^m} > d(f_2(x'),y) = d(f_2(x'),f_2(x))$.  As 
$L > \frac{1}{2^m}$, $d_k(x,x') = d(f_2(x),f_2(x')) < \frac{1}{2^m}$.  So, $y\in f_2\left(\left(\widehat{X_k^m}\right)^{\frac{1}{2^m}}\right)$.
As $y\in \left(f_2 \left(\widehat{X_k^m}\right)\right)^{\frac{1}{2^m}} \cap f_2(X_k)$ was arbitrary, 
$\left(f_2 \left(\widehat{X_k^m}\right)\right)^{\frac{1}{2^m}} \cap f_2(X_k) \subset f_2\left(\left(\widehat{X_k^m}\right)^{\frac{1}{2^m}}\right)$.
\end{proof}

Thus, we have, $\mu_k\left(\left(\widehat{X_k^m}\right)^{\frac{1}{2^m}}\right) \geq \nu_k\left(f_2\left(\left(\widehat{X_k^m}\right)^{\frac{1}{2^m}}\right)\right) \geq \nu_k\left(\left(f_2\left(\widehat{X_k^m}\right)\right)^{\frac{1}{2^m}} \cap f_2(X_k)\right)$.  
Further,  $\nu_k\left(\left(f_2\left(\widehat{X_k^m}\right)\right)^{\frac{1}{2^m}} \cap f_2(X_k)\right) = \nu_k\left(\left(f_2\left(\widehat{X_k^m}\right)\right)^{\frac{1}{2^m}}\right)$
as $\nu_k(E) = \nu_k\left(E \cap f_2(X_k)\right)$.

\begin{lemma}
Given a set $A \subset Z$ and real numbers $a$ and $b$, $(A^a)^b \subset A^{a + b}$.
\end{lemma}

\begin{proof}
Given a set $A$ and real numbers $a$ and $b$, $x\in (A^a)^b$ implies $\exists x_1 \in A^a$ such that $d(x,x_1) < b$ 
which implies that $\exists x_2\in A$ such that $d(x_1,x_2) < a$.  Thus, 
$d(x,x_2) \leq d(x,x_1) + d(x_1,x_2) \leq a + b$.  So, $x\in A^{a+b}$.  As, $x\in (A^a)^b$ was arbitrary, we have 
$(A^a)^b \subset A^{a + b}$.
\end{proof}
  Hence, 
$\nu_k\left(\left(\left(f_2\left(\widehat{X_k^m}\right)\right)^{\frac{1}{2^{m+1}}}\right)^{\frac{1}{2^{m+1}}}\right) \leq \nu_k\left(\left(f_2\left(\widehat{X_k^m}\right)\right)^{\frac{1}{2^m}}\right) \leq \mu_k\left(\left(\widehat{X_k^m}\right)^{\frac{1}{2^m}}\right)$.

\begin{lemma}
$\left(f_1\left(X_{N(m)}^{m+2}\right)\right)^{\frac{1}{2^{m+2}}} \cap f_2(X_k) \subset \left(f_2\left(\widehat{X_k^m}\right)\right)^{\frac{1}{2^{m+1}}}$.
\end{lemma}

\begin{proof}
If $z \in \left(f_1\left(X_{N(m)}^{m+2}\right)\right)^{\frac{1}{2^{m+2}}} \cap f_2(X_k)$ there exists an $i$ such that 
$d(z,f_1(x_i)) < \frac{1}{2^{m+2}}$.  So, $f_2(X_k) \cap B\left(f_1(x_i), \frac{1}{2^{m+2}}\right) \neq \emptyset$.  The 
distance $d(z, f_2(z_i)) \leq d(z, f_1(x_i)) + d(f_1(x_i),f_2(z_i)) \leq \frac{1}{2^{m+2}} + \frac{1}{2^{m+2}} = \frac{1}{2^{m+1}}$.
So, $z \in \left(f_2\left(\widehat{X_k^m}\right)\right)^{\frac{1}{2^{m+1}}}$.  As, $z \in \left(f_1\left(X_{N(m)}^{m+2}\right)\right)^{\frac{1}{2^{m+2}}} \cap f_2(X_k)$
was arbitrary, $\left(f_1\left(X_{N(m)}^{m+2}\right)\right)^{\frac{1}{2^{m+2}}} \cap f_2(X_k) \subset \left(f_2\left(\widehat{X_k^m}\right)\right)^{\frac{1}{2^{m+1}}}$.
\end{proof}

\noindent Therefore,
\begin{align*}
\nu_k\left(\left(\left(f_2\left(\widehat{X_k^m}\right)\right)^{\frac{1}{2^{m+1}}}\right)^{\frac{1}{2^{m+1}}}\right) &\geq \nu_k\left(\left(\left(f_1\left(X_{N(m)}^{m+2}\right)\right)^{\frac{1}{2^{m+2}}}\right)^{\frac{1}{2^{m+1}}}\right)\\
&\geq \nu_{N(m)}\left(\left(f_1\left(X_{N(m)}^{m+2}\right)\right)^{\frac{1}{2^{m+2}}}\right) - \frac{1}{2^{m+1}} \ \ \left( \text{use } d_{\pi}(\nu_N,\nu_k) \leq \frac{1}{2^{m+3}} \leq \frac{1}{2^{m+1}} \right)\\
&\geq \nu_{N(m)}\left(f_1\left(\left(X_{N(m)}^{m+2}\right)^{\frac{1}{2^{m+2}}}\right)\right) - \frac{1}{2^{m+1}} \\
&\ \ \ \ \ \left( f_1\left(\left(X_{N(m)}^{m+2}\right)^{\frac{1}{2^{m+2}}}\right) \subset \left(f_1\left(X_{N(m)}^{m+2}\right)\right)^{\frac{1}{2^{m+2}}}  \text{as $f_1$ is $L$-isometric} \right)\\
&\geq \mu_{N(m)}\left(\left(X_{N(m)}^{m+2}\right)^{\frac{1}{2^{m+2}}}\right) - \frac{1}{2^{m+3}} - \frac{1}{2^{m+1}} \\
&\ \ \ \ \ \left( \text{use $f_1$ is an $\varepsilon$-embedding where } \varepsilon \leq \frac{1}{2^{m+3}} \right)\\
&\geq \mu_{N(m)}(X_{N(m)}) -\frac{1}{2^{m+2}} - \frac{1}{2^{m+3}} - \frac{1}{2^{m+1}}\\
&\geq \mu_k(X_k) -\frac{1}{2^{m+3}} -\frac{1}{2^{m+3}} - \frac{1}{2^{m+2}} - \frac{1}{2^{m+1}}\\
& \ \ \ \ \ \left( \text{use Lemma \ref{total-measures-close-when-the-metric-measure-spaces-are} and } d_{\rho}(X_{N(m)},X_k) \leq \frac{1}{2^{m+3}} \right)\\
&\geq \mu_k(X_k) - \frac{1}{2^m}  
\end{align*}
Thus, we have proved $\mu_k\left(\left(\widehat{X_k^m}\right)^\frac{1}{2^m}\right) \geq \mu_k(X_k) - \frac{1}{2^m}$.
\end{proof}

Redefine $X_n^m$ to be $\widehat{X_n^m} \cup X_n^{m-1}$.  Then, $\left\vert{X_n^m}\right\vert \leq \left\vert{\widehat{X_n^m}}\right\vert + \left\vert{X_n^{m-1}}\right\vert$.
So, for $n > \max\lbrace N(m), N(m-1)\rbrace$, $\left\vert{X_n^m}\right\vert \leq \left\vert{X_{N(m)}^{m+2}}\right\vert + \left\vert{X_{N(m-1)}^{m+1}}\right\vert $.
Thus, we have constructed $X_n^m$ with bounded cardinality.  
\end{proof}

Hence by Theorem \ref{spacefdmscompact}, for each $m$ we have that $(X_n^m,d_n^m,\mu_n^m)$
converges as $n$ tends to infinity.  Denote the limit of $(X_n^m,d_n^m,\mu_n^m)$
as $n$ tends to $\infty$ by $(X_{\infty}^m,d_{\infty}^m,\mu_{\infty}^m)$.  As $X_n^m \subset X_n^{m+1}$, we
have a sequence of isometric embeddings 
$$(X_{\infty}^1,d_{\infty}^1)\hookrightarrow
(X_{\infty}^2,d_{\infty}^2)\hookrightarrow...$$
consider the direct limit of this sequence to get a distance measure space
$(X^{\infty}_{\infty},d_{\infty}^{\infty})$.  Let $(X,d)$ denote the completion
of  $(X^{\infty}_{\infty},d_{\infty}^{\infty})$.  We have isometric embeddings
from $(X_{\infty}^m,d_{\infty}^m) \hookrightarrow
(X^{\infty}_{\infty},d_{\infty}^{\infty}) \hookrightarrow X$, for each $m$. 
Call the isometric embedding $(X_{\infty}^m,d_{\infty}^m) \hookrightarrow X$,
$\psi_m$.  Push forward $\mu_{\infty}^m$ to $X$ by $\psi_m$ and call it $\nu^m$.

\begin{lemma}
The sequence $\nu^n$ is Cauchy.
\end{lemma}

\begin{proof}
Assume without loss of generality that $n<m$.  Then, we have inclusions
$\psi_{n}^{m}:(X_{\infty}^n,d_{\infty}^n) \hookrightarrow
(X_{\infty}^m,d_{\infty}^m)$.
Thus, by Lemma \ref{piinclusionnicebehaviour} we have,
$$d_{\pi}^{(X,d)}(\nu_n,\nu_m) = d_{\pi}^{(X_{\infty}^m,d_{\infty}^m)}((\psi_n^m)_*(\mu_{\infty}^n),\mu_{\infty}^m)$$
The sets $X_k^m$ and $X_k^{n}$ are subsets of $X_k$.  Viewing $\mu_k^m$ and $\mu_k^{n}$ as measures on $X_k$, we have
$$d_{\pi}^{(X_{\infty}^m,d_{\infty}^m)}((\psi_n^m)_*(\mu_{\infty}^n),\mu_{\infty}^m) = \lim_{k \to \infty} d_{\pi}^{(X_k,d_k)}(\mu_k^m,\mu_k^n)$$
\subsubsection*{Claim: $d_{\pi}^{(X_k,d_k)}(\mu_k^m,\mu_k^{m+1}) \leq \frac{1}{2^m} + \frac{1}{2^{m+1}}$}
The measure $\mu_k^m$ and $\mu_k^{m+1}$ are atomic measure.  So, it is enough to check for singleton sets.  

If $x\in X_k^m$, then $\mu_k^m(\lbrace x\rbrace) \geq \mu_k^{m+1}(\lbrace x\rbrace)$.  On the other hand 
$\mu_k^m(\lbrace x\rbrace) \leq \mu_k\left(B\left(x,\frac{1}{2^m}\right)\right)$ and 
$B\left(x,\frac{1}{2^m}\right) \subset \left(\cup_{y\in X_k^{m+1} \ s.t \ d_k(x,y) < \frac{1}{2^m} + \frac{1}{2^{m+1}}} B\left(y,\frac{1}{2^{m+1}}\right)\right) \cup \left(\left(X_k^{m+1}\right)^{\frac{1}{2^{m+1}}}\right)^C$.
Thus,
$$\mu_k^m(\lbrace x\rbrace) \leq \mu_k^{m+1}\left(\lbrace x\rbrace^{\frac{1}{2^m}+ \frac{1}{2^{m+1}}}\right) + \frac{1}{2^{m+1}}. $$

Similarly, if $x\in X_n^{m+1}\setminus X_n^m$, then $\mu_k^{m+1}(\lbrace x\rbrace) \geq \mu_k^{m}(\lbrace x\rbrace)$.  On the other hand
$\mu_k^{m+1}(\lbrace x\rbrace) \leq \mu_k\left(B\left(x,\frac{1}{2^{m+1}}\right)\right)$ and
$B\left(x,\frac{1}{2^{m+1}}\right) \subset \left(\cup_{y\in X_k^{m} \ s.t \ d_k(x,y) < \frac{1}{2^{m}} + \frac{1}{2^{m+1}}} B\left(y,\frac{1}{2^{m}}\right)\right) \cup \left(\left(X_k^{m}\right)^{\frac{1}{2^{m}}}\right)^C$.  Thus,
$$\mu_k^{m+1}(\lbrace x\rbrace) \leq \mu_k^{m}\left(\lbrace x\rbrace^{\frac{1}{2^m}+ \frac{1}{2^{m+1}}}\right) + \frac{1}{2^{m}}. $$ 

For all other points $x$, $\mu_n^m(\lbrace x\rbrace) = 0 = \mu_n^{m+1}(\lbrace x\rbrace)$.  Thus we have the claim.  Hence,
$$d_{\pi}^{(X_k,d_k)}(\mu_k^m,\mu_k^{n}) \leq \sum_{i=n}^m \left( \frac{1}{2^i} +  \frac{1}{2^{i+1}} \right).$$
So, we have,
$$d_{\pi}^{(X,d)}(\nu_n,\nu_m) \leq \sum_{i=n}^m \left( \frac{1}{2^i} +  \frac{1}{2^{i+1}} \right).$$ 
Therefore, $\nu^n$ is Cauchy.

\end{proof}

Further, $\nu^n$ converges because, $(X,d)$ is complete implies
$(\mathcal{P}(X),d_{\pi})$ is complete.  Call the limit $\mu$.  It is clear from the definition of $\mu$ that, 
$(X_{\infty}^m,d_{\infty}^m,\mu_{\infty}^m)$ converges to $(X,d,\mu)$ as $m$ goes to infinity.

We claim that,
$(X_n,d_n,\mu_n)$ converges to $(X,d,\mu)$.  We will first prove that the
convergence of the sequence $(X_n^m,d_n^m,\mu_n^m)$, as $n$ tends to $\infty$,
to $(X_{\infty}^m,d_{\infty}^m,\mu_{\infty}^m)$ is uniform, in some sense, with
respect to $m$.

\begin{lemma}
The sequence $(X_{\infty}^n, d_{\infty}^n, \mu_{\infty}^n)$ is Cauchy.  In fact,
$$d_{\rho}((X_{\infty}^n, d_{\infty}^n, \mu_{\infty}^n), (X_{\infty}^m,
d_{\infty}^m, \mu_{\infty}^m)) \leq \frac{1}{2^m} + \frac{1}{2^n}.$$
\end{lemma}

\begin{proof}
For all $\varepsilon>0$ we can find an $N$ such that 
$$d_{\rho}((X_{\infty}^n, d_{\infty}^n, \mu_{\infty}^n), (X_k^n, d_k^n, \mu_k^n))
\leq \varepsilon, \ \forall k \geq N.$$ 
and
$$d_{\rho}((X_{\infty}^m, d_{\infty}^m, \mu_{\infty}^m), (X_k^m, d_k^m, \mu_k^m))
\leq \varepsilon, \ \forall k \geq N. $$  But,
$d_{\rho}((X_k^i,d_k^i,\mu_k^i),(X_k,d_k,\mu_k)) \leq \frac{1}{2^i}$.  Thus,
\begin{align*}
d_{\rho}((X_{\infty}^n, d_{\infty}^n, \mu_{\infty}^n), (X_{\infty}^m, d_{\infty}^m,
\mu_{\infty}^m)) &\leq d_{\rho}((X_{\infty}^n, d_{\infty}^n, \mu_{\infty}^n),
(X_k^n, d_k^n, \mu_k^n))\\
& \ \ \ \ \ \ + d_{\rho}((X_k^n,d_k^n,\mu_k^n),(X_k,d_k,\mu_k))\\
& \ \ \ \ \ \ + d_{\rho}((X_k^m,d_k^m,\mu_k^m),(X_k,d_k,\mu_k)) \\
& \ \ \ \ \ \ + d_{\rho}((X_{\infty}^m, d_{\infty}^m, \mu_{\infty}^m), (X_k^m,
d_k^m, \mu_k^m))
\end{align*}
Hence, $d_{\rho}((X_{\infty}^n, d_{\infty}^n, \mu_{\infty}^n), (X_{\infty}^m,
d_{\infty}^m, \mu_{\infty}^m)) \leq \varepsilon + \frac{1}{2^m} + \frac{1}{2^n} +
\varepsilon = \frac{1}{2^m} + \frac{1}{2^n} + 2\varepsilon $.  As, $\varepsilon >0$
was arbitrary we have the result.  
\end{proof}

\begin{lemma}
$$d_{\rho}((X_n^m,d_n^m,\mu_n^m),(X_{\infty}^m, d_{\infty}^m ,\mu_{\infty}^m)) \leq
d_{\rho}((X_n^{m'},d_n^{m'},\mu_n^{m'}),(X_{\infty}^{m'},d_{\infty}^{m'},\mu_{\infty
}^{m'})) + 2(\frac{1}{2^m} +\frac{1}{2^{m'}}).$$   
\end{lemma}
 
 \begin{proof}
 We have already proved that
 $$d_{\rho}((X_{\infty}^{m'}, d_{\infty}^{m'}, \mu_{\infty}^{m'}), (X_{\infty}^m,
d_{\infty}^m, \mu_{\infty}^m)) \leq \frac{1}{2^m} + \frac{1}{2^{m'}}.$$
 Similarly,
\begin{align*}
 d_{\rho}((X_{n}^{m'}, d_{n}^{m'}, \mu_{n}^{m'}), (X_{n}^m, d_{n}^m, \mu_{n}^m))
&\leq d_{\rho}((X_{n}^{m'}, d_{n}^{m'}, \mu_{n}^{m'}), (X_n,d_n,\mu_n))\\ 
&\ \ \ \ \ + d_{\rho}( (X_n,d_n,\mu_n), (X_{n}^m, d_{n}^m, \mu_{n}^m))\\
&\leq \frac{1}{2^m} + \frac{1}{2^{m'}}
\end{align*}   
Thus,
 \begin{align*}
 d_{\rho}((X_n^m,d_n^m,\mu_n^m),(X_{\infty}^m, d_{\infty}^m ,\mu_{\infty}^m)) &\leq
d_{\rho}((X_n^m,d_n^m,\mu_n^m), (X_n^{m'},d_n^{m'},\mu_n^{m'}))\\
 & \ \ \ \ \ \ + d_{\rho}((X_n^{m'},d_n^{m'},\mu_n^{m'}),
(X_{\infty}^{m'},d_{\infty}^{m'},\mu_{\infty}^{m'}))\\
 & \ \ \ \ \ \ + d_{\rho}((X_{\infty}^{m'}, d_{\infty}^{m'}, \mu_{\infty}^{m'}),
(X_{\infty}^m, d_{\infty}^m, \mu_{\infty}^m))\\
&\leq \left(\frac{1}{2^m} + \frac{1}{2^{m'}}\right) + d_{\rho}((X_n^{m'},d_n^{m'},\mu_n^{m'}),
(X_{\infty}^{m'},d_{\infty}^{m'},\mu_{\infty}^{m'})) \\
&\ \ \ \ \ \ + \left(\frac{1}{2^m} + \frac{1}{2^{m'}}\right)
 \end{align*}
\end{proof}
 
 Now to prove $(X_n,d_n,\mu_n)$ converges to $(X,d,\mu)$.  Given $\varepsilon
>0$, choose $N$ such that for all $n\geq N$ we have,
 $$d_{\rho}((X_n^m,d_n^m,\mu_n^m),(X_{\infty}^m, d_{\infty}^m ,\mu_{\infty}^m)) \leq
\frac{\varepsilon}{3}$$
 and choose $m$ large enough that $\frac{1}{2^m} \leq \frac{\varepsilon}{3}$ and 
 $$d_{\rho}\left((X_{\infty}^m, d_{\infty}^m, \mu_{\infty}^m), (X,d,\mu)\right) \leq \frac{\varepsilon}{3}.$$
 Then for $n\geq N$ and $m$ as described above, we have;
 \begin{align*}
 d_{\rho}((X_n,d_n,\mu_n), (X,d,\mu)) &\leq d_{\rho}((X_n,d_n,\mu_n),
(X_n^m,d_n^m,\mu_n^m))\\
 & \ \ \ \ \ \ + d_{\rho}((X_n^m,d_n^m,\mu_n^m),(X_{\infty}^m, d_{\infty}^m
,\mu_{\infty}^m))\\
 & \ \ \ \ \ \ + d_{\rho}((X_{\infty}^m, d_{\infty}^m, \mu_{\infty}^m), (X,d,\mu))\\
 &\leq \frac{1}{2^m} + \frac{\varepsilon}{3} + \frac{\varepsilon}{3} \leq
\varepsilon
 \end{align*}
Thus we have completed the proof of Theorem \ref{completenessunderrho} by proving that $(X_n,d_n,\mu_n)$ converges to $(X,d,\mu)$.

\end{proof}

\subsection{A pre-compactness theorem}
Now, we will prove a theorem analogous to the Gromov's compactness theorem for
metric spaces, for distance measure spaces.    

\begin{theorem}
Let $\mathfrak{X}_{\varepsilon,M}$ be a collection of complete separable distance measure spaces with
the property that $\mu(X)\leq M$ for all $X\in \mathfrak{X}_{\varepsilon,M}$.  Suppose, given $\varepsilon > 0$ there exists $N(\varepsilon)$ such that,
for all $(X,d,\mu) \in \mathfrak{X}_{\varepsilon,M}$ there exists a set
$S_{X,\ \varepsilon}$ such that $\vert S_{X,\ \varepsilon} \vert \leq
N(\varepsilon)$ and $\mu((S_{X,\ \varepsilon})^{\varepsilon})^C)\leq \varepsilon$. Then, $\mathcal{\mathfrak{X}_{\varepsilon,M}}$ is
totally bounded and hence pre-compact.
\label{main}
\end{theorem}

\begin{proof}
The space of finite distance measure spaces
with bounded area and bounded number of points is totally bounded, as it is compact.  Thus given
any $\varepsilon$ we can choose a finite cover of $\frac{\varepsilon}{2}$ balls
for the same.  Thus $\varepsilon$ balls (centred at the same points) gives a
$\varepsilon$ sized cover for  $\mathfrak{X}_{\varepsilon,M}$, as, given any $X\in
\mathfrak{X}_{\varepsilon,M}$ you can choose a finite distance measure space which is at
an $\frac{\varepsilon}{2}$ distance from $S$.  Thus
$\mathfrak{X}_{\varepsilon,M}$ is totally bounded. 
\end{proof}

\begin{rmk}
Consider the sequence $(\{x\}, d(x,x) = 0, \mu_n(x) =n)$.  This sequence has no convergent subsequence.  This illustrates that the condition of bounded area is necessary.  Scaling the Riemannian metric on a fixed surface gives us a similar example with Riemann surfaces.  
\end{rmk}

\section{Deligne-Mumford compactification as completion}\label{CH: DMCompactificationasCompletion}

\subsection{Cusp curves}
Let $S$ be a closed surface and let $\left(\gamma_i\right)_{i\in I}$ be a possibly empty
family of finitely many smooth simple closed and pairwise disjoint loops in $S$.  Let 
$\widehat{S}$ be the surface obtained from $S \setminus \cup_i \gamma_i$ by the one point
compactification at each end.  By $s_k'$ and $s_k''$, $k\in I$, we denote the two points
of $\widehat{S}$ added to $S \setminus \cup_i \gamma_i$ at the two ends which arise from
removing $\gamma_k$.  We now identify $s_k'$ and $s_k''$ for $k \in I$.  In this way we
obtain a compact topological space $\overline{S}$.  Hence we can obtain $\overline{S}$ from
$S$ by collapsing each loop $\gamma_i$ to a point.  By $\alpha: \widehat{S} \to \overline{S}$
we denote the canonical projection.  The points $\overline{s_k}:= \alpha(s_k')$ are called
singular points of the singular surface $\overline{S}$.  We denote the set of singular points on $\overline{S}$ by $si(\overline{S})$.  Via $\alpha$ the subset 
$\overline{S}\setminus \left\lbrace \overline{s_i} \ \ \vert \ i\in I \right\rbrace \subset \overline{S}$
inherits a differentiable structure.  A complex structure $\overline{j}$ on the singular
surface $\overline{S}$ is a complex structure on $\widehat{S}$.   

\subsection{The space of cusp curves equipped with $d_{\rho}$ }

Given a Riemann surface $S$, the metric on $S$ defines an area form on $S$ viz. 
$$\sigma_g (v,w) = \left[g(v,v).g(v,w)-g^2(v,w)\right]^{\frac{1}{2}}.$$
Then, $\mu_g(E)= \int_E \sigma_g$ gives a measure, making $(S,g,\mu_g)$ a metric
measure space.  Thus the generalised GHLP metric gives a (pseudo)metric on the
space of Riemann surfaces.

\begin{theorem}
The generalised Gromov-Hausdorff-Levi-Prokhorov distance, $d_{\rho}$, is a metric on the space of Riemann surfaces.
\end{theorem}

\begin{proof}
Two distance measure spaces $(X_1,d_1,\mu_1),(X_2,d_2,\,u_2)$ are at zero
distance from each other if and only if there are open sets $U_1 \subset X_1$
and $U_2 \subset X_2$, of zero measure, such that there is a measure preserving
isometry between  $(X_1 \setminus U_1,d_1,\mu_1)$ and 
$(X_2 \setminus U_2, d_2,\mu_2)$. But, given a Riemannian metric $g$ on a surface $S$, 
$\mu_g(U) > 0$ for all open sets $U \subset S$.  Thus the result.
\end{proof}

Given a cusp curve, associated to it is the hyperbolic surface $S \setminus si(S)$. 
Given two cusp curves $S_1$ and $S_2$ we define the $d_{\rho}(S_1,S_2) := d_{\rho}(S_1 \setminus
si(S_1), S_2 \setminus si(S_2))$.  This gives a metric on the space of all cusp
curves.

\subsection{Fenchel Nielson co-ordinates and $d_{\rho}$}

In this section we show that the topology generated by $d_{\rho}$ is the same as
that given by the Fenchel Nielson co-ordinates.

\begin{dfn}[$q$-quasi-isometry]
A homeomorphism $\varphi: X \to Y$ between two metric spaces $X$ and $Y$ is a $q$-quasi-isometry ($q \geq 1$) or quasi-isometry if 
$$\frac{1}{q} d_X(x,y) \leq d_Y(\varphi(x),\varphi(y)) \leq q d_X(x,y)$$
for all $x,y \in X$.
\end{dfn}

\begin{lemma}
If $R$ and $S$ are q-quasi-isometric Riemann surfaces then, 
$$d_{\rho}(R,S) \leq max\left\lbrace \left(1 - \frac{1}{q^2}\right)\mu_R(R),\left(1 - \frac{1}{q^2}\right)\mu_S(S), diam(S) \left(q - 1\right) \right\rbrace =:
\varepsilon .$$ 
\end{lemma}

\begin{proof}
Let $\varphi: R \to S$ be a q-quasi-isometry.  Then, $\frac{1}{q^2} \times g|_x \leq g'|_{\varphi(x)} \leq q^2 \times g|x $
where, $g$ and $g'$ denotes the Riemannian metric on $R$ and $S$ respectively.  So, 
$\frac{1}{q^2} \times \sigma_{g}|_x \leq \sigma_{g'}|_{\varphi(x)}\leq q^2 \times \sigma_g|_x$.  Thus, denoting the 
measures on $R$ and
$S$ described above by $\mu_R$ and $\mu_S$ we have 
$$\frac{1}{q^2} \times \mu_S( \varphi(E)) \leq \mu_R(E) \leq q^2 \times \mu_S(
\varphi(E)). $$
Define $Z := R\sqcup S$ and let the metric on $Z$ be the maximal metric 
$d_{\varphi}^{\varepsilon}$ defined in Section \ref{maxmet}.  As $\varphi$ is a 
q-quasi-isometry,
\begin{align*}
d_R(x,y) - d_S(\varphi(x),\varphi(y)) \leq d_R(x,y) - \frac{1}{q} d_R(x,y) = d_R(x,y)\left(1 - \frac{1}{q}\right)  
\end{align*}
and 
\begin{align*}
d_S(\varphi(x),\varphi(y)) - d_R(x,y) \leq q d_R(x,y) - d_R(x,y) = d_R(x,y) (q-1).
\end{align*}
That is,
$$d_R(x,y) (1-q)\leq d_R(x,y) - d_S(\varphi(x),\varphi(y)) \leq d_R(x,y)\left(1 - \frac{1}{q}\right) \leq d_R(x,y) (q - 1)  $$
Thus, if $\varepsilon \geq diam(S) \left(q - 1\right)$,
then by Theorem \ref{inclmaxmet} the inclusion $R\hookrightarrow Z$ and 
$S\hookrightarrow Z$ are isometric $0$-embeddings.  Thus the inclusions from 
$R$ and $S$ to $Z$ will give the required result because, given $E \subset Z$,
if $\varepsilon' > \varepsilon$;
\begin{align*}
\mu_R(E^{\varepsilon'}) + \varepsilon' &\geq \mu_R((E\cap S)^{\varepsilon'}) + \varepsilon' \geq \mu_R(\phi^{-1}(E \cap S)) + \varepsilon'\\
&\geq \frac{1}{q^2} \mu_S(\varphi \circ \varphi^{-1}(E \cap S)) + \varepsilon'\\
&\geq \frac{1}{q^2} \mu_S(E \cap S) + \varepsilon'\\
&= \frac{1}{q^2} \mu_S(E) + \varepsilon'
\geq \mu_S(E).
\end{align*}
and,
\begin{align*}
\mu_S(E^{\varepsilon'}) + \varepsilon' &\geq \mu_S((E\cap R)^{\varepsilon'}) + \varepsilon' \geq \mu_S(\phi(E \cap R)) + \varepsilon'\\
&\geq \frac{1}{q^2} \mu_R(E \cap R) + \varepsilon'\\
&= \frac{1}{q^2} \mu_R(E) + \varepsilon' \geq \mu_R(E).
\end{align*}
As $\varepsilon' > \varepsilon$ was arbitrary, we have the result.
\end{proof}

Chapter $3$ of \cite{buser} describes construction of quasi-isometries between
two non degenerate pair of pants and quasi isometry while gluing pair of pants. 
So, all that remains is the case of degenerate pair of pants.

\begin{lemma}
Let $H_1, H_2$ be the hexagons with sides $(b_1,b_2,b_3)$ and $(b_1,0,b_3)$. 
Then the distance $d_{\rho}(H_1, H_2)$ is small if $b_2$ is small
\end{lemma}

\begin{proof}
The two hexagons can be embedded in $\mathbb{H}$ as shown below

\begin{center}
\includegraphics[scale=.49]{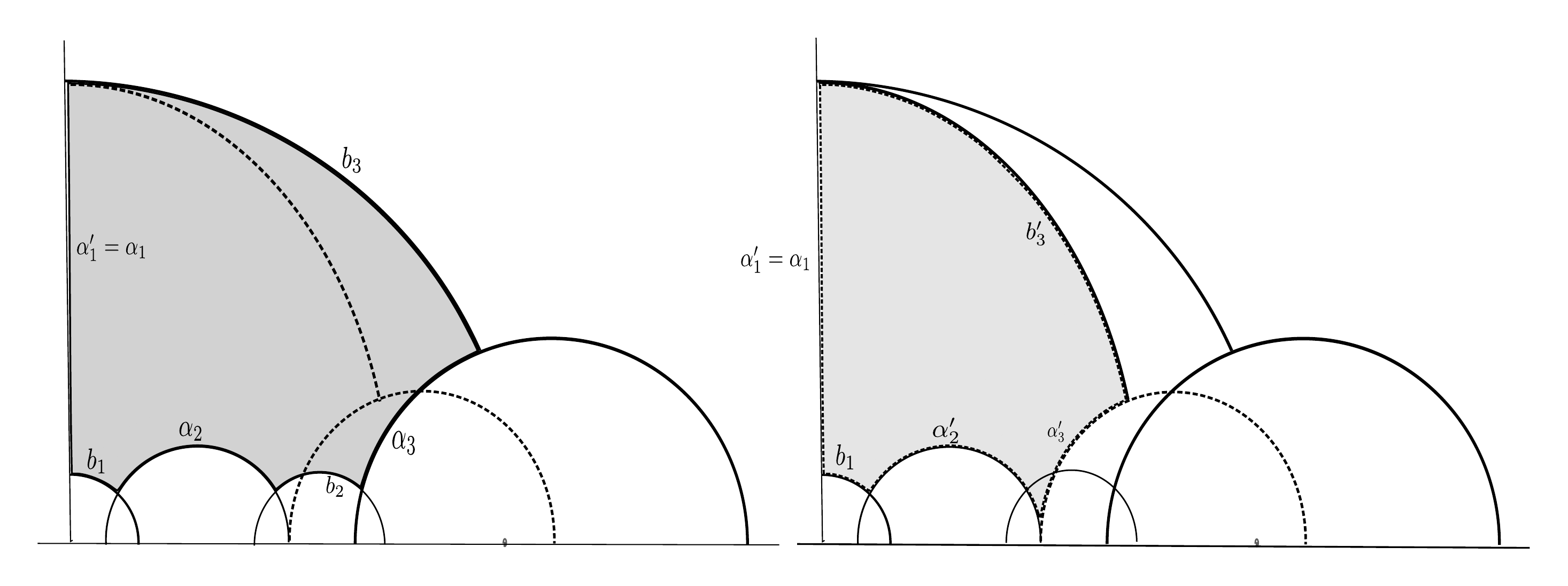}
\end{center}

The hexagon shaded in the picture on the left is $H_1$ and on the right is
$H_2$.  The distance $d_{\rho}(H_1,H_2)$ is clearly less than the area of $(H_1\cup H_2) - (H_1
\cap H_2)$, the symmetric difference. 

\begin{center}
\includegraphics[scale=.5]{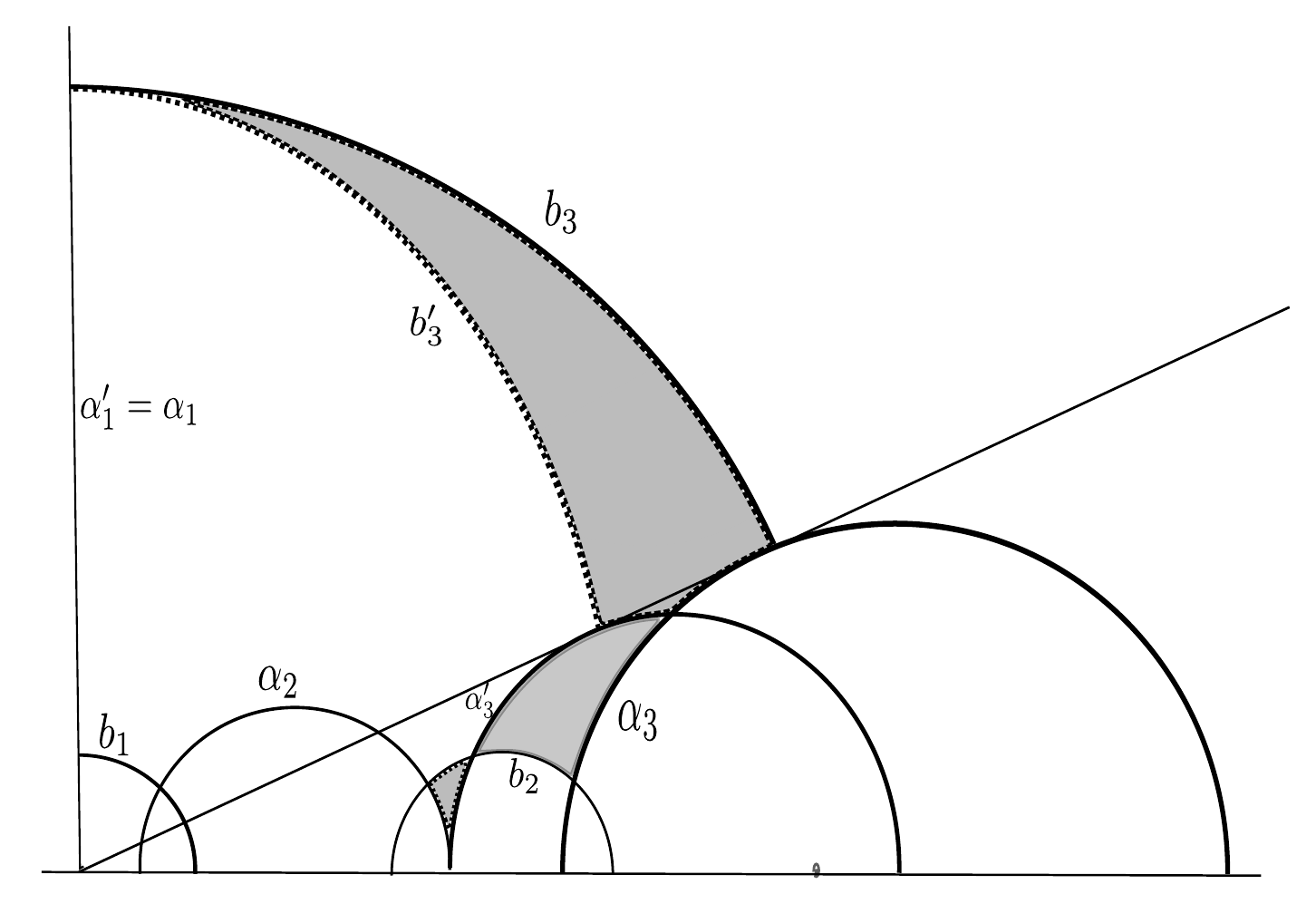}
\end{center}

Using the notations as in the figure below we will now compute this area of the
symmetric difference.

\begin{center}
\includegraphics[scale=.8]{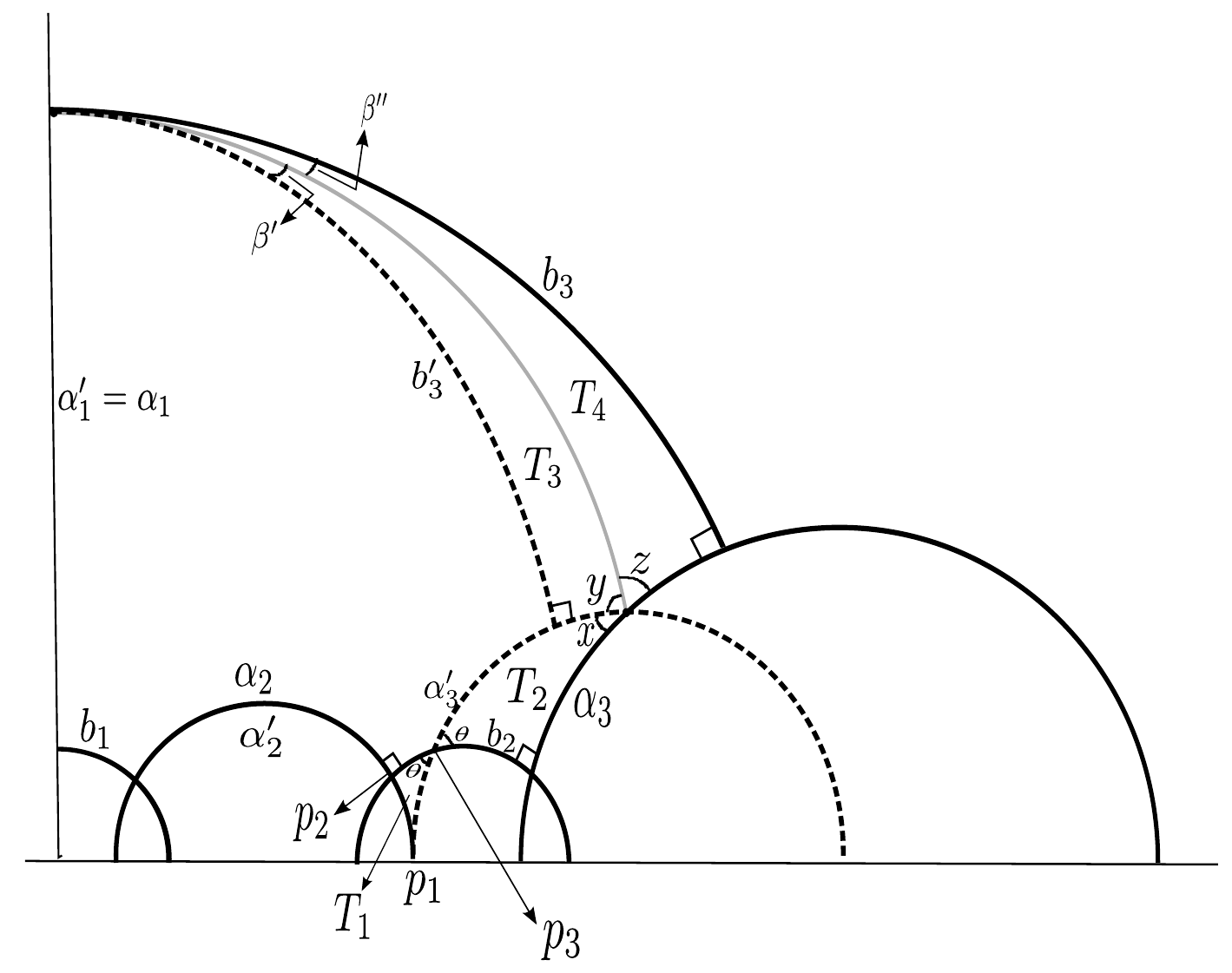}
\end{center}

\begin{align*}
area((H_1\cup H_2) - (H_1 \cap H_2)) &= area(T_1) + area(T_2)+ area(T_3) +
area(T_4)\\
area(T_1) &= \pi - \frac{\pi}{2} - \theta\\
area(T_2) &= \pi - \frac{\pi}{2} - \theta - x\\
area(T_3) &= \pi - \frac{\pi}{2} - \beta' - y\\
area(T_4) &= \pi - \frac{\pi}{2} - \beta'' - z \\
area((H_1\cup H_2) - (H_1 \cap H_2)) &= 2\pi - 2 \theta - (x + y + z) - (\beta'
+ \beta'')\\
&= 2\pi - 2 \theta - \pi - \beta\\
&= \pi - 2 \theta - \beta
\end{align*}

As $b_2$ tends to zero, $\beta$ tends to zero and $\theta$ tends to
$\frac{\pi}{2}$.  It will be slightly easier to see the fact that $\theta$ tends
to $\frac{\pi}{2}$ if, we take a map which takes $p_1$ to $\infty$.  See the
picture below.

\begin{center}
\includegraphics[scale=.35]{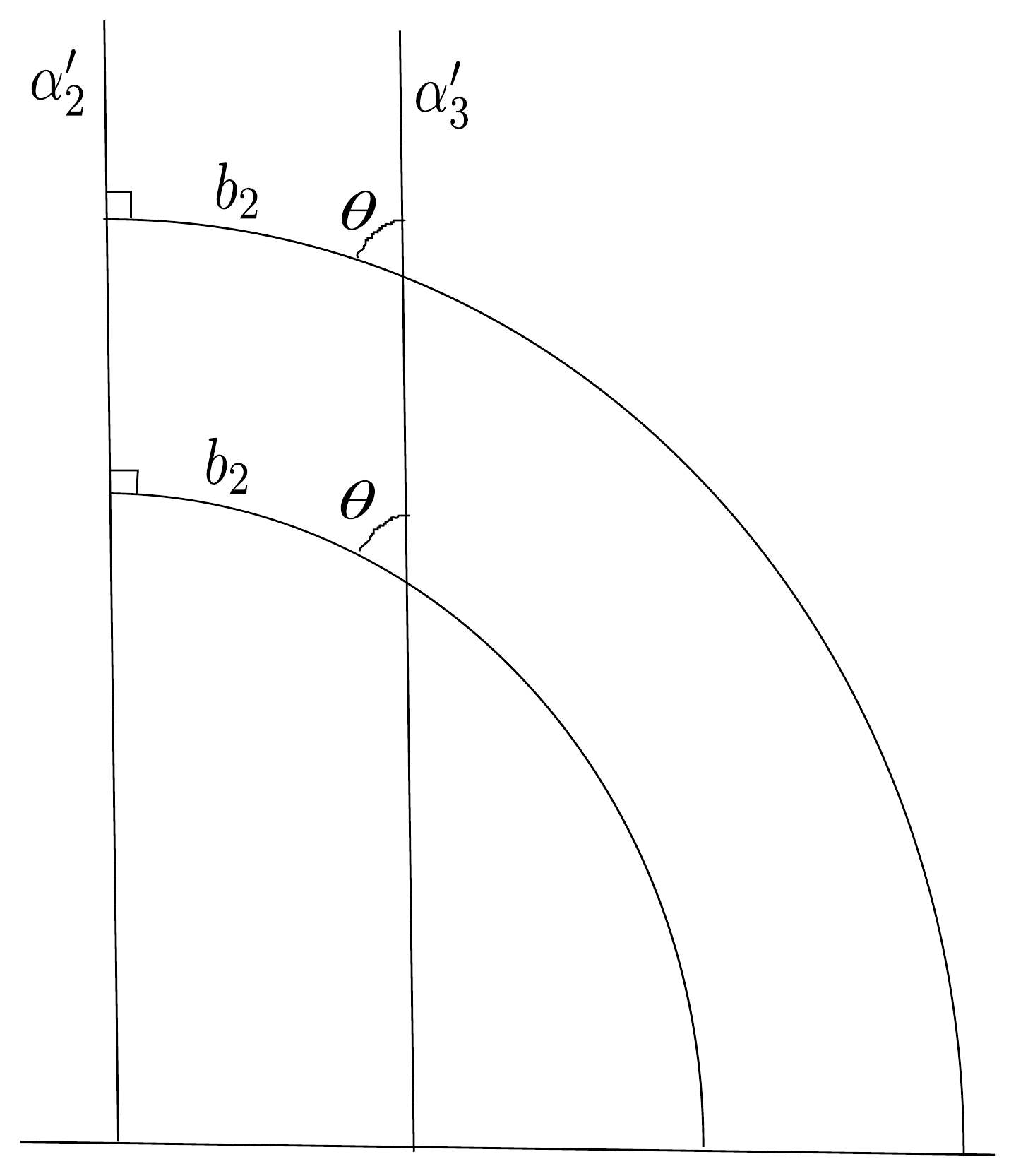}
\end{center}

\end{proof}

\begin{lemma}
Let $H_1, H_2$ be the hexagons with sides $(0,b_2,b_3)$ and $(0,0,b_3)$.  Then
$d_{\rho}(H_1, H_2)$ is small if $b_2$ is small.
\end{lemma}

\begin{proof}
The same calculations as before would work with the slightly modified picture
given below.
\begin{center}
\includegraphics[scale=.8]{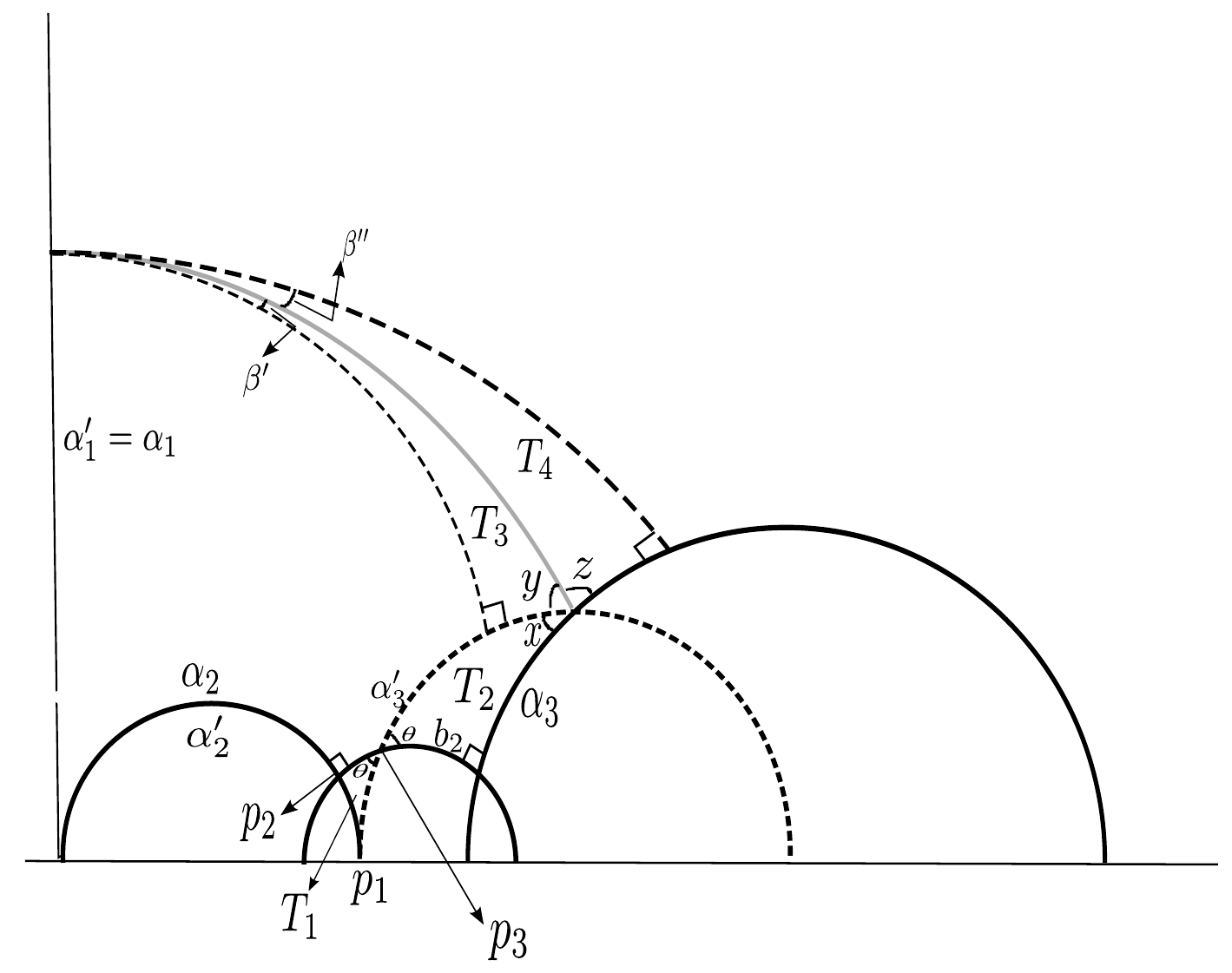}
\end{center}

\end{proof}

\begin{lemma}
Let $H_1, H_2$ be the hexagons with sides $(0,0,b_3)$ and $(0,0,0)$.  Then
$d_{\rho}(H_1, H_2)$ is small if $b_3$ is small.
\end{lemma}

\begin{proof}
Again we compute the symmetric difference as in the previous two cases.
\begin{center}
\includegraphics[scale=.48]{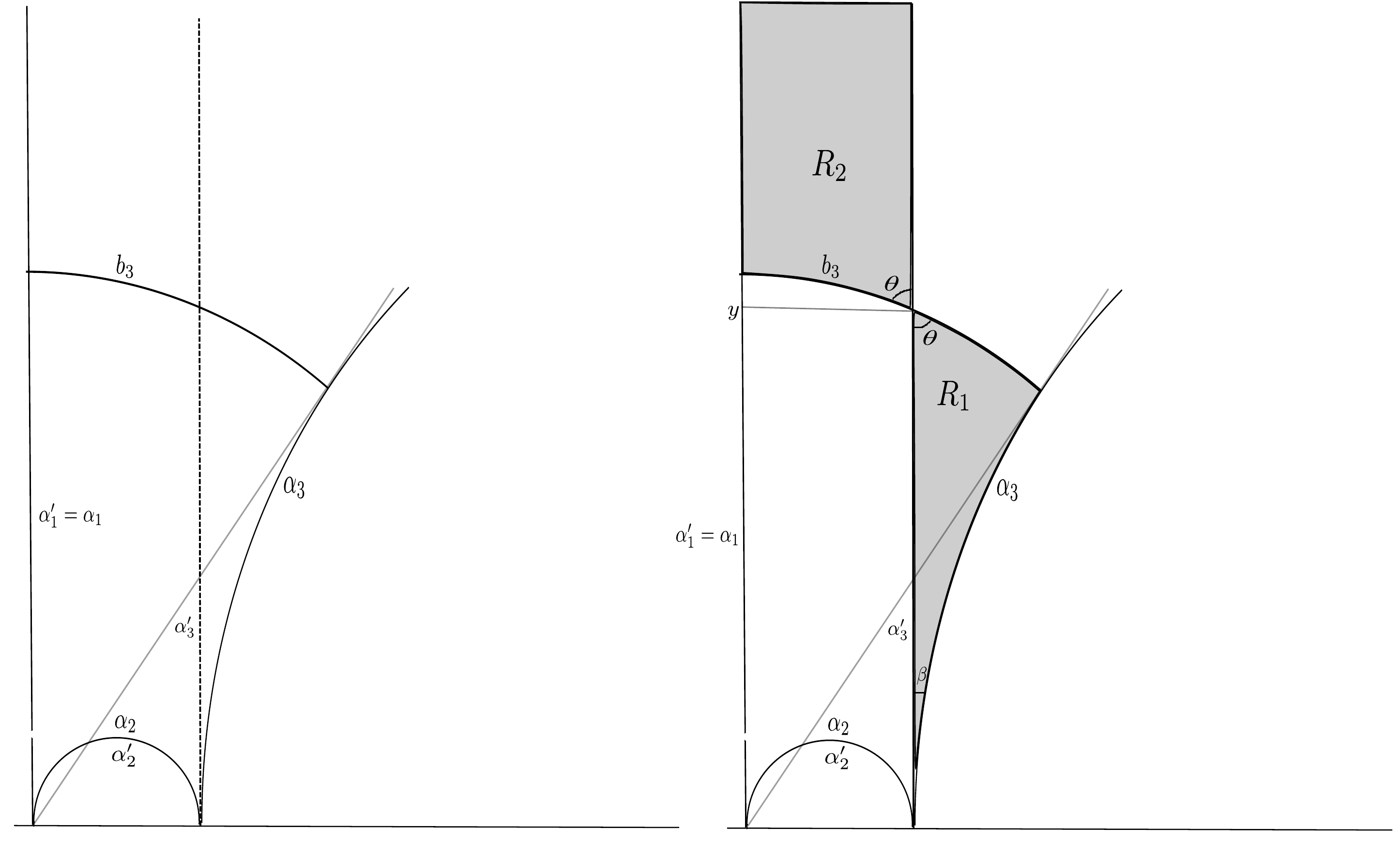}
\end{center}

\begin{align*}
area((H_1\cup H_2) - (H_1 \cap H_2)) &= area(R_1) + area(R_2)\\
area(R_1) &= \pi - \frac{\pi}{2} - \theta - \beta\\
area(R_2) &= \pi - \frac{\pi}{2} - \theta \\
area((H_1\cup H_2) - (H_1 \cap H_2)) &= \pi - 2\theta - \beta.
\end{align*}

As $\theta \to \frac{\pi}{2}$ and $\beta \to 0$, we have the
result.
\end{proof}

\begin{lemma}
If $d_{\pi}^{\mathbb{H}}(H_1,H_2) = \varepsilon$ then the pair of paints $P_1$ and
$P_2$, formed from $H_1$ and $H_2$, satisfy $d_{\rho}(P_1,P_2) = \varepsilon$.
\end{lemma}

\begin{proof}
Let $Z := \mathbb{H} \times \lbrace 1,2\rbrace$, $E$ the union of all edges of
$H_1$ and $H_2$, and $I:E\times\{1\}\to E\times \{2\}$ be $I(x,1) = (x,2)$.
Then $I$ is an isometry.  So, by Theorem \ref{inclmaxmet} the canonical
embedding of $P_j$ into $(Z, d_I^0)$ are isometries and $d_{\pi}^{Z}(P_1,P_2) = \varepsilon$.  Hence the result.
\end{proof}


Thus we have,

\begin{theorem}
The topology generated by $d_{\rho}$ is the same as that given by the Fenchel
Nielson co-ordinates.
\end{theorem}

\section{Laminations and Riemann surface laminations}

\subsection{Laminations}
We will recall some of the key definitions related to laminations.  The exposition here is based on \cite{candelfol},
\cite{lyubich} and the article on Riemann surface laminations in \cite{ghys} by Etienne Ghys.

Roughly speaking, a lamination $\mathcal{L}$ is a topological space which is decomposed into smooth immersed 
sub-manifolds(called ``leaves") nicely organised in a local product structure.
  
A formal definition goes as follows.  A $d$-dimensional product lamination is a topological space of the form 
$U^d \times T$, where $U^d$ is a domain in $\mathbb{R}^d$ called a \textit{local leaf} or \textit{plaque}, and 
$T$ a topological space called \textit{regular transversal}.

A morphism between two product laminations is a continuous map that maps local leaves to local leaves.  It is also 
called a laminar map. 

\begin{dfn}[Laminations]
A $d$-dimensional lamination $\mathcal{L}$ (or briefly $d$-lamination) is a topological space $X$ endowed with the 
following structure: For any point in $x \in X$ there is a neighbourhood $U$ containing $x$ and a homeomorphism 
$\phi: U \to U^d \times T$ onto a $d$-dimensional product lamination $U^d \times T$ such that the transition maps 
$\phi_2\circ \phi_1^{-1}|_{\phi_1(U_1 \cap U_2)}:\phi_1(U_1 \cap U_2) \to \phi_2(U_1 \cap U_2)$ between the product 
laminations are laminar.  These homeomorphisms are called local charts.  The corresponding neighbourhoods $U$ are 
called flow boxes. The sets $\phi^{-1}(U^d\times \{t\rbrace)$ are called local leaves or plaques of $\mathcal{L}$.  
The sets $\phi^{-1}(\{x\rbrace\times T)$ are called regular transversals.
\end{dfn}

\begin{dfn}[Global leaf]
Any lamination $\mathcal{L}$ is decomposed into disjoint union of global leaves in the following way:  two points $x$
and $y$ belong to the same global leaf if there is a sequence of local leaves $L_0,L_1,...,L_k$ such that 
$x \in L_0$, $y \in L_k$ and $L_i \cap L_{i+1} \neq \emptyset$ for $i = 0,1,...,k-1$.  The global leaf passing through
$x$ will be denoted by $L(x)$.
\end{dfn}

\begin{dfn}
A \emph{transversal} is a Borel subset which intersects each leaf in at-most countably many points.
\end{dfn}

\subsection{Riemann surface lamination} A Riemann surface lamination is a locally compact, separable,
 metrisable space $M$ with an open cover by flow boxes $\{U_i\rbrace_{i\in I}$ and homeomorphisms $\phi_i:U_i \to D_i \times T_i$,
with $D_i$ an open set in $\mathbb{C}$ and $T_i$ a metric space, such that the coordinate changes in $U_i \cap U_j$ are of the form 
$$\phi_j \circ \phi_i^{-1}(z,t) = (\lambda_{ji}(z,t),\tau_{ji}(t))$$
where the map $z\to \lambda_{ji}(z,t)$ is holomorphic for each $t$.

A map $f: M \to N$ of Riemann surface laminations is holomorphic if it is continuous and maps each leaf of $M$ holomorphically to a leaf of $N$.

\subsubsection{Invariant transverse measures} A transverse measure for $M$ is a measure on the $\sigma$-ring of transversals 
which restrict to a $\sigma$-finite measure on each transversal and such that each compact regular transversal has finite 
mass.  It is called invariant if it is invariant by the holonomy transformations acting on transversals.

\subsection{Hyperbolic Riemann surface lamination}  A Riemann surface lamination is called a hyperbolic Riemann surface lamination if all 
leaves are hyperbolic Riemann surfaces.  By Theorem 4.3 in \cite{candeluniform} and Proposition 5.7 in the article on Riemann surface 
laminations in \cite{ghys} we know that this matches with the general definition.

\section{Compact subsets of the space of Riemann surface laminations}

To study  Riemann surface laminations using the same techniques, there are some hurdles.  Firstly, a lamination 
structure is just a topological structure and not a metric structure, i.e., the transition maps in the transverse 
directions are just homeomorphisms.  So, If we want to have a natural metric structure, we need to make the 
definition stricter.  A natural class of maps between metric spaces are Lipschitz maps.  So we defined

\begin{dfn}
\label{LLipschitzRSL}
An \emph{$L$-Lipschitz Riemann surface lamination} is a Riemann surface lamination where the coordinate changes 
$\tau_{ji}(t)$ are $L$-Lipschitz for all $i,j$.
\end{dfn}

Now we can use the leaf-wise metric and the transverse metric to define a metric on a
Riemann surface lamination.

\begin{dfn}[Metric on L-Lipschitz Riemann Surface Laminations]\label{metriconLLipschitzRSL} 
To obtain a metric, we use a Kobayashi type
 construction, namely, we consider the maximal metric on $M$ such that, the maps 
 $\phi_i^{-1}: D_i \times T_i \to U_i$ are distance decreasing functions.  
 
\end{dfn}

\subsection{Measure on a Riemann surface lamination}\label{SEC: measonRSL} The second hurdle was that a Riemann 
surface lamination does not come equipped with
a measure.  But, given an invariant transverse measure $\nu_i$ on $T_i$, on each $D_i \times T_i$ there is a natural
measure given by 
$$\int_{T_i} \left(\int _{D_i \times \lbrace t\rbrace} \sigma_i \right) d\nu_i(t).$$
where, $\sigma_i$ is the hyperbolic area measure on $D_i$.

\begin{lemma}
If $E \subset U_i \cap U_j$ then,
$$\int_{T_i} \left(\int _{D_i \times \lbrace t\rbrace} \chi_{\phi_i(E)} \cdot \sigma_i \right) d\nu_i(t) = \int_{T_j} \left(\int _{D_j \times \lbrace t\rbrace} \chi_{\phi_j(E)} \cdot \sigma_j \right) d\nu_j(t).$$ 
\label{measindepchart} 
\end{lemma} 

\begin{proof}
Observe that, $\phi_j \circ \phi_i^{-1}|_{\phi_i(U_i \cap U_j)}$ and hence $\phi_j \circ \phi_i^{-1}|_{\phi_i(E)}$ are 
biholomorphisms.  Thus they are isometries and area-preserving maps, i.e., 
$(\phi_j \circ \phi_i^{-1})_*(\sigma_i) = \sigma_j$.  Also, the transverse measure is holonomy invariant implies that
$\nu_j = (\gamma_{ij})_*(\nu_i)$ where $\gamma_{ij}$ is the holonomy co-cycle corresponding to the charts $U_i$ and 
$U_j$.  Thus, standard chain rule arguments give us the result. 
\end{proof}  

\begin{dfn}
\label{measonRSL}
If the lamination structure is given by charts $\phi_i: U_i \to D_i \times T_i$ then, for $E \subset M$ define $\mu$ as 
$$\mu(E) = \sum_{i} \int_{T_i} \left(\int _{D_i \times \lbrace t\rbrace} \chi_{\phi_i(E \setminus (\cup_{j=1}^{i-1} U_j)} \cdot \sigma_i \right) d\nu_i(t) $$
\end{dfn}

\noindent Lemma \ref{measindepchart} tells us that this definition is independent of the ordering of the charts.

Now we want to prove an analogue of the Bers' compactness theorem for Riemann surface laminations.  To this end we 
will define a notion of injectivity radius.  Injectivity radius, in general, tells us to what extent is the given 
space canonical.  So for Riemann surface laminations it should be a measure of the extent to which it is a product 
lamination.  Hence the following definition. 

\begin{dfn}[Injectivity radius for L-Lipschitz hyperbolic Riemann surface laminations]\label{injradRSL}
Given an L-Lipschitz
Riemann surface lamination $(M, \mathcal{L})$, we say the injectivity radius at a point $x$ is greater than $r$ if
there exists an injective map $\varphi: B(0,r)(\subset \mathbb{D}) \times T \to M$ such that:
\begin{enumerate}
\item The map $\varphi^{-1}|_{image(\varphi)}$ is a compatible chart. 
\item The point $x = \varphi(0,t)$ for some $t\in T$.
\item The map $\varphi$ is an isometry under the metric defined earlier.
\item The ball $B(x,r)\subset M$ is contained in the image of $\varphi$.
\end{enumerate} 
\end{dfn}

\subsection{A topology for the space of Riemann surface laminations}
\label{topRSL}

Let $(X_n,\mathcal{L}_n)$ be a sequence of $L$-Lipschitz Riemann surface laminations with invariant transverse measure
$\mu_n$ and $(X,\mathcal{L})$ an $L$-Lipschitz Riemann surface laminations  with invariant transverse measure $\mu$.
Let $d_n$ be the induced metric on $(X_n,\mathcal{L}_n)$ and $d$ be the induced metric on $(X,\mathcal{L})$.  We say
the sequence $(X_n,\mathcal{L}_n,\mu_n)$ converge to $(X,\mathcal{L},\mu)$ if 
\begin{itemize}
\item The sequence $d_{\rho}((X_n,d_n,\mu_n),(X,d,\mu))$ converges to zero.
\item There exist a system of compatible charts on $(X,\mathcal{L})$ such that every local parametrisation
     $\psi: B(0,r) \times T \to X$ is the limit of some local parametrisations $\psi_n: B(0,r) \times T_n \to X_n$, 
     with $T_n \to T$, described in Proposition \ref{limmapbnlimsp}. 
\end{itemize}

\subsection{Analogue of Bers' compactness theorem}

Let $\mathcal{X}(L,r,A)$ be the collection of all L-Lipschitz hyperbolic Riemann surface
laminations equipped with an invariant transverse measure such that,
\begin{enumerate}
\item Injectivity radius is greater than or equal to $r$.
\item The measure $\mu(M)\leq A$  for all $M \in \mathcal{X}(L,r,A)$. \label{prop2X(L,r,A)}
\item Given $\varepsilon > 0$ there exists $B(\varepsilon) > 0$ such that, for all $m\in M \in \mathcal{X}(L,r,A)$ 
      and for all regular transversals $T$ passing through $m$, the transversal measure of 
      $B(m,\varepsilon) \cap T \geq B(\varepsilon)$.\label{prop3X(L,r,A)} 
\end{enumerate} 

\begin{lemma}
Given a distance measure space $(X,d,\mu)$ with
$\mu(X)\leq A$ and the property, given $\varepsilon > 0$ there exists a $B(\varepsilon) > 0$ such that 
$\mu(B(x,\varepsilon)) \geq B(\varepsilon)$ for all $x\in X$, there exists a set $S_{\varepsilon}$ such that 
$\mu(((S_{\varepsilon})^{\varepsilon})^C) \leq \varepsilon$ and the cardinality of $S_{\varepsilon}$ is less 
than or equal to $\frac{A}{B(\frac{\varepsilon}{2})}$.
\label{netandball}
\end{lemma}

\begin{proof}
Choose an arbitrary point $x_1$.  If $B(x_1,\varepsilon) \nsupset X$, then choose $x_2$ to be an arbitrary point in 
$X \setminus B(x_1,\varepsilon)$.  Given that we have choose $x_1,...,x_k$, if 
$X \nsubset \cup_{i=1}^k B(x_i,\varepsilon)$, choose $x_{k+1}$ to be an arbitrary point in 
$X \setminus \cup_{i=1}^k B(x_i,\varepsilon)$.  Let, $N$ be the smallest integer strictly greater than 
$\frac{A}{B(\frac{\varepsilon}{2})}$.  We will prove that $X \subset \cup_{i=1}^m B(x_i,\varepsilon)$ for some 
$m < N$ using proof by contradiction. 
 
Assume that $X \nsubset \cup_{i=1}^m B(x_i,\varepsilon)$ for any $m < N$.  Note that, 
$B(x_i,\frac{\varepsilon}{2}) \cap B(x_j,\frac{\varepsilon}{2}) = \emptyset$.  Therefore, 
$\mu(X) \geq \sum_{i=1}^N \mu(B(x_i,\frac{\varepsilon}{2})) \geq N \times B(\frac{\varepsilon}{2}) > A$, which 
contradicts with the assumption that $\mu(X) \leq A$.  Thus, $X \subset \cup_{i=1}^m B(x_i,\varepsilon)$ for some 
$m < N$, say $m_0$. 
 
Thus, $S_{\varepsilon} = \lbrace x_1,...,x_{m_0}\rbrace$ satisfies the conditions we wanted.   
\end{proof}

\begin{lemma}
Properties \ref{prop2X(L,r,A)} and \ref{prop3X(L,r,A)} together implies the property that given $\varepsilon >0$ 
there exists $N(\varepsilon)$ such that, for all $M \in \mathcal{X}(L,r,A)$ there exists a set $S_{M,\varepsilon}$ 
such that $\vert S_{M,\varepsilon} \vert \leq N(\varepsilon)$ and 
$\mu((M_{S,\varepsilon})^{\varepsilon})^C)\leq \varepsilon$.
\label{netandball2}
\end{lemma}

\begin{proof}
Note that given a chart $\varphi_x: B(0,r)(\subset \mathbb{D}) \times T \to M$ with $x = \varphi_x(0,t)$,  
$B(x,\varepsilon) \supset \varphi_x(B_{\mathbb{H}}(0,\frac{\varepsilon}{2})\times B_T(t,\frac{\varepsilon}{2}))$.  
Property \ref{prop3X(L,r,A)} implies that $\mu(B_T(t,\frac{\varepsilon}{2})) \geq B(\frac{\varepsilon}{2})$, i.e., 
$\mu(B(x,\varepsilon))$ is greater than or equal to the hyperbolic area of $B_{\mathbb{H}}(0,\frac{\varepsilon}{2})$ 
times $B(\frac{\varepsilon}{2})$.  So, by Lemma \ref{netandball} we have the result.
\end{proof}

\begin{theorem}
The space $\mathcal{X}(L,r,A)$ is compact under the topology described in Section \ref{topRSL}.
\label{bersforlam}
\end{theorem}

\begin{proof}
Given a sequence of L-Lipschitz hyperbolic Riemann surface laminations $M^n$ in the space $\mathcal{X}(L,r,A)$, compactness 
theorem of metric measure spaces and Lemma \ref{netandball2} implies that $M^n$ converge to a metric measure space 
$M$.  We will show that $M \in \mathcal{X}(L,r,A)$ and $M^n$ converges to $M$ in the topology described in 
\ref{topRSL}.

Denote the measure on $M^n$ by $\mu^n$ and the measure on $M$ by $\mu$.  Note that $d_{\rho}(M^n, M) \to 0$ implies 
there exists a sequence of metric spaces $(Z^n,d^n)$ and $L(n)$-isometric $\varepsilon(n)$-embeddings 
$f^n:M^n \to Z^n$ and $g^n: M \to Z^n$ such that $d_{\pi}((f^n)_*(\mu^n), (g^n)_*(\mu)) + \frac{1}{L(n)} + \varepsilon(n)$
converges to zero as $n$ tends to infinity.

\begin{lemma}
The set of all points $x\in M$ such that there exists no sequence $x^n \in M^n$ with $d_{Z^n}(f^n(x^n),g^n(x))$ 
converging to zero as $n$ tends to infinity is a set of measure zero.
\label{cvgsalmost}
\end{lemma}   

\begin{proof}
Call this set $E$ and assume the contrary, i.e., $\mu(E) > 0$.  There exists no sequence $x^n \in M^n$ with 
$d_{Z^n}(f^n(x^n),g^n(x))$ converging to zero as $n$ tends to infinity, implies, there exists an $\varepsilon > 0$ 
such that $d_{Z^n}(f^n(M^n),g^n(E)) >\varepsilon$ for all $n$.  This along with $\mu(E) > 0$ contradicts with the 
fact that $d_{\pi}(f^n(\mu^n), g^n(\mu))$ converges to zero as $n$ tends to infinity.  
\end{proof}

So, by throwing this set away, without loss of generality we can assume that the set is empty.  Thus, given a point 
$x\in M$ there exists points $x^n \in M^n$ such that $d_{Z^n}(f^n(x^n),g^n(x))$ converges to zero as $n$ tends to 
infinity.  As injectivity radius of $M^n$ is greater than $r$, for each point in $M^n$ we can find an injective map 
$\varphi_x : B(0,r)(\subset \mathbb{D}) \times T(x) \to M$ such that:
\begin{enumerate}
\item The map $\varphi^{-1}|_{image(\varphi)}$ is a compatible chart. 
\item The point $x = \varphi(0,t)$ for some $t\in T$.
\item The map $\varphi$ is an isometry.
\item The ball $B(x,r)\subset M$ is contained in the image of $\varphi$.
\end{enumerate}
such that the coordinate changes in $U_i \cap U_j$ are of the form 
$$\varphi_x \circ \varphi_y^{-1}(z,t) = (\lambda_{xy}(z,t),\tau_{xy}(t))$$
where the map $z\to \lambda_{xy}(z,t)$ is holomorphic for each $t$ and the map $\tau_{xy}(t)$ is L-Lipschitz.  
Let $\varphi^n: B(0,r)(\subset \mathbb{D}) \times T^n \to M^n$ be such a chart about $x^n$.  We shall use these maps
to construct a chart $\varphi$ around $x$.  First note that,

\begin{lemma}
The sequence $T^n$ has a subsequence which converges, say to $T$, as $n$ tends to infinity.
\end{lemma}

\begin{proof}
Note that, 
$$\mu^n(T^n) \leq \frac{\mu^n(M^n)}{Area(B(0,r) \subset \mathbb{D})} \leq \frac{A}{Area(B(0,r) \subset \mathbb{D})}.$$
So, property (3) of elements of $\mathcal{X}(L,r,A)$ and Lemma \ref{netandball} implies that $\lbrace T^n \ | \forall n \rbrace$
is a compact subset of the space of metric measure spaces.   
\end{proof} 

So, without loss of generality, we can assume $T_n$ converges.  Now use Proposition \ref{limmapbnlimsp}, to construct
a measure preserving isometry $\varphi: B(0,r) \times T \to M$.  

\subsubsection*{Transition maps are well-behaved:} Let $x,y$ be two points in $M$ and 
$\varphi: B(0,r) \times S \to M,\psi: B(0,r) \times T \to M$ be local parametrisations created as explained before 
such that
$im(\varphi) \cap im(\psi) \neq \emptyset$.   We will prove that $\psi^{-1} \circ \varphi : \varphi^{-1}(\psi(B(0,r) \times T)) \to \psi^{-1}(\varphi(B(0,r)\times S))$
satisfies all conditions to make $M$ an $L$-Lipschitz Riemann surface lamination.

Let $x^n , y^n\in M^n$ be such that $d_{Z^n}(f^n(x^n),g^n(x))$, $d_{Z^n}(f^n(y^n),g^n(y))$ converges to zero as 
$n$ tends to infinity.  Let $\varphi^n: B(0,r) \times S(x^n) \to M^n$ and
$\psi^n: B(0,r) \times T(y^n) \to M^n$ be charts around $x^n$ and $y^n$ respectively.  As 
$im(\varphi) \cap im(\psi) \neq \emptyset$, $im(\varphi^n) \cap im(\psi^n) \neq \emptyset$ for large
enough $n$.  We have seen before that $d_{\rho}(S(x^n),S), d_{\rho}(T(y^n),T)$ can be assumed to be converging
to zero.  Construct $h^n: S \setminus \widehat{S} \to S(x^n)$ and $j^n: T(y^n)\setminus \widehat{T(y^n)} \to T$ as
in Theorem \ref{quasiisomgivdistsmall}.  Then,
\begin{align*}
\psi^{-1} \circ \varphi(z,t) &= \lim_{n\to \infty} (id \times j^n) \circ ((\psi^n)^{-1} \circ \varphi^n) \circ (id \times h^n)(z,t)\\
&= \lim_{n\to \infty} (id \times j^n) \circ ((\psi^n)^{-1} \circ \varphi^n)(z,h^n(t))\\
&= \lim_{n\to \infty} (id \times j^n)(\lambda^n(z,h^n(t)),\tau^n(h^n(t)))\\
&= \lim_{n\to \infty}(\lambda^n(z,h^n(t)),j^n(\tau^n(h^n(t))))\\
&= (\lim_{n\to \infty}\lambda^n(z,h^n(t)), \lim_{n\to \infty} j^n(\tau^n(h^n(t)))).
\end{align*}
It can be proved as in the proof of Proposition \ref{limmapbnlimsp} that given 
$\lambda^n:B(0,r) \times S(x^n) \to B(0,r)$ and $\tau^n: S(x^n) \to T(x^n)$ there exists 
$\lambda: B(0,r) \times S \to B(0,r)$ and $\tau: S \to T$ such that 
$\lim_{n\to \infty}\lambda^n(z,h^n(t)) = \lambda(z,t)$ and $\lim_{n\to \infty} j^n(\tau^n(h^n(t)))) = \tau(t)$.  To 
do this just note that to prove that the function is well defined we only needed that Cauchy sequences goes to Cauchy
sequences.  As Lipschitz maps satisfy this we are done.  We have $\lambda$ for a fixed $t$ is holomorphic by Montel's
theorem.  

Define $\delta_3^n = d_{\rho}(S(x^n),S)$ and $\delta_4^n = d_{\rho}(T(y^n),T)$.  Given $t,s\in T$, choose $n$ so 
large that $d(s,t) \leq \min \left\lbrace \left( \frac{1}{\delta_3^n} - 2\delta_3^n \right), \left( \frac{1}{\delta_3^n} - 2\delta_3^n - 2\delta_4^n \right) \right\rbrace$.  Then,
\begin{align*}
d(j^n(\tau^n(h^n(t))), j^n(\tau^n(h^n(s)))) &\leq  d(\tau^n(h^n(t)), \tau^n(h^n(s))) + 2\delta_4^n (\text{Theorem \ref{quasiisomgivdistsmall}})\\
&\leq L\times d(h^n(t), h^n(s)) + 2\delta_4^n (\tau^n \text{is $L$-Lipshcitz})\\
&\leq L \times d(t,s) + 2\delta_3^n + 2\delta_4^n (\text{Theorem \ref{quasiisomgivdistsmall}}).
\end{align*}  
As $\delta_3^n$ and $\delta_4^n$ tends to zero as $n$ tends to infinity, this shows that $\tau$ is $L$- Lipschitz. 

\end{proof}

\begin{theorem}\label{main2}
Let $\mathcal{X}(L,A)$ be the collection of all L-Lipschitz hyperbolic Riemann surface
laminations equipped with an invariant transverse measure such that
\begin{enumerate}
\item The measure $\mu(M)\leq A$  for all $M \in \mathcal{X}(L,A)$.
\item Given $\varepsilon$ there exists $B(\varepsilon) > 0$ such that, for all $m\in M \in \mathcal{X}(L,A)$ and 
      for all transversals $T$ passing through $m$, the transversal measure of 
      $B(m,\varepsilon) \cap T \geq B(\varepsilon)$. 
\item If $M(r)$ denotes the subset of $M$ where injectivity radius is less than or equal to $r$ then $\mu(M(r))$ goes
      uniformly to zero for all $M\in \mathcal{X}(L,A) $ as $r$ goes to zero. \label{3main2}
\end{enumerate} 
Then, $\mathcal{X}(L,A)$ is compact under the topology described in Section \ref{topRSL}.   
\end{theorem}

\begin{rmk}
Condition \ref{3main2} is required for the limit to have the structure of a lamination.  This was redundant in the 
case of Riemann surfaces because of Margulis lemma.  In Section \ref{nomargulis} we have showed that it is not 
redundant for Riemann surface laminations.
\end{rmk}

\begin{proof}
Given a sequence of L-Lipschitz hyperbolic Riemann surface laminations $M^n$ in the space $\mathcal{X}(L,A)$, consider the 
sequence $M^n \setminus M^n(r) $.  Now this sequence has a convergent subsequence by previous theorem.  Without loss 
of generality we can assume that $M^n \setminus M^n(r) $ converges and the limit is $M_r$.  If $r_1\geq r_2$ then we 
have an inclusion from $M_{r_1}\hookrightarrow M_{r_2}$.  Consider the direct limit of the sequence $M_{r}$ as $r$ 
tends to zero and call it $M$.  Observe that $M$ is a lamination, as every point in $M$ belongs to $M_r$ for some $r$ so, use 
the chart around that point in $M_r$.  Furthermore,
\begin{align*}
d_{\rho}(M^n, M) \leq d_{\rho}(M^n, M^n \setminus M^n(r)) + d_{\rho}(M^n \setminus M^n(r), M_r) + d_{\rho}(M_r,M)
\end{align*} 
The first term goes to zero as $r$ tends to zero, the second term goes to zero as $n$ tends to infinity and the third
term goes to zero as $r$ tends to zero.  Thus $M^n$ converges to $M$ as $n$ tends to infinity.
\end{proof}

\noindent We will now show that $\mathcal{X}(L,A)$ is far from being empty.  Below are some examples.

\subsection{Finite covers}\label{Finite covers}
Let $S$ be a hyperbolic Riemann surface and let $p:M \to S$ be an N-sheeted covering.  By definition of covering 
given any point, $p \in S$, there exists an evenly covered neighbourhood, $U$.  Thus, we have a map 
$\phi: U \times \lbrace 1,2,...N\rbrace \to p^{-1}(U)$ such that $\phi|_{U \times \lbrace i\rbrace}$ is a biholomorphism.  We will define a
metric and measure on each transversal, i.e., $p^{-1}(x_0)$ for all $x_0 \in S$, so that $M \in \mathcal{X}(L,r,A)$ for
some $L,r,A$.  Define the measure on $T$, $\nu$, as follows: 
$$\nu(\lbrace i\rbrace) = \frac{1}{N}$$
Define the metric on $p^{-1}(x_0)$ as follows:
\begin{align*}
d(m_1,m_2) = \frac{1}{N}
\end{align*} 
\begin{rmk}
Note that this Riemann surface lamination is $1$-Lipschitz.
\end{rmk}

\begin{lemma}
Injectivity radius of the Riemann surface lamination $M$ at a point $m$ is greater than or equal to injectivity radius
of the Riemann surface $S$ at the point $p(m)$.
\label{coverinjectivity}
\end{lemma}
 
\begin{proof}
If  $r = \text{injrad}_{p(m)}(S)$, then $B(p(m), r)$ is evenly covered.  Thus, we have the result.
\end{proof}

\begin{lemma}
If we denote the area measure on $S$ by $\mu$ and the measure induced on $M$ by the lamination structure as 
$\overline{\mu}$, then $\overline{\mu}(p^{-1}(U)) = \mu(U)$.  Hence $\overline{\mu}(M) = \mu(S)$
\label{covervolbound}
\end{lemma}

\begin{lemma}
$\nu(B(t,\varepsilon)) \geq \min\lbrace 1,\varepsilon\rbrace$
\end{lemma}

\begin{proof}
If $\varepsilon \lneq \frac{1}{N}$ then $\nu(B(t,\varepsilon)) = \frac{1}{N} \gneq \varepsilon$.  If 
$\varepsilon \geq \frac{1}{N}$ then, $\nu(B(t,\varepsilon)) = 1$.  Thus, 
$\nu(B(t,\varepsilon)) \geq \min(1,\varepsilon)$.
\end{proof}

\noindent Also, as the measure of every point is same under the measure we defined, it is invariant under holonomy 
transformations.  Thus we have,
\begin{theorem}
The space of finite covers of a Riemann surface $S$ is a subset of the space $\mathcal{X}(1,\text{injrad}(S),Area(S))$.
\end{theorem}

\begin{rmk}
Note that if the cover is connected, there is only one leaf for this foliation.  Thus, given a finite cover $X$ of 
$S$, $X$ with this lamination structure is isomorphic to $X$ as a surface (considered as trivial lamination).
\end{rmk}

\begin{rmk}
Even though there is only one leaf, $X$ with this lamination structure is different from the surface $X$ 
(considered as trivial lamination) as $L$-Lipschitz Riemann surface laminations.
\end{rmk}

\begin{rmk}
Furthermore, we get interesting laminations as limits of these laminations.  This will be the next example.
\end{rmk}

\subsection{Solenoid}\label{Solenoid}
Given a sequence $X_0,X_1,X_2,...$ of surfaces, and regular coverings $f_{n,n-1}: X_n \rightarrow X_{n-1}$ for all 
$n\in \mathbb{N}$, the inverse limit of this sequence, $X_{\infty}$ will be called a Solenoidal space.  Recall the 
definition of an inverse limit:

\noindent \textbf{Inverse Limit:}
Let $X_0,X_1,X_2,...$ be countable sequence of topological spaces, and suppose that for each $n>0$ there exists a 
function $f_{n,n-1}: X_n \rightarrow X_{n-1}$, called bonding maps.  The sequence of spaces and mappings, 
$\lbrace X_n,f_{n,n-1}\rbrace$, is called an inverse limit sequence and may be represented by the diagram

\begin{center}
\xymatrix{
...\ar[r]^{f_{n+1,n}} &X_n \ar[r]^{f_{n,n-1}} &X_{n-1} \ar[r]^{f_{n-1,n-2}} &...\ar[r]^{f_{1,0}}&X_0}
\end{center}

Clearly if $n>m$ then there exists a map $f_{n,m}:X_n\rightarrow X_m$ given by the composition 
$f_{n,m}=f_{m+1,m}\circ f_{m+2,m+1}\circ...\circ f_{n-1,n-2}\circ f_{n,n-1}$.

Consider a sequence $(x_0,x_1,...,x_n,...)$ such that $x_n \in X_n$ and $f_{n+1}(x_{n+1})=x_n$ for all $n\geq 0$.  
Such a sequence can be identified with a point in the product space $\Pi_{n=0}^{\infty} X_n$.  The set of all such 
sequences is thus a subset of $\Pi_{n=0}^{\infty} X_n$ and hence has a topology as a subspace.  This topological 
space is the inverse limit space of the sequence $\lbrace X_n, f_{n,n-1}\rbrace$.  We will denote it by $X_{\infty}$ and let 
$f_n: X_{\infty} \to X_n$ be the canonical projection.

 Given a point $(x_0,x_1,...) \in X_{\infty}$ we have, as described earlier, local parametrisations 
 $\varphi_n: B(0,r) \times T_n \to X_n$ where $r \leq \text{injrad}(X_0)$ with 
 $x_n = \varphi(0,t) \in Image(\varphi_n)$.  Similarly, if $\pi_n$ is the canonical projection from 
 $B(0,r) \times T_n \to T_n$, the sequence 
 $\lbrace T_n,h_{n,n-1}: = \pi_{n-1}((\phi_{n-1})^{-1}\circ f_{n,n-1} \circ \phi_n (0,t)) \rbrace$ is an inverse limit sequence.
 Let $T_{\infty}$ be the inverse limit of the sequence $\lbrace T_n,h_{n,n-1}\rbrace$.
 
Denote the map, $(x,t) \mapsto (x,h_n(t))$, from  $B(0,r) \times T_{\infty}$ to $B(0,r) \times T_{n}$ by $g_n$.  
Corresponding to the maps $\varphi_n \circ g_n: B(0,r) \times T_{\infty} \to X_n$ there exists a unique map 
$\varphi_{\infty}: B(0,r) \times T_{\infty} \to X_{\infty}$ around the point $(x_0,x_1,...) \in X_{\infty}$.  
More precisely, $\varphi_{\infty}: (x,t)\mapsto (\varphi_0\circ g_0(x,t), \varphi_1 \circ g_1(x,t),...)$.  
We will show that this map is injective.  
\begin{align*}
\varphi_{\infty}(x,s) = \varphi_{\infty}(y,s) &\iff \varphi_n \circ g_n(x,t) = \varphi_n \circ g_n(y,s), \ \forall n. \\
&\iff \varphi_n(x,h_n(t)) = \varphi_n(y,h_n(s)), \ \forall n.\\
&\iff x=y \text{ and } h_n(s)=h_n(t), \ \forall n. \ \ (\varphi_n \text{ is injective})\\
&\iff x=y \text{ and } s=t
\end{align*}

By similar analysis we can show that transition maps behave well.  Thus, $X_{\infty}$ has a lamination structure.  

\begin{dfn}[Measure on $T_n$]
Define a measure $\nu_n$ on $T_n$ recursively as 
$$\nu_n(t) = \frac{\nu_{n-1}(h_{n,n-1}(t))}{\deg(h_{n,n-1})}, \ \forall \ t \in T_n. $$
\end{dfn}
As every point in $T_n$ has the same measure, this measure is clearly holonomy invariant.

\begin{dfn}[Measure on $T_{\infty}$]
Define $\nu_{\infty}$ on $T_{\infty}$ as
$$\nu_{\infty}(E) = \lim_{n \to \infty}\nu_n(h_n(E)).$$
\end{dfn}
\noindent The sequence $\nu_n(h_n(E))$ is a decreasing sequence bounded below by zero.  Hence, the limit exists.  
Moreover, as each $\nu_n$ is holonomy invariant, the limit $\nu_{\infty}$ is also holonomy invariant.  

\noindent \textbf{Metric on $T_n$:}
From the inverse limit sequence $\lbrace T_n,h_{n,n-1}\rbrace$ we can produce a rooted tree.  The vertices of the tree are the 
elements of $\sqcup_{n \in \mathbb{N}} T_n$.  There is an edge between two vertices if they belong to consecutive 
$T_n$'s and the bonding map takes the element in $T_n$ to the element in $T_{n-1}$.  Assign a length of 
$\log_2 (\deg(h_{n,n-1}))$ to an edge connecting an element in $T_n$ with an element in $T_{n-1}$.  This way we get a
rooted tree with the lone element in $T_0$ being the root.  See figure below.    

\begin{center}
\label{gromov_product}
\includegraphics[scale=.68]{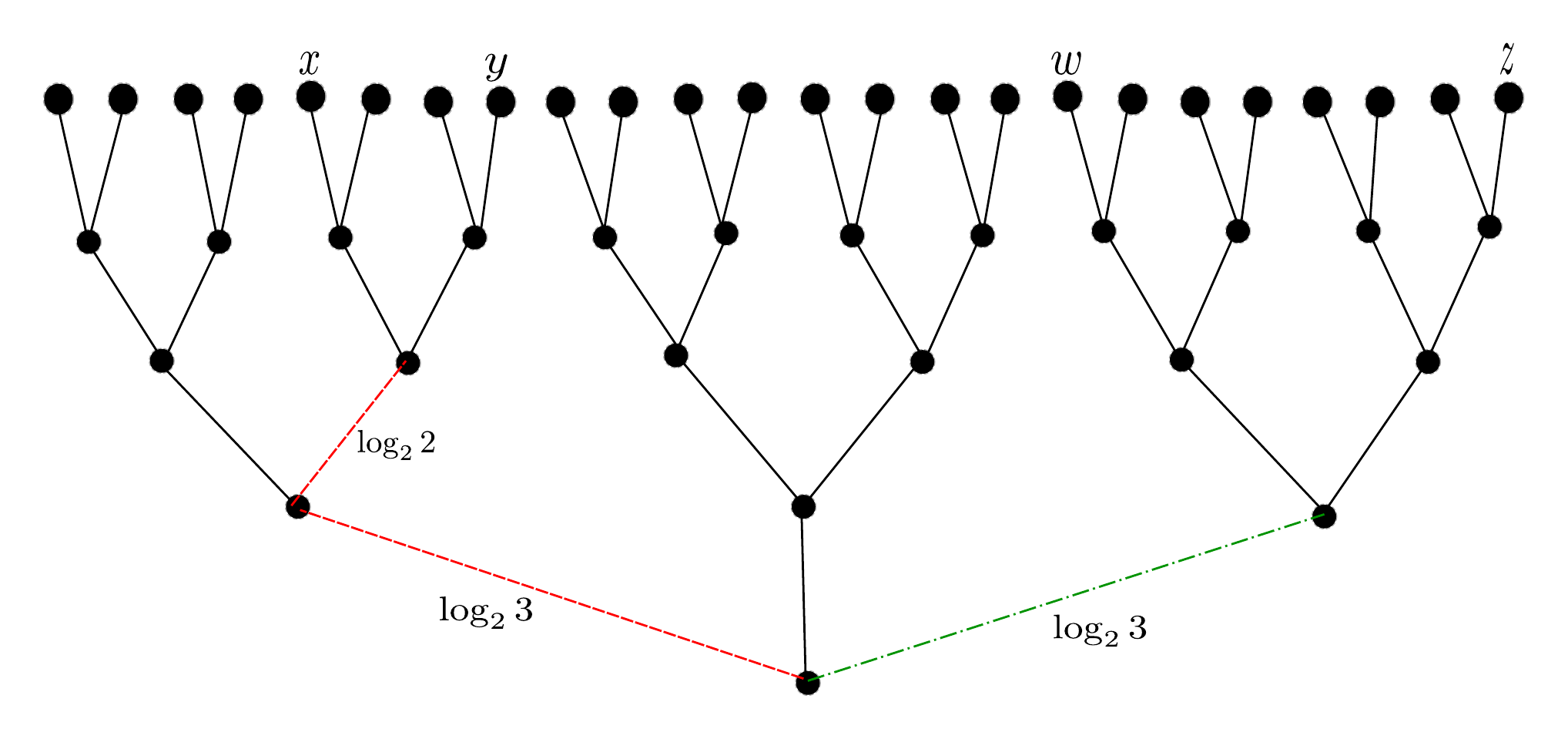}
\end{center}

\begin{dfn}[Gromov product]
Given two points, $x_1 \in T_{n_1}$ and $x_2 \in T_{n_2}$, define 
\[x\curlyvee y = \max\lbrace m \  | \ h_{n_1,m}(x_1) = h_{n_2,m}(x_2)\rbrace.\]
Note that $ 0 \leq m \leq \min(n_1,n_2)$.  The Gromov product, $(x_1|x_2)$, is defined as the length of the unique 
path from $h_{n_1,(x\curlyvee y)}(x_1) = h_{n_2,(x\curlyvee y)}(x_2)$ to the root.
\end{dfn}

\begin{dfn}[Metric on $T_n$]
The distance, $d_{T_n}$, is defined as
\begin{align*}
d_{T_n}(x,y) &= 2^{-(x|y)} \text{ if } x\neq y,\ \forall x,y \in T_n \\ 
d_{T_n}(x,x) &= 0
\end{align*} 
\end{dfn}

\noindent \textbf{Triangle inequality:} Given points $x,y,z \in T_n$,  if $h_{n,m}(x) = h_{n,m}(y)$, then 
$h_{n,m'}(x) = h_{n,m'}(y)$ for all $m' \leq m$.  Similarly for any other combination of two points.  Let $m$ be the
largest number such that $h_{n,m}(x) =  h_{n,m}(y)$ and let $k$ be the largest number such that 
$h_{n,k}(y) =  h_{n,k}(z)$.  Then 
$h_{n,\min(m,k)}(x) = h_{n,\min(m,k)}(y) = h_{n,\min(m,k)}(z)$, i.e., $x\curlyvee z \geq \min(x\curlyvee y, y\curlyvee z)$ which implies that $(x|z) \geq \min((x|y),(y|z))$.  Hence, $2^{-(x|z)} \leq \max(2^{-(x|y)},2^{-(y|z)})$, i.e., $d_{T_n}(x,z) \leq \max(d_{T_n}(x,y),d_{T_n}(y,z)) \leq d_{T_n}(x,y) + d_{T_n}(y,z)$.

\begin{dfn}[Metric on $T_{\infty}$]
Define a metric on $T_{\infty}$ as 
$$d_{\infty}(s,t) := \lim_{n \to \infty} d_{T_n}(h_n(s),h_n(t)).$$
\end{dfn}

\noindent Note that the bonding maps $h_{n,n-1}$ are distance decreasing.  Thus, 
$0 \leq d_{T_n}(h_n(s),h_n(t)) \leq d_{T_{n+1}}(h_{n+1}(s),h_{n+1}(t)) \leq...$.  On the other hand it is bounded 
above by $1$.  Hence the above limit exists.

\noindent If $h_n(s) \neq h_n(t)$ then, $d_{\infty}(s,t) = d_{T_n}(h_n(s),h_n(t))$. Otherwise 
$$d_{\infty}(s,t)\leq \frac{1}{\deg(h_{1,0})\times \deg(h_{2,1})\times ... \times \deg(h_{n,n-1})}.$$ 
and $d_{T_n}(h_n(s),h_n(t)) =  0$.  Thus,
$$d_{\infty}(s,t) -  d_{T_n}(h_n(s),h_n(t))\leq \frac{1}{\deg(h_{1,0})\times \deg(h_{2,1})\times ... \times \deg(h_{n,n-1})} =: \delta_n .$$ 
It is clear that $\delta_n$ tends to zero as $n$ tends to infinity.  Define 
$Z_n = T_n \sqcup T_{\infty}$ with metric on $Z_n$ defined as 
$d_{Z_n} = d_{h_n}^{\delta_n}$.  By Theorem \ref{inclmaxmet} the inclusions $T_n \hookrightarrow Z_n$ and
$T_{\infty} \hookrightarrow Z_n$ are isometries.  The Prokhorov distance between the push-forward measures is less than or equal to $\delta_n$ 
($\nu_n(E) \leq \nu_{\infty}(E^{\delta_n + \varepsilon})$ and $\nu_{\infty}(E) \leq \nu_n(E^{\delta_n + \varepsilon})$
for all $\varepsilon > 0$).  Thus, $d_{\rho}(T_n,T_{\infty})$ converges to zero as $n$ tends to infinity.

Use the metric on $T_{n}$ along with the local parametrisations for $X_{n}$ to construct a metric on $X_{n}$.  More 
precisely, choose the maximal metric such that all the local parametrisations are distance decreasing.  Similarly, 
use the metric on $T_{\infty}$ along with the local parametrisations for $X_{\infty}$ to construct a metric on 
$X_{\infty}$.  More precisely, choose the maximal metric such that all the local parametrisations are distance 
decreasing.  Then, it can be proved that $d_{X_{\infty}}(x,y) := \lim_{n \to \infty} d_{X_n}(f_n(x),f_n(y))$.  
Thus, as before define $Z_n = X_n \sqcup X_{\infty}$ and define the metric on $Z_n$ as
$d_{Z_n} = d_{h_n}^{\delta_n}$.  As before, the inclusions $X_n \hookrightarrow Z_n$ and $X_{\infty} \hookrightarrow Z_n$ are isometries by Theorem \ref{inclmaxmet} and the Prokhorov distance between the push-forward measures is less than or equal to $\delta_n$.  Thus, 
$d_{\rho}(X_n,X_{\infty})$ converges to zero as $n$ tends to infinity.

\begin{lemma}
Given $0 < \varepsilon < \frac{1}{deg(h_{1,0})}$, let $n(\varepsilon)$ be the smallest number, $n$,  such that 
$$\zeta_n:= \frac{1}{\deg(h_{1,0})\times \deg(h_{2,1})\times ... \times \deg(h_{n,n-1})} < \varepsilon.$$  Then,
$$\nu_n(B(t,\varepsilon)) \geq \frac{\varepsilon}{\deg(h_{n(\varepsilon),n(\varepsilon) - 1})}.$$
\end{lemma}

\begin{proof}
Note that $\zeta_{n(\varepsilon)-1} > \varepsilon$, i.e., $\deg(h_{n(\varepsilon),n(\varepsilon) - 1}).\zeta_{n(\varepsilon)} > \varepsilon$.
If $n >n(\varepsilon)$ then, $B(t,\varepsilon)$ contains all $s$ such that 
$h_{n,n(\varepsilon)}(s) = h_{n,n(\varepsilon)}(t)$.  Thus, by definition of $\nu_n$, 
$$\nu_n(B(t,\varepsilon)) =  \nu_{n(\varepsilon)}(h_{n,n(\varepsilon)}(t)) = \zeta_{n(\varepsilon)} > \frac{\varepsilon}{\deg(h_{n(\varepsilon),n(\varepsilon) - 1})}.$$
If $n \leq n(\varepsilon)$, $B(t,\varepsilon)$ contains all points in $T_n$.  So $\mu(B(t,\varepsilon)) = 1$.  Hence the result.
\end{proof}
\noindent Also, If $\varepsilon \geq \frac{1}{deg(h_{1,0})}$, then $B(t,\varepsilon) = T_n$.  As $n(\varepsilon)$ 
depends only on $\varepsilon$ and does not depend on $n$, given $\varepsilon$, there exists a uniform lower bound on
the the volume of $\varepsilon$-ball in $T_n$.    Thus this along with Lemma \ref{covervolbound} and 
Lemma \ref{coverinjectivity} gives us 

\begin{lemma}
$\lbrace X_n: n\in \mathbb{N}\rbrace \subset \mathcal{X}(1,\text{injrad}(X_0),Area(X_0))$
\end{lemma}

\begin{theorem}
$X_{\infty} \in \mathcal{X}(1,\text{injrad}(X_0),Area(X_0))$
\end{theorem}

\begin{proof}
As $\lbrace X_n: n\in \mathbb{N}\rbrace \subset \mathcal{X}(1,\text{injrad}(X_0),Area(X_0))$ and 
$\mathcal{X}(1,\text{injrad}(X_0),Area(X_0))$ is compact, $X_{\infty} \in \mathcal{X}(1,\text{injrad}(X_0),Area(X_0))$.
\end{proof} 

\begin{theorem}
The Riemann surface lamination $X_{\infty}$ has more than one leaf.
\end{theorem}

\begin{proof}
The intersection of a transversal with a leaf is discrete on the leaf with respect to the leaf wise metric.  The 
second countability of the leaf (separability is preserved under inverse limit) implies that the number of 
intersections of any transversal with a leaf is countable.  For a generic inverse sequence 
$\lbrace T_n,h_{n,n-1}\rbrace$, $T_{\infty}$ is the cantor set hence, uncountable.  Thus, such a lamination will have more than
one leaf.  
\end{proof}

\subsection{Suspension}\label{Suspension}
In this section we will closely follow section 3.1 of \cite{candelfol}. Let $S$ be a closed hyperbolic surface and
$p: \mathbb{H} \to S$ be the universal covering map.  We adopt the convention that the group $\Gamma$ of covering 
transformations acts from the right:
\begin{align*}
&\mathbb{H} \times \Gamma \to S\\
&(z,\gamma) \mapsto z.\gamma =  \gamma^{-1}(z)
\end{align*}
Fix a metric space $F$ together with a choice of base-point $z_0 \in p^{-1}(x_0)$, and let
$$h:\Gamma \to Isom(F)$$
be a group homomorphism.  The choice of base-point fixes an identification $\Gamma = \pi_1(S,x_0)$.  We define the
``diagonal" action of $\Gamma$ on $\mathbb{H} \times F$ to be the left action.
\begin{align*}
&\Gamma \times (\mathbb{H} \times F) \to \mathbb{H} \times F\\
&(\gamma,(z,y)) \mapsto (z.\gamma^{-1}, \gamma.y)
\end{align*} 
where $\gamma.y = h(\gamma)(y)$.  The quotient space 
$$M = \Gamma \setminus (\mathbb{H} \times F) = \mathbb{H} \times_{\Gamma} F$$
has a Riemann surface lamination structure.  The projection $\mathbb{H} \times F \to S$ defined by 
$$(z,y) \mapsto p(z)$$
respects the group action, hence passes to a well-defined submersion
$$\pi : M\to S.$$
Let $x\in S$ choose a neighbourhood $U$ of $x$ that is evenly covered by $p$.  That is,
$$p^{-1}(U) = \coprod_{i\in \mathcal{I}}U_i$$
and $p|U_i: U_i \to U$ is a biholomorphism, for all $i \in \mathcal{I}$.  The family $\lbrace U_i \times F\rbrace$ consists of
disjoint open subsets of $\mathbb{H}\times F$ that are permuted simply transitively by the diagonal action of 
$\Gamma$.  Thus, the equivalence relation defined by the action of $\Gamma$ makes no internal identification in any 
of the sets $U_i \times F$ and the quotient projections carries each one onto $\pi^{-1}(U)$.  The inverse $\varphi_i$
of this gives a commutative diagram
\[
\xymatrix{
\pi^{-1}(U) \ar[r]^{\varphi_i} \ar[d]^{\pi} &U_i \times F\ar[d]^{p} \ar[r]^{p \times id} &U \times F \ar[d]^{p_1}\\
U \ar[r]^{id} &U \ar[r]^{id} &U} 
\]
where $p_1$ is the projection onto the first factor.  This gives a Riemann surface lamination structure to $M$.  As 
$h$ maps to $Isom(F)$, we have a $1$-Lipschitz Riemann surface lamination.  As in the previous example the injectivity
radius of $M$, as a Riemann surface lamination, is greater than or equal to the injectivity radius of $S$.  Choose a 
finite measure $\mu$ on $F$, invariant under isometries of $F$ and satisfying the condition, given $\varepsilon > 0$ 
there exists $B(\varepsilon) > 0$ such that, for all $t\in T$, $\mu(B(t,\varepsilon)) \geq B(\varepsilon)$.  Then, 
$M \in \mathcal{X}(1,\text{injrad}(S),Area(S) \times \mu(F))$.

\subsection{Sequence of laminations with no convergent subsequence}\label{nomargulis}

Let $S$ be a genus two surface such that, the thin part is isometric to the hyperbolic annulus $T(\lambda)$.  For 
each $n\in \mathbb{N}$, we will define an $L$-Lipschitz Riemann surface lamination structure on this surface as follows.     

\subsubsection*{Definition of $\mathcal{L}_n$:} Given $0 < \varepsilon < C_m$, the Margulis constant, every point 
in the thick part has a neighbourhood that is isometric to the standard ball of radius $\varepsilon$ in $\mathbb{H}$.
Construct such charts around points in the thick part for $\varepsilon < \frac{C_m}{2}$.  Thus, if we define a 
compatible chart for the thin part as well, we get a lamination structure on $S$.  

The canonical $\mathbb{Z}/n\mathbb{Z}$ action on $S^1$ induces a group action on the annulus.  Let 
$p: T(\lambda) \to T(\lambda)/\left(\mathbb{Z}/n\mathbb{Z}\right)$ be the quotient map induced by this group action.
So as in Section \ref{Finite covers} we get a lamination structure on $T(\lambda)$.  More precisely, given any point
in $T(\lambda)/\left(\mathbb{Z}/n\mathbb{Z}\right)$ there exists an evenly covered neighbourhood $U$ and a map 
$U\times \lbrace 1,...,n\rbrace \to p^{-1}(U)$. 
We define a metric and measure on each transversal.  The metric on $T_n := \lbrace 1,2,...,n\rbrace$ is the discrete metric
\begin{align*}
d_n(i,j) = \frac{C_m}{n}
\end{align*}

Define the measure on $T_n$ as 
$$\nu_n(\lbrace i\rbrace) = 1$$

Use an isometry, $\psi$, from $T(\lambda)$ to the thin part of $S$ to get a lamination on the thin part.  It is easy
to see that this lamination structure is compatible to the one we defined on the thick part.  So we have a lamination
on $S$.  This lamination is $1$-Lipschitz.  Denote the induced distance on $S$ by $D_n$ and the induced measure by $\mu_n$.  

Let $E(C_m)$ be the thin part of $S$.  Note that for every point in $E(C_m)$, the injectivity radius in $(S,\mathcal{L}_n)$
converges to zero as $n$ tends to infinity.  On the other hand, if $\sigma_g$ is the area form on 
$T(\lambda)$, $\mu_n(E(C_m)) = \int_{T(\lambda)} \sigma_g$ for $E$ subset of the thin part, for all 
$n \in \mathbb{N}$.  Thus the measure of $\mu_n(E(C_m))$ is equal to the hyperbolic area of $T(\lambda)$.  So, the 
measure of the part of $(S,\mathcal{L}_n,\mu_n)$ with injectivity less than $r$, does not converge to zero uniformly 
as $r$ tends to zero.  Thus, there is no theorem analogous to the Margulis lemma for surface laminations.  We will 
also show that this sequence does not have any convergent subsequence.  

Let $E\left(\frac{C_m}{2}\right)$ be the subset of $S$ with injectivity radius less than or equal to $\frac{C_m}{2}$.  Fixing a 
normal to the shortest geodesic, the distance from the shortest geodesic gives a map, $f$, from $E\left(\frac{C_m}{2}\right)$ to
an interval $[-l,l]$.  Define a measure, $\nu$, on $[-l,l]$ as $\nu = \lim_{n\to \infty} f_*(\mu_n)$.  We will show
that $\left(E\left(\frac{C_m}{2}\right),D_n,\mu_n\right)$ converges to $([-l,l],\vert x-y \vert,\nu)$.

Let $Z_n = E\left(\frac{C_m}{2}\right) \cup [-l,l]$ and the distance on $Z_n$ be defined as
$d_{Z_n} = d_{f}^{\frac{C_m}{n}}$.  Notice that the inclusions $E\left(\frac{C_m}{2}\right) \hookrightarrow Z_n$
and $[-l,l] \hookrightarrow Z_n$ are isometries by Theorem \ref{inclmaxmet}.  Abusing
notation we will denote the push forwards of $\mu_n$ and $\nu$ to $Z_n$ by $\mu_n$ and 
$\nu$ itself.  Then,
$d_{\pi}^{Z_n}(\mu_n,\nu) \leq \frac{C_m}{n} + d_{\pi}^{[-l,l]}(f_*(\mu_n),\nu)$.
Thus, the sequence 
$d_{\rho}\left(\left(E\left(\frac{C_m}{2}\right),D_n,\mu_n\right),\left([-l,l],\vert x-y \vert,\nu\right)\right)$ converges to zero.  But, $[-l.l]$ cannot have 
any Riemann surface lamination structure as, the Hausdorff dimension of an interval is $1$, whereas, the Hausdorff 
dimension of a Riemann surface lamination is greater than or equal to $2$.

\bibliographystyle{alpha} \bibliography{paper}

\begin{thebibliography}{CGSY03}

\bibitem[BBI01]{metricgeometry}
Dmitri Burago, Yuri Burago, and Sergei Ivanov.
\newblock {\em A course in metric geometry}, volume~33 of {\em Graduate studies
  in mathematics}.
\newblock American Mathematical Society, Providence, RI, 2001.

\bibitem[Bus10]{buser}
Peter Buser.
\newblock {\em Geometry and spectra of compact {R}iemann surfaces}.
\newblock Modern Birkh\"auser Classics. Birkh\"auser Boston, Inc., Boston, MA,
  2010.
\newblock Reprint of the 1992 edition.

\bibitem[Can93]{candeluniform}
Alberto Candel.
\newblock Uniformization of surface laminations.
\newblock {\em Ann. Sci. \'Ecole Norm. Sup. (4)}, 26(4):489--516, 1993.

\bibitem[CC00]{candelfol}
Alberto Candel and Lawrence Conlon.
\newblock {\em Foliations. {I}}, volume~23 of {\em Graduate Studies in
  Mathematics}.
\newblock American Mathematical Society, Providence, RI, 2000.

\bibitem[CGSY03]{ghys}
Dominique Cerveau, {\'E}tienne Ghys, Nessim Sibony, and Jean-Christophe Yoccoz.
\newblock {\em Complex dynamics and geometry}, volume~10 of {\em SMF/AMS Texts
  and Monographs}.
\newblock American Mathematical Society, Providence, RI; Soci\'et\'e
  Math\'ematique de France, Paris, 2003.
\newblock With the collaboration of Marguerite Flexor, Papers from the Meeting
  ``State of the Art of the Research of the Soci{\'e}t{\'e} Math{\'e}matique de
  France'' held at the {\'E}cole Normale Sup{\'e}rieure de Lyon, Lyon, January
  1997, Translated from the French by Leslie Kay.

\bibitem[GK13]{gadgil}
Siddhartha Gadgil and Manjunath Krishnapur.
\newblock Lipschitz correspondence between metric measure spaces and random
  distance matrices.
\newblock {\em Int. Math. Res. Not. IMRN}, (24):5623--5644, 2013.

\bibitem[GPW09]{MetMeasSpref}
Andreas Greven, Peter Pfaffelhuber, and Anita Winter.
\newblock Convergence in distribution of random metric measure spaces
  ({$\Lambda$}-coalescent measure trees).
\newblock {\em Probab. Theory Related Fields}, 145(1-2):285--322, 2009.

\bibitem[Gro07]{gromov}
Misha Gromov.
\newblock {\em Metric structures for {R}iemannian and non-{R}iemannian spaces}.
\newblock Modern Birkh\"auser Classics. Birkh\"auser Boston, Inc., Boston, MA,
  english edition, 2007.
\newblock Based on the 1981 French original, With appendices by M. Katz, P.
  Pansu and S. Semmes, Translated from the French by Sean Michael Bates.

\bibitem[Lyu]{lyubich}
Mikhail Lyubich.
\newblock Laminations and holomorphic dynamics.
\newblock {\em Notes}.

\bibitem[Par05]{parthasarathy}
K.~R. Parthasarathy.
\newblock {\em Probability measures on metric spaces}.
\newblock AMS Chelsea Publishing, Providence, RI, 2005.
\newblock Reprint of the 1967 original.

\bibitem[vG03]{gaans}
Onno van Gaans.
\newblock Probability measures on metric spaces, notes of the seminar
  stochastic evolution equations.
\newblock {\em Delft University of Technology}, 2003.

\end{thebibliography}

\end{document}